\documentclass[12pt]{amsart}
\usepackage{amscd,amssymb,longtable, rotating, lscape, graphicx}
\usepackage[matrix,arrow,curve]{xy}
\usepackage{supertabular}
\usepackage{setspace}
\usepackage{tikz}
\usetikzlibrary{calc}
\usepackage{tikz-cd}
\usepackage{multirow}
\usepackage{graphicx}
\usepackage{subfigure}
\usepackage{lineno}
\usepackage{longtable}
\allowdisplaybreaks

\newtheorem{theorem}{Theorem}[section]
\newtheorem*{theoremE*}{Theorem E}
\newtheorem*{theoremD*}{Theorem D}
\newtheorem*{theoremC*}{Theorem C}
\newtheorem*{theoremB*}{Theorem B}
\newtheorem*{theoremA*}{Theorem A}
\newtheorem*{mainthm*}{Main Theorem}
\newtheorem{lemma}[theorem]{Lemma}

\newtheorem{definition}[theorem]{Definition}
\newtheorem{proposition}[theorem]{Proposition}
\newtheorem*{definition*}{Definition}

\theoremstyle{remark}

\makeatletter\@addtoreset{equation}{section}
\makeatother

\makeatletter\@addtoreset{equation}{section}
\makeatletter\@addtoreset{section}{part}

\makeatother

\usepackage{xcolor}

\usepackage[symbol]{footmisc}

\title{Conical K\"ahler-Einstein metrics on K-unstable del Pezzo surfaces}
\author{Dasol Jeong}
\author{Jihun Park}

\begin{document}
%\linenumbers

\begin{abstract}
    We establish the optimal upper bounds for cone angles of  K\"ahler-Einstein metrics with conical singularities along smooth anticanonical divisors on smooth K-unstable del Pezzo surfaces.
\end{abstract}

\maketitle

\section{Introduction}

Let $X$ be a Fano manifold and let $D$ be a smooth anticanonical divisor on $X$. A Kähler–Einstein metric $\omega$ on $X$ with  conical singularities along $D$ satisfies the following 
equation:
\begin{equation}\label{cKE-def}
    \mathrm{Ric}(\omega) = \lambda \omega + (1 - \lambda)[D],
\end{equation}
where the cone angle along $D$ is $2\pi\lambda$ for some $\lambda$ in $(0,1]$, and $[D]$ denotes the current of integration along $D$. Metrics satisfying \eqref{cKE-def} play a key role in the study of the Yau–Tian–Donaldson conjecture (see \cite{Don12, CDS14a, CDS14b, CDS14c}). In particular, if $\lambda=1$, meaning that the cone angle is $2\pi$, then the metric is a K\"ahler-Einstein metric on $X$.

If $X$ does not admit a smooth Kähler–Einstein metric, then the supremum
\[
    R(X,D) := \sup\left\{ \lambda > 0 \,\middle|\, \text{equation \eqref{cKE-def} admits a solution} \right\}
\]
is strictly less than $1$.

An interesting  comparison is with the equation
\[
    \mathrm{Ric}(\omega) = \lambda \omega + (1 - \lambda)\rho,
\]
where $\rho$ is a fixed smooth positive $(1,1)$-form representing $2\pi c_1(X)$. This formulation appears in the continuity method used to establish the existence of Kähler–Einstein metrics on manifolds with $c_1(X) \leq 0$ (see \cite{Aub76, Yau77, Yau78}). The corresponding greatest Ricci lower bound is defined by
\begin{equation}\label{continuity}
    R(X) := \sup\left\{ t \in \mathbb{R} \,\middle|\,  \ \text{there is a (1,1)-form } \omega \in c_1(X)\ \text{such that}\ \mathrm{Ric}(\omega) > t \omega \right\}.
\end{equation}
Note that $R(X) = 1$ if and only if $X$ admits a smooth Kähler–Einstein metric.

Due to the formal similarity between equations \eqref{cKE-def} and \eqref{continuity}, Donaldson conjectured in 2012 that $R(X,D) = R(X) $ (\cite[Conjecture~1]{Don12}). Song and Wang verified the conjecture in a variational framework, by further considering pluri-anticanonical divisors (\cite{SW16}). However, Székelyhidi provided counterexamples in the surface case, showing that the equality does not hold in general (\cite{Sze12}).

To explain Székelyhidi’s counterexamples, let $\phi_1 : S_1 \to \mathbb{P}^2$ denote the blowup at a point $x$ on $\mathbb{P}^2$ with exceptional divisor $E$.  
Similarly, let $\phi_2 : S_2 \to \mathbb{P}^2$ be the blowup at two distinct points $x_1, x_2$ in $\mathbb{P}^2$ with corresponding exceptional divisors $A^1$ and $A^2$.  
We will retain this notation throughout the remainder of the article.

It is known that neither $S_1$ nor $S_2$ admits a Kähler–Einstein metric (\cite{Tian90}). They are the only smooth del Pezzo surfaces that allow no Kähler–Einstein metric. The greatest Ricci lower bounds for these surfaces have been computed (see \cite{Sze11, Li11}) as
\begin{equation}\label{delta S_i}
    R(S_1) = \frac{6}{7} < 1, \quad R(S_2) = \frac{21}{25} < 1.
\end{equation}
Furthermore, the assertion below gives us counterexamples to Donaldson's conjecture.
\begin{theorem}[{\cite[Theorem 1]{Sze12}}]
    On $S_1$, for any smooth anticanonical divisor $C^1$ 
    \[R(S_1, C^1) \leq \frac{4}{5} < \frac{6}{7} = R(S_1).\]
    On $S_2$, if a smooth anticanonical divisor $C^2$ passes through the intersection point of two $(-1)$-curves, then
    \[
        R(S_2, C^2) \leq \frac{7}{9} < \frac{21}{25} = R(S_2).
    \]
\end{theorem}

Meanwhile, Cheltsov and Martinez-Garcia gave the following lower bounds for $R(S_i,C^i)$ for $i=1,2$.
\begin{theorem}[{\cite[Corollaries 1.11~and 1.12 ]{CM-G16}}]  On $S_1$, we have $R(S_1, C^1) \geq \frac{3}{10}$. Moreover, if $C^1$ is chosen to be general in the linear system $|-K_{S_1}|$, then $R(S_1, C^1) \geq \frac{3}{7}$.
On $S_2$, $R(S_2, C^2) \geq \frac{3}{7}$, and in fact $R(S_2, C^2) \geq \frac{1}{2}$ unless $C^2$ passes through the intersection point of two $(-1)$-curves.
  \end{theorem}

In the present article, we determine the explicit values $R(S_1,C^1)$ and $R(S_2,C^2)$ for arbitrary smooth anticanonical divisors $C^1$ on $S_1$ and $C^2$ on $S_2$,
using the techniques developed by Denisova (\cite{Den25}).
%derived from \cite{Den25}. 

\begin{mainthm*}\label{mainthm}
Let $C^1$ (resp. $C^2$) be a smooth anticanonical divisor on $S_1$ (resp. $S_2$).
Then
\[
\begin{aligned}
    R(S_1, C^1) &= 
    \begin{cases}
        \dfrac{3}{4} &  \text{if $C^1$ is tangent to the $0$-curve at the intersection point of $E$ and $C^1$};
        \\[1.1em]
        \dfrac{4}{5} & \text{otherwise;}
    \end{cases} \\[1em]
    R(S_2, C^2) &=
    \begin{cases}
        \dfrac{7}{9} & \text{if $C^2$ passes through the intersection of two $(-1)$-curves;}\\[1.1em]
        \dfrac{21}{25} & \text{otherwise.}
    \end{cases}
\end{aligned}
\]
\end{mainthm*}

\section{K\"ahler-Einstein metric, K-stability and $\delta$-invariant }

Let $(X,\Delta)$ be a log Fano pair, that is,
\begin{itemize}
    \item $X$ is a normal projective $\mathbb{Q}$-factorial variety;
    \item $\Delta$ is an effective $\mathbb{Q}$-divisor;
    \item $(X,\Delta)$ is a klt pair;
    \item $-(K_X+\Delta)$ is ample.
\end{itemize}

The existence of a K\"ahler-Einstein metric on the log Fano pair $(X,\Delta)$ is known to be equivalent to the K-polystability of $(X,\Delta)$ in a fully general setting (see \cite{B16, BBJ21,  CDS14a, CDS14b, CDS14c,  Li22, LTW22, LXZ22, T15}). However, for our purposes, we only need the following special case:

\begin{theorem}[{\cite[Corollary 1.2]{LTW22}},\cite{LXZ22}]\label{Kst and delta}
Let $(X,\Delta)$ be a log Fano pair with discrete automorphism group. Then 
$(X,\Delta)$ admits a K\"ahler-Einstein metric if and only if it is K-stable.
\end{theorem}

In particular, let $X$ be a Fano manifold and $D$ be a smooth anticanonical divisor on $X$. Then
the K\"ahler-Einstein metric on the log Fano pair $(X,(1-\lambda)D)$ corresponds to a K\"ahler-Einstein metric with conical singularities of angle $2\pi\lambda$ along $D$.

K-stability of a log Fano pair can be effectively checked using the $\delta$-invariant, defined as follows.

Let $f\colon \hat{X} \to X$ be a birational morphism. A prime divisor $G$ on $\hat{X}$ is called a divisor over $X$ and is denoted by $G/X$. The image $f(G)\subset X$ is referred to as the center of $G$, denoted by $c_X(G)$. We also denote the log discrepancy of $(X,\Delta)$ along $G$ by $A_{X,\Delta}(G)$.

We define  a key invariant associated to $G$:
\begin{equation}\label{eq:S}
     S_{X,\Delta}(G) = \frac{1}{(-K_X - \Delta)^n} \int_0^{\tau_{X,\Delta}(G)} \mathrm{vol}\left(f^*(-K_X - \Delta) - tG\right) \, dt,
\end{equation}
where $\tau_{X,\Delta}(G)$ is the pseudoeffective threshold of $G$ with respect to $-(K_X + \Delta)$, defined by
\[
\tau_{X,\Delta}(G) := \sup \left\{ t \in \mathbb{Q}_{>0} \mid f^*(-K_X - \Delta) - tG \text{ is pseudoeffective} \right\}.
\]

\begin{definition}
The $\delta$-invariant of the pair $(X,\Delta)$ is given by
\[
\delta(X,\Delta) := \inf_{G/X} \frac{A_{X,\Delta}(G)}{S_{X,\Delta}(G)}.
\]
The local $\delta$-invariant at a point $p$ on $X$ is defined as
\[
\delta_p(X,\Delta) := \inf_{\substack{G/X\\ c_X(G)\ni p}} \frac{A_{X,\Delta}(G)}{S_{X,\Delta}(G)}.
\]
\end{definition}

It follows immediately from the definition that
\begin{equation} \label{local delta}
\delta(X,\Delta) = \inf_{p \in X} \delta_p(X,\Delta).
\end{equation}

In the case when $X$ is smooth, then the $\delta$-invariant relates to the greatest  Ricci lower bound by 
$
R(X) = \min\{\delta(X), 1\}$ (see \cite[Theorem 5.7]{CRYZ}).

As mentioned earlier,  the $\delta$-invariant serves as a criterion for K-stability:

\begin{theorem}[{\cite{Fuj19}, \cite{Li17}, \cite{BX19}}]
A log Fano pair $(X, \Delta)$ is K-stable (resp. K-semistable) if and only if $\delta(X,\Delta) > 1$ (resp. $\delta(X,\Delta) \geq 1$).
\end{theorem}

Combining the above results, we obtain the following characterization when $\mathrm{Aut}(X, D)$ is discrete:
\begin{equation}\label{R(X,D)}
R(X, D) = \sup\left\{ \lambda > 0 \ \middle| \ \delta\big(X, (1 - \lambda)D\big) > 1 \right\}.
\end{equation}

One of the key issues we need to address toward proving the Main Theorem is how to estimate or evaluate the $\delta$-invariant. To explain our method, we restrict our attention to the case of surfaces.

Let $(S, \Delta)$ be a two-dimensional log Fano pair, i.e., a log del Pezzo surface. Consider a birational morphism $f: \hat{S} \rightarrow S$ and let $G$ be a prime divisor on $\hat{S}$. Suppose that $f$ is a plt blowup associated to $G$, i.e.,
\begin{itemize}
    \item $-G$ is $f$-ample;
    \item the pair $(\hat{S}, \hat{\Delta} + G)$ is plt, where $\hat{\Delta}$ denotes the strict transform of $\Delta$.
\end{itemize}

For a real number $t \in (0, \tau_{S,\Delta}(G))$, consider the Zariski decomposition
\[
    -f^*(K_S + \Delta) - tG \equiv P(t) + N(t),
\]
where $P(t)$ and $N(t)$ are the positive and the negative parts, respectively. Let $q$ be a point on $G$. Define
\begin{equation}\label{eq:SS}
    S(W^G_{\bullet,\bullet}; q) := \frac{2}{(-K_S - \Delta)^2} \int_0^{\tau_{S,\Delta}(G)} (P(t) \cdot G) \cdot \mathrm{ord}_q(N(t)|_G) + \frac{1}{2}(P(t) \cdot G)^2 \, dt.
\end{equation}

We recall the following adjunction formula
\[
    (K_{\hat{S}} + \hat{\Delta} + G)|_G = K_G + \Delta_G,
\]
where $\Delta_G$ is the different of the pair $(\hat{S}, \hat{\Delta} + G)$. If $q$ is a quotient singularity of type $\frac{1}{n}(a, b)$, then
\[
    A_{G, \Delta_G}(q) = \frac{1}{n} - (\hat{\Delta} \cdot G)_q.
\]

\begin{theorem}[{\cite[Theorem 4.8 (2) and Corollary 4.9]{Fuj23}}]\label{AZ}
    Suppose that $\hat{S}$ is a Mori dream space. Then the local $\delta$-invariant of $(S, \Delta)$ at a point $p$ in $c_S(G)$ satisfies
    \[
        \delta_p(S, \Delta) \geq \min \left\{ \frac{A_{S, \Delta}(G)}{S_{S, \Delta}(G)}, \inf_{q \in f^{-1}(p)} \left\{ \frac{A_{G, \Delta_G}(q)}{S(W^G_{\bullet,\bullet}; q)} \right\} \right\}.
    \]
    \end{theorem}

To apply Theorem~\ref{AZ}, it is necessary to verify whether the surface $\hat{S}$ is a Mori dream space. For example, weak del Pezzo surfaces are Mori dream spaces because they are of Fano type (\cite[Corollary 1.3.2]{BCHM}). In particular, if $G$ is a prime divisor on $S$ itself (i.e., $f = \mathrm{id}_S$) and the pair $(S, \Delta + G)$ is plt, then
\[
    \delta_p(S, \Delta) \geq \min \left\{ \frac{1}{S_{S, \Delta}(G)}, \frac{A_{G, \Delta_G}(p)}{S(W^G_{\bullet,\bullet}; p)} \right\}.
\]

Now, consider a weak del Pezzo surface $T$ containing $(-2)$-curves $A_1, \ldots, A_N$ and $(-1)$-curves $E_1, \ldots, E_m$. By \cite[Corollary 3.3.(2)]{Nik96}, the Mori cone $\overline{NE}(T)$ is generated by the classes $[A_i]$ and $[E_j]$. Assume that the dual graph of $\cup_{i=1}^nA_i$ is a Dynkin diagram of type $\mathrm{A}_n$ for some $1\leq n\leq N$. There then exists a contraction 
$\tau:T\rightarrow \hat{T}$ such that $$\tau(\cup_{i=1}^nA_i)=p\in \hat{T}$$
where $\hat{T}$ is a normal $\mathbb{Q}$-factorial projective surface, and $p$ is an $\mathrm{A}_n$ singularity. Since $\tau_*[A_i] = 0 \in \overline{NE}(\hat{T})$ for $i = 1, \ldots, n$, the following proposition follows directly from \cite[Theorem 1.1]{Oka16}:

\begin{proposition}\label{MDS}
    In the above setting, the surface $\hat{T}$ is a Mori dream space. Moreover, the Mori cone $\overline{NE}(\hat{T})$ is spanned by the extremal classes $[\tau(A_{n+1})], \ldots, [\tau(A_N)]$, $[\tau(E_1)], \ldots, [\tau(E_M)]$.
\end{proposition}

\section{K-unstable del Pezzo surfaces}

The surfaces $S_1$ and $S_2$ are the only smooth del Pezzo surfaces that are K-unstable; all other smooth del Pezzo surfaces are either K-polystable or K-stable.

To apply Theorem~\ref{Kst and delta}, it is essential to verify that the relevant automorphism groups are finite. Although the surfaces $S_1$ and $S_2$ themselves have infinite automorphism groups, the automorphism groups of the pairs $(S_1, C^1)$ and $(S_2, C^2)$ are finite for smooth anticanonical divisors $C^1$ and $C^2$. This follows from the natural inclusions
$$
\mathrm{Aut}(S_i, C^i) \hookrightarrow \mathrm{Aut}(\mathbb{P}^2, \phi_i(C^i))
$$
for each $i$, where $\phi_i(C^i)$ is a smooth plane cubic curve. Since the group $\mathrm{Aut}(\mathbb{P}^2, \phi_i(C^i))$ is finite (see, for example, \cite[Th\'eor\`eme~3.1]{B13}), it follows that $\mathrm{Aut}(S_i, C^i)$ is finite as well.

To prove the Main Theorem, we invoke \eqref{R(X,D)}.  It allows us to deduce the Main Theorem directly from the following two theorems by identifying a value of $\lambda_0$ such that  
\[
\delta(S_i, (1 - \lambda_0)C^i) = 1.
\]  
It is worth noting that the pair $(S_i, (1 - \lambda_0)C^i)$ is strictly K-semistable, owing to the finiteness of the automorphism group $\mathrm{Aut}(S_i, C^i)$.

\begin{theorem}\label{thm1} Let $C^1$ be a smooth anticanonical divisor on $S_1$.
\begin{enumerate}
   \item  If $C^1$ is tangent to the $0$-curve at the intersection point of $E$ and $C^1$, then 
    \[
        \delta(S_1,(1-\lambda)C^1)
        \left\{\aligned
            &=\frac{6}{7\lambda}&\text{ for } \frac{13}{14}\leq\lambda\leq1,  \\[.5em]
            &=\frac{3+6\lambda}{10\lambda}&\text{ for } \frac{5}{22}\leq\lambda\leq\frac{13}{14},  \\[.5em]
              &\geq \frac{48}{25} &\text{ for }  0<\lambda\leq\frac{5}{22}.
        \endaligned\right.
    \]
    \item Otherwise, 
    \[
        \delta(S_1,(1-\lambda)C^1)
       \left\{\aligned
            &=\frac{6}{7\lambda}&\text{ for } \frac{13}{14}\leq\lambda\leq1,     \\[.5em]
          & = \frac{4+4\lambda}{9\lambda}&\text{ for } \frac{1}{2}\leq\lambda\leq\frac{13}{14}  \\[.5em]
            &  \geq \frac{4}{3}&\text{ for } 0<\lambda\leq\frac{1}{2}.
        \endaligned\right.
    \]
    \end{enumerate}
\end{theorem}
Note that $C^1$ is tangent to the 0-curve at the intersection point of $E$ and $C^1$ if and only if~$x$ is an inflection point of the smooth plane cubic curve $\phi_1(C^1)$.

\begin{theorem}\label{thm2}
Let $C^2$ be a smooth anticanonical divisor on $S_2$.
\begin{enumerate}
\item If $C^2$ passes through the intersection of two $(-1)$-curves, then
    \[
        \delta(S_2,(1-\lambda)C^2)
        \left\{\aligned
           &= \frac{21}{25\lambda}  &\text{ for } \frac{23}{25}\leq\lambda\leq 1,  \\[.5em]
           &= \frac{7+7\lambda}{16\lambda} &\text{ for } \frac{13}{35}\leq\lambda\leq\frac{23}{25},  \\[.5em]
            &\geq \frac{42}{23} &\text{ for } 0<\lambda\leq\frac{23}{73}.
         \endaligned\right.
    \]

  \item Otherwise, then
         \[
        \delta(S_2,(1-\lambda)C^2)
      \left\{\aligned
         &  = \frac{21}{25\lambda} &\text{ for } \frac{18}{25}\leq\lambda\leq 1,  \\[.5em]
         &   \geq \frac{7}{6} &\text{ for } 0<\lambda\leq\frac{18}{25}.
         \endaligned\right.
    \]
    \end{enumerate}
\end{theorem}

We remark here that $C^2$ passes through the intersection of two $(-1)$-curves if and only if the line determined by $x_1$ and $x_2$ is tangent to the smooth plane cubic curve $\phi_2(C^2)$ at either $x_1$ or~$x_2$.

The proofs of these theorems will be presented in the following section. 

\section{Proofs}
To prove Theorems~\ref{thm1} and~\ref{thm2}, we compute or estimate the local $\delta$-invariants of the pairs $(S_i, (1-\lambda)C^i)$ at every point in $S_i$, for $i = 1, 2$. Throughout this section, given a pseudoeffective divisor $D(t)$ depending on a variable $t$, we always denote its Zariski decomposition by $$D(t) \equiv P(t) + N(t),$$ where $P(t)$ is the positive part and $N(t)$ is the negative part. 
In addition, given a divisor~$G$ over $S_i$ and a point $q$ on $G$, we consistently denote the integrand in \eqref{eq:SS} by
\[
h(G, q, t) := (P(t) \cdot G) \cdot \mathrm{ord}_q\left(N(t)\big|_G\right) + \frac{1}{2} (P(t) \cdot G)^2.
\]

To apply Theorem~\ref{AZ}, we will frequently use the notion of the weighted $(1,m)$-blowup, which we now construct.

Let $p$ be a point in $S_i$ and $D$ be a smooth curve passing through $p$. Denote by $\pi^i_1:T^i_1\rightarrow S_i$ the blowup at $p$. For $2\leq j\leq m$, define $\pi^i_j:T^i_j\rightarrow T^i_{j-1}$ inductively to be the blowup at the intersection point of the strict transform of $D$ in $T^i_{j-1}$ and the exceptional curve of~$\pi^i_{j-1}$.
Next, let $\tau_i: T^i_m \to \hat{S}_i$ be the birational morphism obtained by contracting the exceptional curves of $\pi^i_1, \ldots, \pi^i_{m-1}$. Observe that the dual graph of these curves is the Dynkin diagram of type $\mathrm{A}_{m-1}$. The image of the exceptional divisor of $\pi^i_m$ under $\tau_i$ is denoted by $\hat{G}$ (resp.~$\hat{M}$) when $i = 1$ (resp. $i = 2$). If $C^1$ (resp. $C^2$) intersects $D$ at $p$ with multiplicity~$m$, the contraction of $\hat{G}$ (resp. $\hat{M}$) defines a plt blowup $\sigma_1: \hat{S}_1 \to (S_1,(1-\lambda)C^1)$ (resp. $\sigma_2: \hat{S}_2 \to (S_2,(1-\lambda)C^2)$). We simply call $\sigma_i$  the $(1,m)$-blowup at $p$ with respect to the curve $D$. In fact, if we choose local analytic coordinates $x,y$ near $p$ such that~$D$ is given locally by the zero set of $y$, then the above blowup agrees with the  weighted~$(1,m)$-blowup.

We adopt the following notation for curves:
\begin{itemize}
    \item A curve on $S_i$ is denoted by an uppercase Roman letter, possibly with a numeric superscript (e.g., $D$, $C^1$).
    \item A curve on $T^i_j$ is denoted by an uppercase Roman letter with subscript $j$, possibly with a numeric superscript (e.g., $D_j$, $C_j^1$).
    \item  The exceptional curve of $\pi^1_j$ (resp. $\pi^2_j$) on $T^1_j$ (resp. $T^2_j$) is denoted by $G^j_j$ (resp. $M^j_j$).
    \item If a curve on $T^i_j$ is the strict transform of a curve on $T^i_{j-1}$ (or $S_i$) via $\pi_j^i$, it is denoted by the same Roman letter and superscript, with the subscript updated to~$j$.
    \item The strict transform of a curve on $S_i$ via $\sigma_i$ is denoted by the same Roman letter and superscript, with a hat (e.g., $\hat{D}$, $\hat{C}^1$).
    \item The point of intersection between $\hat{G}$ (or $\hat{M}$) and the strict transform of a curve from $S_i$ under $\sigma_i$ is denoted by $q$, with the same Roman letter and superscript of the intersecting curve as a subscript (e.g., $q_D$, $q_{C^1}$).
\end{itemize}

\subsection{Proof of Theorem~\ref{thm1}} 
Let $p$ be a point on $S_1$.
There is a unique $0$-curve passing through the point $p$. In fact, it is the member of the pencil $|\phi_1^\ast\mathcal{O}_{\mathbb{P}^2}(1)-E|$. We will denote this $0$-curve by $F$ throughout this subsection.
Then we have the numerical equivalence
\[
-K_{S_1} - (1 - \lambda) C^1 - tF \equiv 2\lambda E + (3\lambda - t) F.
\]
This divisor is pseudoeffective only when $t \leq 3\lambda$. Its Zariski decomposition is given by
      \begin{align*}
        P(t)
        =\begin{cases}
            2\lambda E+(3\lambda-t)F\\
            (3\lambda-t)(E+F)
        \end{cases};  \ 
         N(t)=\begin{cases}
            0, &\ \ \ 0\leq t\leq \lambda,\\
            (u-\lambda)E,& \ \ \ \lambda\leq t\leq3\lambda.
        \end{cases}
    \end{align*}
    Then, 
    \begin{align*}
        \mathrm{vol}\left(-K_{S_1}-(1-\lambda)C^1-tF\right)=P(t)^2=
        \begin{cases}
            8\lambda^2-4\lambda t,&0\leq t\leq\lambda,\\
            (3\lambda-t)^2,&\lambda\leq t\leq3\lambda,
        \end{cases}
    \end{align*}
    and hence
    \[
        S_{S_1,(1-\lambda)C^1}(F)=\frac{13}{12}\lambda.
    \]

    We then obtain an upper bound
\begin{equation}\label{1_F-E_ub}
    \delta_p(S_1, (1 - \lambda)C^1) \leq \frac{A_{S_1, (1 - \lambda)C^1}(F)}{S_{S_1, (1 - \lambda)C^1}(F)} = \frac{12}{13\lambda}.
\end{equation}

We now consider the exceptional divisor $E$ on $S_1$. Then
\[
    -K_{S_1} - (1 - \lambda)C^1 - tE \equiv (2\lambda - t)E + 3\lambda F.
\]
The divisor is nef and big for $0 \leq t \leq 2\lambda$, and not pseudoeffective for $t > 2\lambda$. Then we compute
\[
    \mathrm{vol}\left(-K_{S_1} - (1 - \lambda)C^1 - tE\right) = -t^2 - 2\lambda t + 8\lambda^2,
\]
and hence
\[
    S_{S_1, (1 - \lambda)C^1}(E) = \frac{7\lambda}{6}.
\]
This shows that if $p$ belongs to $E$, then
\begin{equation}\label{1_E-C_ub}
    \delta_p(S_1, (1 - \lambda)C^1) \leq \frac{6}{7\lambda}.
\end{equation}

\begin{lemma}\label{1_S-C}
Suppose that $p$ is in $S_1\setminus C^1$. Then
    \[
        \delta_p(S_1,(1-\lambda)C^1)=\left\{\aligned &\frac{12}{13\lambda}& \text{ if }p\not\in E,\\
        &\frac{6}{7\lambda}&\text{ if }p\in E.
        \endaligned\right.
    \]
\end{lemma}

\begin{proof}
Suppose that $p$ is not in $E$. 
We choose an irreducible curve $L$ in the linear system $|E + F|$ passing through $p$. Then $\phi_1(L)$ is a line not passing through  $x$. We compute
\[
    -K_{S_1} - (1 - \lambda)C^1 - tL \equiv (2\lambda - t)E + (3\lambda - t)F,
\]
which is nef and big for $0 \leq t < 2\lambda$, and not pseudoeffective for $t > 2\lambda$. Therefore, we have
\[
    \mathrm{vol}\left(-K_{S_1} - (1 - \lambda)C^1 - tL\right) = 8\lambda^2 - 6\lambda t + t^2,
\]
and hence
\[
    S_{S_1, (1 - \lambda)C^1}(L) = \frac{5}{6\lambda}.
\]

Meanwhile, we also have
\[
    S(W^L_{\bullet, \bullet}; p) = \frac{2}{(-K_{S_1} - (1 - \lambda)C^1)^2} \int_0^{2\lambda} \frac{1}{2} (P(t) \cdot L)^2 \, dt 
    = \frac{1}{4\lambda^2} \int_0^{2\lambda} \frac{(3\lambda - t)^2}{2} \, dt = \frac{13\lambda}{12}.
\]

Put $K_L + \Delta_L := (K_{S_1} + (1 - \lambda)C^1 + L)|_L$. Then $A_{L, \Delta_L}(p) = 1$ since $p$ is not in $C^1$. It then follows from Theorem~\ref{AZ} that
\begin{equation}\label{1_F-E_lb}
    \delta_p(S_1, (1 - \lambda)C^1) \geq \min\left\{ \frac{1}{13\lambda / 12}, \frac{1}{5\lambda / 6} \right\} = \frac{12}{13\lambda}.
\end{equation}

Consequently, combining \eqref{1_F-E_ub} and \eqref{1_F-E_lb}, we conclude the proof for the case when $p$ does not lie on $E$.

Suppose that $p$ is on $E$, then
    \[
        S(W^E_{\bullet,\bullet};p)=\frac{2}{(-K_{S_1}-(1-\lambda)C^1)^2}\int_0^{2\lambda}\frac{1}{2}(P(t)\cdot E)^2dt=\frac{1}{4\lambda^2}\int_0^{2\lambda}\frac{(\lambda+t)^2}{2}dt=\frac{13\lambda}{12}
    \]
    Put $K_E+\Delta_E:=(K_{S_1}+(1-\lambda)C^1+E)|_E$, then $A_{E,\Delta_E}(p)=1$. It then follows from Theorem~\ref{AZ} that
    \begin{equation}\label{1_E-C_lb}
        \delta_p(S_1,(1-\lambda)C^1)\geq\min\left\{\frac{1}{7\lambda/6},\frac{1}{13\lambda/12}\right\}=\frac{6}{7\lambda}.
    \end{equation}
    Consequently, combining \eqref{1_E-C_ub} and \eqref{1_E-C_lb} determines the value of the local $\delta$-invariant  for the case when $p$ belongs to $E$.
\end{proof}

\begin{lemma}\label{1_C-E_notflex} Suppose that $p$ is a point in $C^1\setminus E$ such that $\phi_1(p)$ is not an inflection point of the smooth cubic curve $\phi_1(C^1)$ and $C^1$ is transverse to $F$. Then, $$\min\left\{\frac{12}{13\lambda},\frac{4+8\lambda}{11\lambda},\frac{48}{25}\right\}\leq\delta_p(S_1,(1-\lambda)C^1)\leq\min\left\{\frac{12}{13\lambda},\frac{4+8\lambda}{11\lambda}\right\}.$$
    In particular, for $\lambda\geq\frac{25}{82}$, we have $$\delta_p(S_1,(1-\lambda)C^1)=\min\left\{\frac{12}{13\lambda},\frac{4+8\lambda}{11\lambda}\right\}=\begin{cases}
        \frac{12}{13\lambda}, &\frac{10}{13}\leq\lambda\leq 1,\\
        \frac{4+8\lambda}{11\lambda}, &\frac{25}{82}\leq\lambda\leq\frac{10}{13}.
    \end{cases}$$
\end{lemma}

\begin{proof}
    Let $L$ be the unique curve in the linear system $|E + F|$ that is tangent to $C^1$ at~$p$.
    Note that $L$ is irreducible since $C^1$ is transverse to $F$.
    Then, define $\sigma_1:\hat{S}_1\rightarrow (S_1,(1-\lambda)C^1)$ as the $(1,2)$-blowup with respect to $L$. Note that $q_F$ is an $\mathrm{A}_1$ singularity. The construction of $\hat{S}_1$ is illustrated as follows:

    
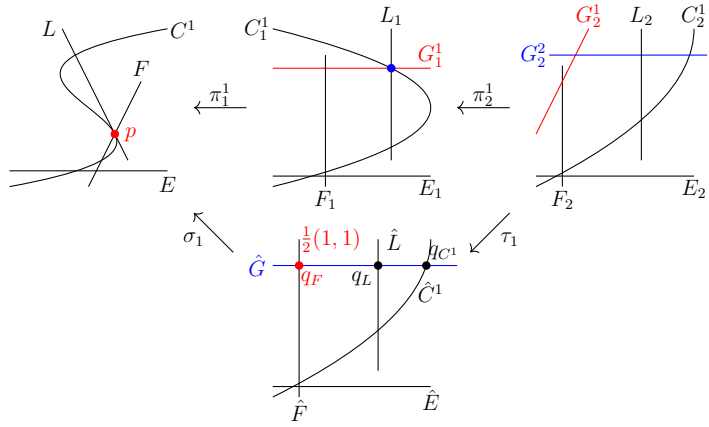
\begin{center}
    \begin{tikzpicture}[scale=0.7, every node/.style={scale=0.7}]
        \draw (-8,-0.7) -- (-5,-.7);
        \draw (-5,-.95) node {$E$};
        \draw (-5.75,-.5) -- (-7,2);
        \draw (-7.25,2) node {$L$};
        \draw (-5.5,1) -- (-6.5,-1);
        \draw (-5.5,1.25) node {$F$};
        \draw (-8,-1) .. controls (-2.5,0) and (-10.5,1) .. (-5,2);
        \draw (-4.7,2) node {$C^1$};
        \filldraw[red] (-6,0) circle (2pt);
        \draw[red] (-5.7,0) node {$p$};

        \draw[->] (-3.5,.5) -- (-4.5,.5);
        \draw (-4,.75) node {$\pi^1_1$};
        \draw[->] (-3.75,-2.25) -- (-4.5,-1.5);
        \draw (-4.5,-2) node {$\sigma_1$};

        \draw (-3,-0.8) -- (0,-.8);
        \draw (0,-1.05) node {$E_1$};
        \draw (-3,-1) .. controls (1,0) and (1,1) .. (-3,2);
        \draw (-3.3,2) node {$C^1_1$};
        \draw[red] (-3,1.25) -- (0,1.25);
        \draw[red] (0,1.5) node {$G^1_1$};
        \draw (-.75,2) -- (-.75,-.5);
        \draw (-.75,2.25) node {$L_1$};
        \draw (-2,-1) -- (-2,1.5);
        \draw (-2,-1.25) node {$F_1$};
        \filldraw[blue] (-.75,1.25) circle (2pt);

        \draw[->] (1.5,.5) -- (.5,.5);
        \draw (1,.75) node {$\pi^1_2$};

        \draw (2,-.8) -- (5,-.8);
        \draw (5,-1.05) node {$E_2$};
        \draw (2,-1) .. controls (4,0) and (5,1) .. (5,2);
        \draw (5,2.25) node {$C^1_2$};
        \draw (2.5,-1) -- (2.5,1.3);
        \draw (2.5,-1.25) node {$F_2$};
        \draw[red] (2,0) -- (3,2);
        \draw[red] (3,2.25) node {$G^1_2$};
        \draw[blue] (2.25,1.5) -- (5.25,1.5);
        \draw[blue] (1.95,1.5) node {$G^2_2$};
        \draw (4,-.5) -- (4,2);
        \draw (4,2.25) node {$L_2$};

        \draw[->] (1.5,-1.5) -- (.75,-2.25);
        \draw (1.5,-2) node {$\tau_1$};

        \draw (-3,-5) .. controls (-1,-4) and (0,-3) .. (0,-2);
        \draw (0,-3) node {$\hat{C}^1$};
        \draw (-3,-4.8) -- (0,-4.8);
        \draw (0,-5.05) node {$\hat{E}$};
        \draw (-2.5,-5) -- (-2.5,-2);
        \draw (-2.5,-5.25) node {$\hat{F}$};
        \draw[blue] (-3, -2.5) -- (.5,-2.5);
        \draw[blue] (-3.3,-2.5) node {$\hat{G}$};
        \draw (-1,-4.5) -- (-1,-2);
        \draw (-0.7,-2) node {$\hat{L}$};
        \filldraw[red] (-2.5,-2.5) circle (2pt);
        \draw[red] (-1.9,-2) node {$\frac{1}{2}(1,1)$};
        \draw[red] (-2.25,-2.75) node {$q_F$};
        \filldraw (-1,-2.5) circle (2pt);
        \draw (-1.3,-2.75) node {$q_L$};
        \filldraw (-1/12,-2.5) circle (2pt);
        \draw (.25,-2.25) node {$q_{C^1}$};
    \end{tikzpicture}
    \phantom{spacespa}
    \end{center}

    Note that $\hat{L}\equiv\hat{E}+\hat{F}-\hat{G}$ and
    \[
        \sigma_1^\ast L=\hat{L}+2\hat{G},\quad\sigma_1^\ast K_{S_1}=K_{\hat{S}_1}-2\hat{G},\quad\sigma_1^\ast C^1=\hat{C}^1+2\hat{G},\quad\sigma_1^\ast F=\hat{F}+\hat{G},\quad\sigma_1^\ast E=\hat{E}.
    \]
    In particular, we have
    \[
         A_{S_1,(1-\lambda)C^1}(\hat{G})=1+2\lambda.
    \]
    
    The intersections are given as follows:
        \[
    \begin{split}
        &\hat{E}^2=\hat{L}^2=-1,\quad\hat{F}^2=\hat{G^2}=-\frac{1}{2},\quad\hat{E}\cdot\hat{F}=\hat{G}\cdot\hat{L}=1,\\
        &\hat{E}\cdot\hat{G}=\hat{E}\cdot\hat{L}=\hat{F}\cdot\hat{L}=0,\quad\hat{F}\cdot\hat{G}=\frac{1}{2}.
    \end{split}
    \]
    
    Since $T_2$ is a weak del Pezzo surface, $\hat{S}$ is a Mori dream space, and its Mori cone is
    \[
        \overline{NE}(\hat{S}_1)=\mathrm{Cone}\{[\hat{E}],[\hat{F}],[\hat{G}],[\hat{L}]\},
    \]
    by Proposition~\ref{MDS}. We have
  \begin{align*}
      \sigma_1^\ast\left(-K_{S_1}-(1-\lambda)C^1\right)-t\hat{G}
       &\equiv2\lambda\hat{E}+3\lambda\hat{F}+(3\lambda-t)\hat{G}\\
       &\equiv(t-3\lambda)\hat{L}+(5\lambda-t)\hat{E}+(6\lambda-t)\hat{F},
    \end{align*}
    and it is pseudoeffective only for  $t\leq5\lambda$. Its Zariski decomposition is given as follows:
      \begin{align*}
        P(t)=\begin{cases}
            2\lambda\hat{E}+3\lambda\hat{F}+(3\lambda-t)\hat{G}\\
            (5\lambda-t)\hat{E}+(6-\lambda)\hat{F}\\
            (5\lambda-t)(\hat{E}+2\hat{F})
        \end{cases};\quad
        N(t)=\begin{cases}
            0, &\ \ \ 0\leq t\leq3\lambda,\\
            (t-3\lambda)\hat{L}, &\ \ \ 3\lambda\leq t\leq4\lambda,\\
            (t-4\lambda)\hat{F}+(t-3\lambda)\hat{L}, &\ \ \ 4\lambda\leq t\leq5\lambda.
        \end{cases}
        \\
    \end{align*}
    Then,
    \begin{align*}
        \mathrm{vol}\left(\sigma_1^\ast\left(-K_{S_1}-(1-\lambda)C^1\right)-t\hat{G}\right)=P(t)^2=\begin{cases}
            8\lambda^2-\frac{1}{2}t^2, &0\leq t\leq 3\lambda,\\
            \frac{1}{2}t^2-6\lambda t+17\lambda^2, &3\lambda\leq t\leq4\lambda,\\
            (5\lambda-t)^2, &4\lambda\leq t\leq5\lambda,
        \end{cases}
    \end{align*}
    and hence
    \[
        S_{S_1,(1-\lambda)C^1}(\hat{G})=\frac{11\lambda}{4}.
    \]
    Combining \eqref{1_F-E_ub}, we obtain the upper bound
    \begin{equation}\label{1_E-C_notflex_ub}
        \delta_p(S_1,(1-\lambda)C^1)
        \leq\min\left\{\frac{A_{S_1,(1-\lambda)C^1}(F)}{S_{S_1,(1-\lambda)C^1}(F)},\frac{A_{S_1,(1-\lambda)C^1}(\hat{G})}{S_{S_1,(1-\lambda)C^1}(\hat{G})}\right\}=\min\left\{\frac{12}{13\lambda},\frac{4+8\lambda}{11\lambda}\right\}.
    \end{equation}
    On the other hand, for each $q$ on $\hat{G}$, 
    \[
        h(\hat{G},q,t)=\begin{cases}
            \frac{1}{8}t^2, &0\leq t\leq3\lambda,\\
            \frac{6\lambda-t}{2}\cdot\mathrm{ord}_q(t-3\lambda)q_L+\frac{(6\lambda-t)^2}{8}, &3\lambda\leq t\leq4\lambda,\\
            (5\lambda-t)\cdot\mathrm{ord}_q\left(\frac{t-4\lambda}{2}q_F+(t-3\lambda)q_L\right)+\frac{(5\lambda-t)^2}{2}, &4\lambda\leq t\leq5\lambda,
        \end{cases}
    \]
    and hence
    \[
        S(W^{\hat{G}}_{\bullet,\bullet};q)=\begin{cases}
            \frac{25\lambda}{48}, &q\neq q_F,q_L,\\
            \frac{13\lambda}{24}, & q=q_F,\\
            \frac{5\lambda}{6}, &q=q_L.
        \end{cases}
    \]
    Put $K_{\hat{G}}+\Delta_{\hat{G}}:=(K_{\hat{S}_1}+(1-\lambda)\hat{C}+\hat{G})|_{\hat{G}}$, then
    \[
        A_{\hat{G},\Delta_{\hat{G}}}(q)=\begin{cases}
            1, &q\neq q_F,q_{C^1},\\
            \frac{1}{2}, &q=q_F,\\
            \lambda, &q=q_{C^1}.
        \end{cases}
    \]
    It then follows from Theorem~\ref{AZ} that
    \begin{equation}\label{1_E-C_notflex_lb}
        \delta_p(S_1,(1-\lambda)C^1)
        \geq\min\left\{\frac{4+8\lambda}{11\lambda},\frac{48}{25\lambda},\frac{48}{25},\frac{12}{13\lambda},\frac{6}{5\lambda}\right\}
        =\begin{cases}
            \frac{12}{13\lambda},&\frac{10}{13}\leq\lambda\leq 1,\\
            \frac{4+8\lambda}{11\lambda},&\frac{25}{82}\leq\lambda\leq\frac{10}{13}\\
            \frac{48}{25},&0<\lambda\leq\frac{25}{82}.
        \end{cases}
    \end{equation}
    Consequently, \eqref{1_E-C_notflex_ub} and \eqref{1_E-C_notflex_lb} complete the proof.
\end{proof}

\begin{lemma}\label{1_C-E_flex}
    Suppose that $p$ is on $C^1\setminus E$ such that $\phi_1(p)$ is an inflection point of $\phi_1(C^1)$. Then, $$\min\left\{\frac{12}{13\lambda},\frac{12+36\lambda}{43\lambda},\frac{48}{17}\right\}\leq\delta_p(S_1,(1-\lambda)C^1)\leq\min\left\{\frac{12}{13\lambda},\frac{12+36\lambda}{43\lambda}\right\}.$$
    In particular, for $\lambda\geq\frac{17}{121}$, we have $$\delta_p(S_1,(1-\lambda)C^1)=\min\left\{\frac{12}{13\lambda},\frac{12+36\lambda}{43\lambda}\right\}=\begin{cases}
        \frac{12}{13\lambda}, &\frac{10}{13}\leq\lambda\leq1,\\
        \frac{12+36\lambda}{43\lambda},&\frac{17}{121}\leq\lambda\leq\frac{10}{13}.
    \end{cases}$$
\end{lemma}

\begin{proof}

    As before, let $L$ be the unique curve in $|E + F|$ that is tangent to $C^1$ at $p$. The curve~$L$ is irreducible because $\phi_1(p)$ is an inflection point of $\phi_1(C^1)$.
Let $\sigma_1:\hat{S}_1\rightarrow(S_1,(1-\lambda)C^1)$ be the $(1,3)$-blowup with respect to $L$. Note that the point $q_F$ is an $\mathrm{A}_2$ singularity. This can be illustrated as follows:

    
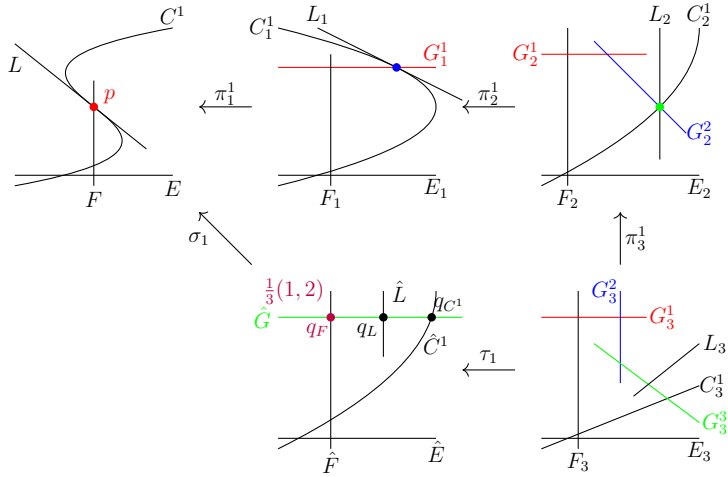
\begin{center}
    \begin{tikzpicture}[scale=0.7, every node/.style={scale=0.7}]
        \draw (-7,1.2) -- (-4,1.2);
        \draw (-4,.95) node {$E$};
        \draw (-7,1) .. controls (-1.5,2) and (-9.5,3) .. (-4,4);
        \draw (-4,4.25) node {$C^1$};
        \draw (-5.5,1) -- (-5.5,3);
        \draw (-5.5,.75) node {$F$};
        \draw (-7,3.7) -- (-4.5,1.7);
        \draw (-7,3.3) node {$L$};
        \filldraw[red] (-5.5,2.5) circle (2pt);
        \draw[red] (-5.2,2.7) node {$p$};

        \draw[->] (-2.5,2.5) -- (-3.5,2.5);
        \draw (-3,2.75) node {$\pi^1_1$};
        \draw[->] (-2.5,-.5) -- (-3.5,.5);
        \draw (-3.5,0) node {$\sigma_1$};

        \draw (-2,1.2) -- (1,1.2);
        \draw (1,.95) node {$E_1$};
        \draw (-2,1) .. controls (2,2) and (2,3) .. (-2,4);
        \draw (-2.3,4) node {$C^1_1$};
        \draw[red] (-2,3.25) -- (1,3.25);
        \draw[red] (1,3.5) node {$G^1_1$};
        \draw (-1.25,4) -- (1.5,2.625);
        \draw (-1.25,4.25) node {$L_1$};
        \draw (-1,1) -- (-1,3.5);
        \draw (-1,.75) node {$F_1$};
        \filldraw[blue] (.25,3.25) circle (2pt);

        \draw[->] (2.5,2.5) -- (1.5,2.5);
        \draw (2,2.75) node {$\pi^1_2$};

        \draw (3,1.2) -- (6,1.2);
        \draw (6,.95) node {$E_2$};
        \draw (3,1) .. controls (5,2) and (6,3) .. (6,4);
        \draw (6,4.25) node {$C^1_2$};
        \draw (3.5,1) -- (3.5,4);
        \draw (3.5,.75) node {$F_2$};
        \draw[red] (3,3.5) -- (5,3.5);
        \draw[red] (2.7,3.5) node {$G^1_2$};
        \draw[blue] (4,3.75) -- (5.75,2);
        \draw[blue] (6,2) node {$G^2_2$};
        \draw (5.25,1.5) -- (5.25,4);
        \draw (5.25, 4.25) node {$L_2$};
        \filldraw[green] (5.25,2.5) circle (2pt);

        \draw[->] (4.5,-.5) -- (4.5,.5);
        \draw (4.8,0) node {$\pi^1_3$};

        \draw (3,-3.8) -- (6,-3.8);
        \draw (6,-4.05) node {$E_3$};
        \draw (3,-4) -- (6,-2.8);
        \draw (6.25,-2.8) node {$C^1_3$};
        \draw (3.7,-4) -- (3.7,-1);
        \draw (3.7,-4.25) node {$F_3$};
        \draw[red] (3,-1.5) -- (5,-1.5);
        \draw[red] (5.3,-1.5) node {$G^1_3$};
        \draw[blue] (4.5,-1) -- (4.5,-2.75);
        \draw[blue] (4.2,-1) node {$G^2_3$};
        \draw[green] (4,-2) -- (6,-3.5);
        \draw[green] (6.3,-3.5) node {$G^3_3$};
        \draw (4.75,-3) -- (6,-2);
        \draw (6.3,-2) node {$L_3$};

        \draw[->] (2.5,-2.5) -- (1.5,-2.5);
        \draw (2,-2.25) node {$\tau_1$};

        \draw (-2,-4) .. controls (0,-3) and (1,-2) .. (1,-1);
        \draw (1,-2) node {$\hat{C}^1$};
        \draw (-2,-3.8) -- (1,-3.8);
        \draw (1,-4.05) node {$\hat{E}$};
        \draw (-1,-4) -- (-1,-1);
        \draw (-1,-4.25) node {$\hat{F}$};
        \draw[green] (-2, -1.5) -- (1.5,-1.5);
        \draw[green] (-2.3,-1.5) node {$\hat{G}$};
        \draw (0,-2.25) -- (0,-1);
        \draw (0.3,-1) node {$\hat{L}$};
        \filldraw[purple] (-1,-1.5) circle (2pt);
        \draw[purple] (-1.7,-1) node {$\frac{1}{3}(1,2)$};
        \draw[purple] (-1.25,-1.75) node {$q_F$};
        \filldraw (0,-1.5) circle (2pt);
        \draw (-.3,-1.75) node {$q_L$};
        \filldraw (11/12,-1.5) circle (2pt);
        \draw (1.25,-1.25) node {$q_{C^1}$};
    \end{tikzpicture}
    \phantom{spacespace}
    \end{center}

    Note that $\hat{L} \equiv \hat{E} + \hat{F} - 2\hat{G}$, and the pullbacks by $\sigma_1$ are given by
\[
    \sigma_1^\ast L = \hat{L} + 3\hat{G}, \quad 
    \sigma_1^\ast K_{S_1} = K_{\hat{S}_1} - 3\hat{G}, \quad 
    \sigma_1^\ast C^1 = \hat{C}^1 + 3\hat{G}, \quad 
    \sigma_1^\ast F = \hat{F} + \hat{G}.
\]
In particular, the log discrepancy of $\hat{G}$ with respect to the pair $(S_1, (1 - \lambda)C^1)$ is
\[
    A_{S_1, (1 - \lambda)C^1}(\hat{G}) = 1 + 3\lambda.
\]

The intersection numbers on $\hat{S}_1$ are given by
 \[
\begin{split}
    &\hat{E}^2 = -1, \quad 
    \hat{F}^2 = \hat{G}^2 = -\tfrac{1}{3}, \quad 
    \hat{L}^2 = -2, \quad 
    \hat{E} \cdot \hat{F} = \hat{G} \cdot \hat{L} = 1, \\
    &\hat{E} \cdot \hat{G} = \hat{E} \cdot \hat{L} = \hat{F} \cdot \hat{L} = 0, \quad 
    \hat{F} \cdot \hat{G} = \tfrac{1}{3}.
\end{split}
\]

    Since $T_3$ is a weak del Pezzo surface, 
    $\hat{S}_1$ is also a Mori dream space, and its Mori cone is generated by 
 $[\hat{E}]$, $[\hat{F}]$, $[\hat{G}]$ and $[\hat{L}]$. From the numerical equivalence
    \[
    \begin{split}
        \sigma_1^\ast\left(-K_{S_1}-(1-\lambda)C^1\right)-t\hat{G}&\equiv2\lambda\hat{E}+3\lambda\hat{F}+(3\lambda-t)\hat{G}\\
        &\equiv\frac{t-3\lambda}{2}\hat{L}+\frac{7\lambda- t}{2}\hat{E}+\frac{9\lambda-t}{2}\hat{F},
    \end{split}
    \]
    we see that the divisor is pseudoeffective only for $t\leq 7\lambda$. The Zariski decomposition is given by
        \begin{align*}
        P(t)=\begin{cases}
            2\lambda\hat{E}+3\lambda\hat{F}+(3\lambda-t)\hat{G}\\
            \frac{7\lambda-t}{2}\hat{E}+\frac{9\lambda-t}{2}\hat{F}\\
            \frac{7\lambda-t}{2}(\hat{E}+3\hat{F})
        \end{cases};\quad
        N(t)=\begin{cases}
            0, &\ \ \ 0\leq t\leq 3\lambda,\\
            \frac{t-3\lambda}{2}\hat{L}, &\ \ \ 3\lambda\leq t\leq 6\lambda,\\
            (t-6\lambda)\hat{F}+\frac{t-3\lambda}{2}\hat{L}, &\ \ \ 6\lambda\leq t\leq7\lambda.
        \end{cases}
    \end{align*}
    Thus, we have
    \[
        \mathrm{vol}\left(\sigma_1^\ast\left(-K_{S_1}-(1-\lambda)C^1\right)-t\hat{G}\right)=P(t)^2=\begin{cases}
            8\lambda^2-\frac{1}{3}t^2, &0\leq t\leq 3\lambda,\\
            \frac{1}{6}t^2-3\lambda t+\frac{25}{2}\lambda^2, &3\lambda\leq t\leq6\lambda,\\
            \frac{(7\lambda-t)^2}{2}, &6\lambda\leq t\leq 7\lambda,
        \end{cases}
    \]
    and hence
    \[
        S_{S_1,(1-\lambda)C^1}(\hat{G})=\frac{43\lambda}{12}.
    \]
    Together with \eqref{1_F-E_ub}, this yields the upper bound
    \begin{equation}\label{1_C-E_flex_ub}
        \delta_p(S_1,(1-\lambda)C^1)
        \leq\min\left\{\frac{A_{S_1,(1-\lambda)C^1}(F)}{S_{S_1,(1-\lambda)C^1)}(F)},\frac{A_{S_1,(1-\lambda)C^1}(\hat{G})}{S_{S_1,(1-\lambda)C^1}(\hat{G})}\right\}=\min\left\{\frac{12}{13\lambda},\frac{12+36\lambda}{43\lambda}\right\}.   
    \end{equation}
    Meanwhile, for each $q$ on $\hat{G}$, 
    \[
        h(\hat{G},q,t)=\begin{cases}
            \frac{1}{18}t^2, &0\leq t\leq3\lambda,\\
            \frac{9\lambda-t}{6}\cdot\mathrm{ord}_q\frac{t-3\lambda}{2}q_L+\frac{(9\lambda-t)^2}{72}, &3\lambda\leq t\leq 6\lambda,\\
            \frac{7\lambda-t}{2}\cdot\mathrm{ord}_q\left(\frac{t-6\lambda}{3}q_F+\frac{t-3\lambda}{2}q_L\right)+\frac{(7\lambda-t)^2}{8}, &6\lambda\leq t\leq 7\lambda,
        \end{cases}
    \]
    and hence,
    \[
        S(W^{\hat{G}}_{\bullet,\bullet};q)=\begin{cases}
            \frac{17\lambda}{48}, &q\neq q_L,q_F,\\
            \frac{5\lambda}{6}, &q=q_L,\\
            \frac{13\lambda}{36}, &q=q_F.
        \end{cases}
    \]
    Put $K_{\hat{G}}+\Delta_{\hat{G}}:=(K_{\hat{S}_1}+(1-\lambda)\hat{C}+\hat{G})|_{\hat{G}}$, then
    \[
        A_{\hat{G},\Delta_{\hat{G}}}(q)=\begin{cases}
            1, &q\neq q_{C^1},q_F,\\
            \lambda, & q=q_{C^1},\\
            \frac{1}{3}, &q=q_F.
        \end{cases}
    \]
    It then follows from Theorem~\ref{AZ} that
    \begin{equation}\label{1_C-E_flex_lb}
        \delta_p(S_1,(1-\lambda)C^1)\geq\min\left\{\frac{12+36\lambda}{43\lambda},\frac{48}{17\lambda},\frac{48}{17},\frac{12}{13\lambda},\frac{6}{5\lambda}\right\}
        =\begin{cases}
            \frac{12}{13\lambda}, &\frac{10}{13}\leq\lambda\leq 1,\\
            \frac{12+36\lambda}{43\lambda}, &\frac{17}{121}\leq\lambda\leq\frac{10}{13\lambda},\\
            \frac{48}{17}, &0<\lambda\leq\frac{17}{121}.
        \end{cases}
    \end{equation}
    Consequently, \eqref{1_C-E_flex_ub} and \eqref{1_C-E_flex_lb} conclude the proof.
\end{proof}

\begin{lemma}\label{1_CF_tangent}
     Suppose that $p$ is on $C^1\setminus E$ and $C^1$ is tangent to $F$. Then, $$\min\left\{\frac{12}{13\lambda},\frac{1+2\lambda}{3\lambda},\frac{12}{5}\right\}\leq\delta_p(S_1,(1-\lambda_1)C^1)\leq\min\left\{\frac{12}{13\lambda},\frac{1+2\lambda}{3\lambda}\right\}.$$ In particular, for $\lambda\geq\frac{5}{26}$, we have $$\delta_p(S_1,(1-\lambda)C^1)=\min\left\{\frac{12}{13\lambda},\frac{1+2\lambda}{3\lambda}\right\}=\begin{cases}
        \frac{12}{13\lambda}, &\frac{23}{26}\leq\lambda\leq1,\\
        \frac{1+2\lambda}{3\lambda}, &\frac{5}{26}\leq\lambda\leq\frac{23}{26}.
    \end{cases}$$
\end{lemma}

\begin{proof}
Define $\sigma_1: \hat{S}_1 \rightarrow (S_1,(1-\lambda)C^1)$ as the $(1,2)$-blowup with respect to $F$.
Denote the image $\tau(G^1_2)$ by $q_0$ which is an $\mathrm{A}_1$ singularity.  
This process is illustrated as follows:
    
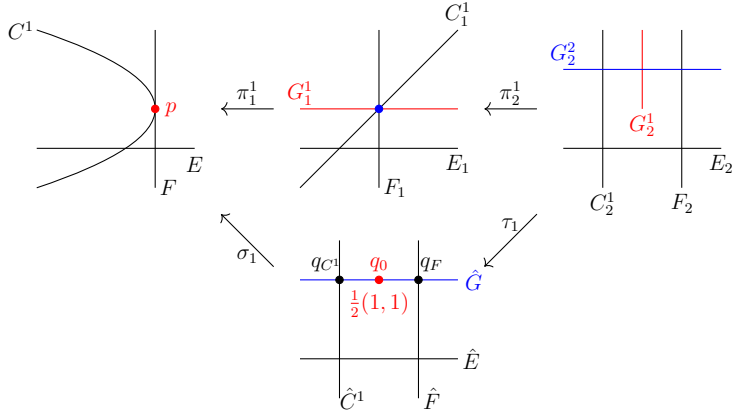
\begin{center}
    \begin{tikzpicture}[scale=0.7, every node/.style={scale=0.7}]
        \draw (-7,1.75) -- (-4,1.75);
        \draw (-4,1.45) node {$E$};
        \draw (-7,1) .. controls (-4,2) and (-4,3) .. (-7,4);
        \draw (-7.3,4) node {$C^1$};
        \draw (-4.75,1) -- (-4.75,4);
        \draw (-4.5,1) node {$F$};
        \filldraw[red] (-4.75,2.5) circle (2pt);
        \draw[red] (-4.45,2.5) node {$p$};

        \draw[->] (-2.5,2.5) -- (-3.5,2.5);
        \draw (-3,2.8) node {$\pi^1_1$};
        \draw[->] (-2.5,-.5) -- (-3.5,.5);
        \draw (-3,-.3) node {$\sigma_1$};

        \draw (-2,1.75) -- (1,1.75);
        \draw (1,1.45) node {$E_1$};
        \draw (-2,1) -- (1,4);
        \draw (1,4.3) node {$C^1_1$};
        \draw (-.5,1) -- (-.5,4);
        \draw (-.2,1) node {$F_1$};
        \draw[red] (-2,2.5) -- (1,2.5);
        \draw[red] (-2,2.8) node {$G^1_1$};
        \filldraw[blue] (-.5,2.5) circle (2pt);

        \draw[->] (2.5,2.5) -- (1.5,2.5);
        \draw (2,2.8) node {$\pi^1_2$};
        \draw[->] (2.5,.5) -- (1.5,-.5);
        \draw (2,.3) node {$\tau_1$};

        \draw (3,1.75) -- (6,1.75);
        \draw (6,1.45) node {$E_2$};
        \draw (3.75,1) -- (3.75,4);
        \draw (3.75,.7) node {$C^1_2$};
        \draw (5.25,1) -- (5.25,4);
        \draw (5.26,.7) node {$F_2$};
        \draw[red] (4.5,2.5) -- (4.5,4);
        \draw[red] (4.5,2.2) node {$G^1_2$};
        \draw[blue] (3,3.25) -- (6,3.25);
        \draw[blue] (3,3.55) node {$G^2_2$};

        \draw (-2,-2.25) -- (1,-2.25);
        \draw (1.25,-2.25) node {$\hat{E}$};
        \draw (-1.25,-3) -- (-1.25,0);
        \draw (-1,-3) node {$\hat{C}^1$};
        \draw (.25,-3) -- (.25,0);
        \draw (.5,-3) node {$\hat{F}$};
        \draw[blue] (-2,-.75)--(1,-.75);
        \draw[blue] (1.3,-.75) node {$\hat{G}$};
        \filldraw (-1.25,-.75) circle (2pt);
        \draw (-1.5,-.45) node {$q_{C^1}$};
        \filldraw[red] (-.5,-.75) circle (2pt);
        \draw[red] (-.5,-.45) node {$q_0$};
        \draw[red] (-.5,-1.2) node {$\frac{1}{2}(1,1)$};
        \filldraw (.25,-.75) circle (2pt);
        \draw (.5,-.45) node {$q_F$};
    \end{tikzpicture}
    \end{center}
    
    We then have the following pullbacks by $\sigma_1$:
\[
    \sigma_1^\ast E = \hat{E}, \quad
    \sigma_1^\ast F = \hat{F} + 2\hat{G}, \quad
    \sigma_1^\ast K_{S_1}  = K_{\hat{S}_1} - 2\hat{G}, \quad
    \sigma_1^\ast C^1 = \hat{C}^1 + 2\hat{G}.
\]
    
Then, the log discrepancy of $\hat{G}$ with respect to the pair $(S_1, (1 - \lambda)C^1)$ equals
\[
    A_{S_1, (1 - \lambda)C^1}(\hat{G}) = 1 + 2\lambda.
\]

The intersection numbers among $\hat{E}$, $\hat{F}$, and $\hat{G}$ are given by
\[
    \hat{E}^2 = -1, \quad
    \hat{F}^2 = -2, \quad
    \hat{G}^2 = -\tfrac{1}{2}, \quad
    \hat{E} \cdot \hat{F} = \hat{F} \cdot \hat{G} = 1, \quad
    \hat{E} \cdot \hat{G} = 0.
\]

The surface $\hat{S}_1$ is a Mori dream space, and its Mori cone is generated by the classes $[\hat{E}], [\hat{F}], [\hat{G}]$. We compute
\[
    \sigma_1^*\big(-K_{S_1} - (1 - \lambda)C^1\big) - t\hat{G} \equiv 2\lambda \hat{E} + 3\lambda \hat{F} + (6\lambda - t)\hat{G},
\]
which is pseudoeffective only for $t \leq 6\lambda$.
The Zariski decomposition of this divisor is
\begin{align*}
    P(t) =
    \begin{cases}
        2\lambda \hat{E} + 3\lambda \hat{F} + (6\lambda - t)\hat{G} \\
        2\lambda \hat{E} + \tfrac{8\lambda - t}{2} \hat{F} + (6\lambda - t)\hat{G} \\
        (6\lambda - t)(\hat{E} + \hat{F} + \hat{G})
    \end{cases};\quad
    N(t) =
    \begin{cases}
        0, & \ \ \ 0 \leq t \leq 2\lambda, \\
        \tfrac{t - 2\lambda}{2} \hat{F}, & \ \ \ 2\lambda \leq t \leq 4\lambda, \\
        (t - 4\lambda)\hat{E} + (t - 3\lambda)\hat{F}, & \ \ \ 4\lambda \leq t \leq 6\lambda.
    \end{cases}
\end{align*}

Consequently, the volume function is
\[
    \mathrm{vol}\left(\sigma_1^*(-K_{S_1} - (1 - \lambda)C^1) - t\hat{G}\right) = P(t)^2 =
    \begin{cases}
        8\lambda^2 - \tfrac{1}{2}t^2, & 0 \leq t \leq 2\lambda, \\
        10\lambda^2 - 2\lambda t, & 2\lambda \leq t \leq 4\lambda, \\
        \tfrac{(6\lambda - t)^2}{2}, & 4\lambda \leq t \leq 6\lambda.
    \end{cases}
\]

This implies that
\[
    S_{S_1, (1 - \lambda)C^1}(\hat{G}) = 3\lambda.
\]

By inequality~\eqref{1_F-E_ub}, we obtain the upper bound
\begin{equation} \label{1_C_tangent_F-E_ub}
    \delta_p(S_1, (1 - \lambda)C^1)
    \leq \min\left\{
        \frac{A_{S_1, (1 - \lambda)C^1}(F)}{S_{S_1, (1 - \lambda)C^1}(F)},
        \frac{A_{S_1, (1 - \lambda)C^1}(\hat{G})}{S_{S_1, (1 - \lambda)C^1}(\hat{G})}
    \right\}
    = \min\left\{ \frac{12}{13\lambda}, \frac{1 + 2\lambda}{3\lambda} \right\}.
\end{equation}

On the other hand, for each point $q$ on $\hat{G}$, the function $h(\hat{G}, q, t)$ is given by
\[
    h(\hat{G}, q, t) =
    \begin{cases}
        \tfrac{1}{18}t^2, & 0 \leq t \leq 2\lambda, \\
        \lambda \cdot \mathrm{ord}_q\left(\tfrac{t - 2\lambda}{2} q_F\right) + \tfrac{\lambda^2}{2}, & 2\lambda \leq t \leq 4\lambda, \\
        \tfrac{6\lambda - t}{2} \cdot \mathrm{ord}_q\left((t - 3\lambda)q_F\right) + \tfrac{(6\lambda - t)^2}{8}, & 4\lambda \leq t \leq 6\lambda,
    \end{cases}
\]
which implies that
\[
    S(W^{\hat{G}}_{\bullet,\bullet}; q) =
    \begin{cases}
        \tfrac{5\lambda}{12}, & q \neq q_F, \\
        \tfrac{13\lambda}{12}, & q = q_F.
    \end{cases}
\]

Let \( K_{\hat{G}} + \Delta_{\hat{G}} := (K_{\hat{S}_1} + (1 - \lambda)\hat{C} + \hat{G})|_{\hat{G}} \). Then the log discrepancy along $q$ satisfies
\[
    A_{\hat{G}, \Delta_{\hat{G}}}(q) =
    \begin{cases}
        1, & q \neq q_0, q_{C^1}, \\
        \lambda, & q = q_{C^1}, \\
        \tfrac{1}{2}, & q = q_0.
    \end{cases}
\]

Applying Theorem~\ref{AZ}, we deduce the lower bound
\begin{equation} \label{1_C_tangent_F-E_lb}
    \delta_p(S_1, (1 - \lambda)C^1) \geq
    \min\left\{
        \frac{1 + 2\lambda}{3\lambda}, \,
        \frac{12}{5\lambda}, \,
        \frac{12}{5}, \,
        \frac{12}{13\lambda}, \,
        \frac{6}{5\lambda}
    \right\}
    =
    \begin{cases}
        \tfrac{12}{13\lambda}, & \tfrac{23}{26} \leq \lambda \leq 1, \\
        \tfrac{1 + 2\lambda}{3\lambda}, & \tfrac{5}{26} \leq \lambda \leq \tfrac{23}{26}, \\
        \tfrac{48}{17}, & 0 < \lambda \leq \tfrac{5}{26}.
    \end{cases}
\end{equation}

Combining the upper bound~\eqref{1_C_tangent_F-E_ub} and the lower bound~\eqref{1_C_tangent_F-E_lb} completes the proof.
\end{proof}

\begin{lemma}\label{1_CE_notflex}
     Suppose that $p$ is the intersection point of $C^1$ and $E$, and that $C^1$ intersects  $F$ transversely. Then, $$\min\left\{\frac{6}{7\lambda},\frac{4+4\lambda}{9\lambda},\frac{48}{25}\right\}\leq\delta_p(S_1,(1-\lambda)C^1)\leq\min\left\{\frac{6}{7\lambda},\frac{4+4\lambda}{9\lambda}\right\}.$$ In particular, for $\lambda\geq\frac{25}{83}$, we have $$\delta_p(S_1,(1-\lambda)C^1)=\min\left\{\frac{6}{7\lambda},\frac{4+4\lambda}{9\lambda}\right\}=\begin{cases}
        \frac{6}{7\lambda}, &\frac{13}{14}\leq\lambda\leq1,\\
        \frac{4+4\lambda}{9\lambda}, &\frac{25}{83}\leq\lambda\leq\frac{13}{14}.
    \end{cases}$$
\end{lemma}

\begin{proof}
    Let $\sigma_1:\hat{S}_1\rightarrow (S_1,(1-\lambda)C^1)$ be the blowup at $p$ with the exceptional curve $\hat{G}$. This ordinary blowup is the desired plt blowup of $(S_1,(1-\lambda)C_1)$. This can be illustrated as follows:

    
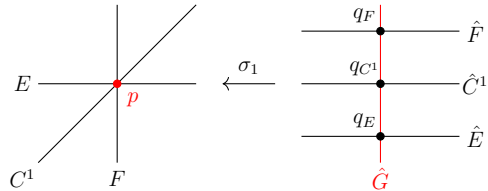
\begin{center}
    \begin{tikzpicture}[scale=0.7, every node/.style={scale=0.7}]
        \draw (-4,0) -- (-1,0);
        \draw (-4.3,0) node {$E$};
        \draw (-2.5,-1.5) -- (-2.5,1.5);
        \draw (-2.5,-1.8) node {$F$};
        \draw (-4,-1.5) -- (-1,1.5);
        \draw (-4.3,-1.8) node {$C^1$};
        \filldraw[red] (-2.5,0) circle (2pt);
        \draw[red] (-2.2,-.3) node {$p$};

        \draw[->] (0.5,0) -- (-.5,0);
        \draw (0,.3) node {$\sigma_1$};

        \draw (1,1) -- (4,1);
        \draw (4.3,1) node {$\hat{F}$};
        \draw (1,0) -- (4,0);
        \draw (4.3,0) node {$\hat{C}^1$};
        \draw (1,-1) -- (4,-1);
        \draw (4.3,-1) node {$\hat{E}$};
        \draw[red] (2.5,-1.5) -- (2.5,1.5);
        \draw[red] (2.5,-1.8) node {$\hat{G}$};
        \filldraw (2.5,1) circle (2pt);
        \draw (2.2,1.3) node {$q_F$};
        \filldraw (2.5,0) circle (2pt);
        \draw (2.2,.3) node {$q_{C^1}$};
        \filldraw (2.5,-1) circle (2pt);
        \draw (2.2,-.7) node {$q_E$};
    \end{tikzpicture}
    \end{center}

    The pullbacks of relevant divisors are
    \[
        \sigma_1^*(E) = \hat{E} + \hat{G}, \quad
        \sigma_1^*(F) = \hat{F} + \hat{G}, \quad
        \sigma_1^*(K_{S_1}) = K_{\hat{S}_1} - \hat{G}, \quad
        \sigma_1^*(C^1) = \hat{C}^1 + \hat{G},
    \]
    and so the log discrepancy is
    \[
        A_{S_1, (1 - \lambda)C^1}(\hat{G}) = 1 + \lambda.
    \]

    The intersections are
       \[
        \hat{E}^2 = -2, \quad
        \hat{F}^2 = \hat{G}^2 = -1, \quad
        \hat{E} \cdot \hat{F} = 0, \quad
        \hat{E} \cdot \hat{G} = \hat{F} \cdot \hat{G} = 1.
    \]

    Since $\hat{S}_1$ is a weak del Pezzo surface, it is a Mori dream space, and its Mori cone is spanned by
    $[\hat{E}]$, $[\hat{F}]$, and $[\hat{G}]$.
    We compute
   \[
        \sigma_1^*\left(-K_{S_1} - (1 - \lambda)C^1\right) - t\hat{G}
        \equiv 2\lambda \hat{E} + 3\lambda \hat{F} + (5\lambda - t)\hat{G},
    \]
    which is pseudoeffective only for $t \leq 5\lambda$.
    The Zariski decomposition is given by
  \begin{align*}
        P(t) = \begin{cases}
            2\lambda \hat{E} + 3\lambda \hat{F} + (5\lambda - t)\hat{G} \\
            \frac{5\lambda - t}{2}\hat{E} + 3\lambda \hat{F} + (5\lambda - t)\hat{G} \\
            \frac{5\lambda - t}{2}(\hat{E} + 2\hat{F} + 2\hat{G})
        \end{cases};\quad
        N(t) = \begin{cases}
            0, & \ \ \ 0 \leq t \leq \lambda, \\
            \frac{t - \lambda}{2} \hat{E}, & \ \ \ \lambda \leq t \leq 2\lambda, \\
            \frac{t - \lambda}{2} \hat{E} + (t - 2\lambda) \hat{F}, & \ \ \ 2\lambda \leq t \leq 5\lambda.
        \end{cases}
    \end{align*}
    The volume function is
    \[
        \mathrm{vol}\left(\sigma_1^*(-K_{S_1} - (1 - \lambda)C^1) - t\hat{G}\right) = P(t)^2 =
        \begin{cases}
            8\lambda^2 - t^2, & 0 \leq t \leq \lambda, \\
            -\tfrac{1}{2}t^2 - \lambda t + \tfrac{17}{2}\lambda^2, & \lambda \leq t \leq 2\lambda, \\
            \tfrac{(6\lambda - t)^2}{2}, & 2\lambda \leq t \leq 5\lambda.
        \end{cases}
    \]
    Hence,
    \[
        S_{S_1, (1 - \lambda)C^1}(\hat{G}) = \frac{9\lambda}{4}.
    \]

    Combining inequality~\eqref{1_E-C_ub}, we get the upper bound
    \begin{equation} \label{1_CE_notflex_ub}
        \delta_p(S_1, (1 - \lambda)C^1)
        \leq \min\left\{
            \frac{A_{S_1, (1 - \lambda)C^1}(E)}{S_{S_1, (1 - \lambda)C^1}(E)},
            \frac{A_{S_1, (1 - \lambda)C^1}(\hat{G})}{S_{S_1, (1 - \lambda)C^1}(\hat{G})}
        \right\}
        = \min\left\{\frac{6}{7\lambda}, \frac{4 + 4\lambda}{9\lambda} \right\}.
    \end{equation}

    For each point $q$ on $\hat{G}$, the function $h(\hat{G}, q, t)$ is
    \[
        h(\hat{G}, q, t) =
        \begin{cases}
            \tfrac{1}{2}t^2, & 0 \leq t \leq \lambda, \\
            \tfrac{t + \lambda}{2} \cdot \mathrm{ord}_q\left(\tfrac{t - \lambda}{2} q_E\right) + \tfrac{(t + \lambda)^2}{8}, & \lambda \leq t \leq 2\lambda, \\
            \tfrac{5\lambda - t}{2} \cdot \mathrm{ord}_q\left(\tfrac{t - 3\lambda}{2} q_E + (t - 2\lambda) q_F\right) + \tfrac{(5\lambda - t)^2}{8}, & 2\lambda \leq t \leq 5\lambda.
        \end{cases}
    \]

    This implies that
    \[
        S(W^{\hat{G}}_{\bullet,\bullet}; q) =
        \begin{cases}
            \frac{25\lambda}{48}, & q \ne q_E, q_F, \\
            \frac{13\lambda}{12}, & q = q_F, \\
            \frac{7\lambda}{6}, & q = q_E.
        \end{cases}
    \]

    Setting \( K_{\hat{G}} + \Delta_{\hat{G}} := (K_{\hat{S}_1} + (1 - \lambda)\hat{C} + \hat{G})|_{\hat{G}} \), we find
    \[
        A_{\hat{G}, \Delta_{\hat{G}}}(q) =
        \begin{cases}
            1, & q \ne q_{C^1}, \\
            \lambda, & q = q_{C^1}.
        \end{cases}
    \]

    Then by Theorem~\ref{AZ}, we obtain the lower bound
    \begin{equation} \label{1_CE_notflex_lb}
        \delta_p(S_1, (1 - \lambda)C^1) \geq
        \min\left\{
            \frac{4 + 4\lambda}{9\lambda},\,
            \frac{48}{25\lambda},\,
            \frac{48}{25},\,
            \frac{12}{13\lambda},\,
            \frac{6}{7\lambda}
        \right\}
        = \begin{cases}
            \frac{6}{7\lambda}, & \frac{13}{14} \leq \lambda \leq 1, \\
            \frac{4 + 4\lambda}{9\lambda}, & \frac{25}{83} \leq \lambda \leq \frac{13}{14}, \\
            \frac{48}{25}, & 0 < \lambda \leq \frac{25}{83}.
        \end{cases}
    \end{equation}

    The result follows by combining the bounds \eqref{1_CE_notflex_ub} and \eqref{1_CE_notflex_lb}.
\end{proof}

\begin{lemma}\label{1_CE_flex}
Suppose that  $p$ is the intersection point of $C^1$ and $E$, and  that $C^1$ is tangent to $F$.
Then,
\[
\min\left\{ \frac{6}{7\lambda}, \frac{3 + 6\lambda}{10\lambda}, 3 \right\} \leq \delta_p(S_1, (1 - \lambda)C^1) \leq \min\left\{ \frac{6}{7\lambda}, \frac{3 + 6\lambda}{10\lambda} \right\}.
\]
In particular, for $\lambda \geq \frac{1}{8}$, we have
\[
\delta_p(S_1, (1 - \lambda)C^1) = \min\left\{ \frac{6}{7\lambda}, \frac{3 + 6\lambda}{10\lambda} \right\} = 
\begin{cases}
\frac{6}{7\lambda}, & \frac{13}{14} \leq \lambda \leq 1, \\
\frac{3 + 6\lambda}{10\lambda}, & \frac{1}{8} \leq \lambda \leq \frac{13}{14}.
\end{cases}
\]
\end{lemma}

\begin{proof}
    Define $\sigma_1:\hat{S}_1\rightarrow(S_1,(1-\lambda)C^1)$ as the $(1,2)$-blowup at $p$ with respect to $F$. We see that $q_F$ is an $\mathrm{A}_1$ singularity.
The construction of $\hat{S}_1$ is illustrated below:

    
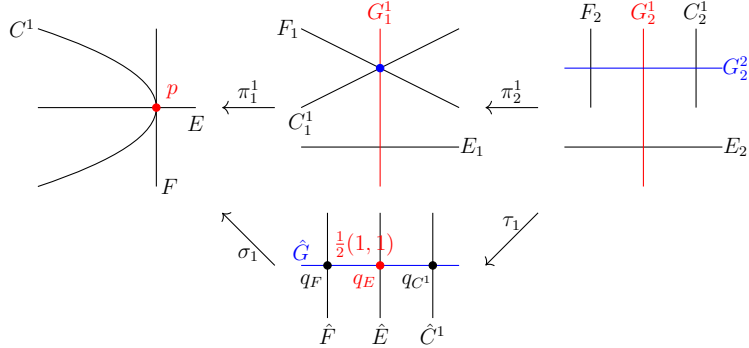
\begin{center}
    \begin{tikzpicture}[scale=0.7, every node/.style={scale=0.7}]
        \draw (-7,2.5) -- (-4,2.5);
        \draw (-4,2.2) node {$E$};
        \draw (-7,1) .. controls (-4,2) and (-4,3) .. (-7,4);
        \draw (-7.3,4) node {$C^1$};
        \draw (-4.75,1) -- (-4.75,4);
        \draw (-4.5,1) node {$F$};
        \filldraw[red] (-4.75,2.5) circle (2pt);
        \draw[red] (-4.45,2.8) node {$p$};

        \draw[->] (-2.5,2.5) -- (-3.5,2.5);
        \draw (-3,2.8) node {$\pi^1_1$};
        \draw[->] (-2.5,-.5) -- (-3.5,.5);
        \draw (-3,-.3) node {$\sigma_1$};

        \draw (-2,1.75) -- (1,1.75);
        \draw (1.25,1.75) node {$E_1$};
        \draw (-2,4) -- (1,2.5);
        \draw (-2.25,4) node {$F_1$};
        \draw (-2,2.5) -- (1,4);
        \draw (-2,2.2) node {$C^1_1$};
        \draw[red] (-.5,1) -- (-.5,4);
        \draw[red] (-.5,4.3) node {$G^1_1$};
        \filldraw[blue] (-.5,3.25) circle (2pt);

        \draw[->] (2.5,2.5) -- (1.5,2.5);
        \draw (2,2.8) node {$\pi^1_2$};
        \draw[->] (2.5,.5) -- (1.5,-.5);
        \draw (2,.3) node {$\tau_1$};

        \draw (3,1.75) -- (6,1.75);
        \draw (6.25,1.75) node {$E_2$};
        \draw (3.5,2.5) -- (3.5,4);
        \draw (3.5,4.3) node {$F_2$};
        \draw (5.5,2.5) -- (5.5,4);
        \draw (5.5,4.3) node {$C^1_2$};
        \draw[red] (4.5,1) -- (4.5,4);
        \draw[red] (4.5,4.3) node {$G^1_2$};
        \draw[blue] (3,3.25) -- (6,3.25);
        \draw[blue] (6.25,3.25) node {$G^2_2$};

        \draw (-.5,-1.5) -- (-.5,.5);
        \draw (-.5,-1.8) node {$\hat{E}$};
        \draw (-1.5,-1.5) -- (-1.5,.5);
        \draw (-1.5,-1.8) node {$\hat{F}$};
        \draw (.5,-1.5) -- (.5,.5);
        \draw (.5,-1.8) node {$\hat{C}^1$};
        \draw[blue] (-2,-.5) -- (1,-.5);
        \draw[blue] (-2,-.2) node {$\hat{G}$};
        \filldraw (-1.5,-.5) circle (2pt);
        \draw (-1.8,-.8) node {$q_F$};
        \filldraw[red] (-.5,-.5) circle (2pt);
        \draw[red] (-.8,-.1) node {$\frac{1}{2}(1,1)$};
        \draw[red] (-.8,-.8) node {$q_E$};
        \filldraw (.5,-.5) circle (2pt);
        \draw (.2,-.8) node {$q_{C^1}$};
    \end{tikzpicture}
    \end{center}

We compute the pullbacks and the log discrepancy as follows:
\[
\sigma_1^*(E) = \hat{E} + \hat{G}, 
\sigma_1^*(F) = \hat{F} + 2\hat{G}, 
\sigma_1^*(C^1) = \hat{C}^1 + 2\hat{G}, 
\sigma_1^*(K_{S_1}) = K_{\hat{S}_1} - 2\hat{G}, \]
\[A_{S_1, (1 - \lambda)C^1}(\hat{G}) = 1 + 2\lambda.
\]

The intersection numbers on $\hat{S}_1$ are
\[
\hat{E}^2 = -\frac{3}{2}, \quad \hat{F}^2 = -2, \quad \hat{G}^2 = -\frac{1}{2}, \quad \hat{E} \cdot \hat{F} = 0, \quad \hat{E} \cdot \hat{G} = \frac{1}{2}, \quad \hat{F} \cdot \hat{G} = 1.
\]
 
Since $T_2$ is a weak del Pezzo surface, $\hat{S}_1$ is a Mori dream space. By Proposition~\ref{MDS}, the Mori cone is generated by 
$[\hat{E}]$, $[\hat{F}]$, $[\hat{G}]$.
We compute
\[
\sigma_1^*\left(-K_{S_1} - (1 - \lambda)C^1\right) - t\hat{G} \equiv 2\lambda\hat{E} + 3\lambda\hat{F} + (8\lambda - t)\hat{G},
\]
which is pseudoeffective for $t \leq 8\lambda$. The Zariski decomposition is given by
\[
\begin{aligned}
P(t) = 
\begin{cases}
2\lambda\hat{E} + 3\lambda\hat{F} + (8\lambda - t)\hat{G} \\
\frac{8\lambda - t}{6}(2\hat{E} + 3\hat{F} + 6\hat{G})
\end{cases};\quad
N(t) = 
\begin{cases}
0 & \ \ \ 0 \leq t \leq 2\lambda, \\
\frac{t - 2\lambda}{3}\hat{E} + \frac{t - 2\lambda}{2}\hat{F} & \ \ \ 2\lambda \leq t \leq 6\lambda.
\end{cases}
\end{aligned}
\]
The volume function becomes
\[
\mathrm{vol}\left(\sigma_1^*(-K_{S_1} - (1 - \lambda)C^1) - t\hat{G}\right) = P(t)^2 = 
\begin{cases}
8\lambda^2 - \frac{1}{2}t^2, & 0 \leq t \leq 2\lambda, \\
\frac{(8\lambda - t)^2}{6}, & 2\lambda \leq t \leq 6\lambda.
\end{cases}
\]
From this we compute
\[
S_{S_1, (1 - \lambda)C^1}(\hat{G}) = \frac{10\lambda}{3},
\]
and hence
\begin{equation}\label{1_CE_flex_ub}
\delta_p(S_1, (1 - \lambda)C^1) \leq \min\left\{ \frac{A_{S_1,(1-\lambda),C^1}(E)}{S_{S_1,(1-\lambda),C^1}(E)}, \frac{A_{S_1,(1-\lambda),C^1}(\hat{G})}{S_{S_1,(1-\lambda),C^1}(\hat{G})} \right\} = \min\left\{ \frac{6}{7\lambda}, \frac{3 + 6\lambda}{10\lambda} \right\}.
\end{equation}

To bound from below, consider $h(\hat{G}, q, t)$ for each $q \in \hat{G}$ that is given by
\[
h(\hat{G}, q, t) =
\begin{cases}
\frac{1}{8}t^2, & 0 \leq t \leq 2\lambda, \\
\frac{8\lambda - t}{6}\cdot \mathrm{ord}_q\left( \frac{t - 2\lambda}{6}q_E + \frac{t - 2\lambda}{2}q_F \right) + \frac{(8\lambda - t)^2}{72}, & 2\lambda \leq t \leq 6\lambda.
\end{cases}
\]

It follows that
\[
S(W^{\hat{G}}_{\bullet,\bullet}; q) = 
\begin{cases}
\frac{\lambda}{3}, & q \neq q_E, q_F, \\
\frac{7\lambda}{12}, & q = q_E, \\
\frac{13\lambda}{12}, & q = q_F.
\end{cases}
\]

Let $K_{\hat{G}} + \Delta_{\hat{G}} := (K_{\hat{S}_1} + (1 - \lambda)\hat{C} + \hat{G})|_{\hat{G}}$. Then,
\[
A_{\hat{G}, \Delta_{\hat{G}}}(q) =
\begin{cases}
1, & q \neq q_E, q_{C^1}, \\
\lambda, & q = q_{C^1}, \\
\frac{1}{2}, & q = q_E.
\end{cases}
\]

By Theorem~\ref{AZ}, we then obtain the lower bound
\begin{equation}\label{1_CE_flex_lb}
\delta_p(S_1, (1 - \lambda)C^1) \geq \min\left\{ \frac{3 + 5\lambda}{10\lambda}, \frac{3}{\lambda}, 3, \frac{12}{13\lambda}, \frac{6}{7\lambda} \right\} = 
\begin{cases}
\frac{6}{7\lambda}, & \frac{13}{14} \leq \lambda \leq 1, \\
\frac{3 + 6\lambda}{10\lambda}, & \frac{1}{8} \leq \lambda \leq \frac{13}{14}, \\
\frac{48}{17}, & 0 < \lambda \leq \frac{5}{26}.
\end{cases}
\end{equation}
Combining \eqref{1_CE_flex_ub} and \eqref{1_CE_flex_lb} concludes the proof.
\end{proof}

Note that $\phi_1(C^1)\setminus \{x\}$ contains at least eight inflection points,
and
that  at least three $0$-curves in $S_1$ are tangent to $C^1$ outside $E$. 
Thus, Lemmas~\ref{1_S-C}, \ref{1_C-E_notflex}, \ref{1_C-E_flex}, and \ref{1_CF_tangent} give
\begin{equation}\label{1-semitotal}
    \inf_{p\in S_1 \setminus (C^1\cap E)}\delta_p(S_1,(1-\lambda)C^1)\begin{cases}
            =\frac{6}{7\lambda}, &\frac{11}{14}\leq\lambda\leq 1,\\
            =\frac{1+2\lambda}{3\lambda}, &\frac{7}{22}\leq\lambda\leq\frac{11}{14},\\
            =\frac{12+36\lambda}{43\lambda}, &\frac{25}{97}\leq\lambda\leq\frac{7}{22},\\
            \geq\frac{48}{25}, &0<\lambda\leq\frac{25}{97}.
    \end{cases}
\end{equation}
We now distinguish two cases depending on whether $x$ is an inflection point of $\phi_1(C^1)$.
 If $x$ is an inflection point, then  by applying \eqref{local delta} we obtain 
\[
    \delta(S_1,(1-\lambda)C^1)\begin{cases}
        =\frac{6}{7\lambda}, &\frac{13}{14}\leq\lambda\leq1,\\
        =\frac{3+6\lambda}{10\lambda}, &\frac{5}{22}\leq\lambda\leq\frac{13}{14},\\
        \geq\frac{48}{25}, &0<\lambda\leq\frac{5}{22}
    \end{cases}
\]
from Lemma~\ref{1_CE_flex} and \eqref{1-semitotal}.
Thus, the first part of Theorem~\ref{thm1} follows directly. 
If $x$ is not an inflection point, then it follows from Lemma~\ref{1_CE_notflex} and \eqref{1-semitotal} that
\[
    \delta(S_1,(1-\lambda)C^1)\begin{cases}
        =\frac{6}{7\lambda}, &\frac{13}{14}\leq\lambda\leq1,\\
        =\frac{4+4\lambda}{9\lambda}, &\frac{1}{2}\leq\lambda\leq\frac{13}{14},\\
        =\frac{1+2\lambda}{3\lambda}, &\frac{7}{22}\leq\lambda\leq\frac{1}{2},\\
        =\frac{12+36\lambda}{43\lambda}, &\frac{25}{97}\leq\lambda\leq\frac{7}{22},\\
        \geq\frac{48}{25}, &0<\lambda\leq\frac{25}{97}.
    \end{cases}
\]
This establishes the second part of Theorem~\ref{thm1}.

\subsection{Proof of Theorem~\ref{thm2}} We now consider the surface $S_2$ and a smooth anticanonical curve $C^2$ on $S_2$. We now let $p$ be a point in $S_2$.
We always denote by $B$ the strict transform of the line determined by $x_1$ and $x_2$ on $\mathbb{P}^2$ via $\phi_2$. The surface $S_2$ has only three $(-1)$-curves $A^1$, $A^2$, and $B$.

The pencil $|\phi_2^\ast\mathcal{O}_{\mathbb{P}^2}(1)-A^i|$ for each $i$ is  base point free. There is a unique divisor in the pencil passing through the point $p$ in $S_2$, which will be denoted by $N^i$. Note that $N^i$ is a smooth curve if $p$ is not contained in $A^{3-i}\cup B$.

Consider the $(-1)$-curve $A^1$. 
We have
    \[
        -K_{S_2}-(1-\lambda)C^2-tA^1\equiv(2\lambda-t)A^1+2\lambda A^2+3\lambda B
    \]
    and it is pseudoeffective only for $t\leq 2\lambda$. Its Zariski decomposition is given by
     \begin{align*}
        P(t)=\begin{cases}
            (2\lambda-t)A^1+2\lambda A^2+3\lambda B\\
            (2\lambda-t)A^1+2\lambda A^2+(4\lambda-t)B
        \end{cases};\quad
        N(t)=\begin{cases}
            0, &\ \ \ 0\leq t\leq\lambda,\\
            (t-\lambda)B, &\ \ \ \lambda\leq t\leq2\lambda.
        \end{cases}
    \end{align*}
    Then,
    \[
        \mathrm{vol}\left(-K_{S_2}-(1-\lambda)C^2-tA^1\right)=P(t)^2=\begin{cases}
            7\lambda^2-2\lambda t-t^2, &0\leq t\leq\lambda,\\
            8\lambda^2-4\lambda t, &\lambda\leq t\leq2\lambda,
        \end{cases}
    \]
    and hence,
    \[
        S_{S_2,(1-\lambda)C^2}(A^1)=\frac{23\lambda}{21}.
    \]
    Therefore, if $p$ is on $A^1$, then 
    \begin{equation}\label{2_Ei-CF_ub}
        \delta_p(S_2,(1-\lambda)C^2)\leq\frac{21}{23\lambda}.
    \end{equation}
    By the same computation, we also obtain the same upper bound for points in $A^2$.

    We now consider the $(-1)$-curve $B$. We have
    \[
        -K_{S_2}-(1-\lambda)C^2-tB\equiv2\lambda A^1+2\lambda A^2+(3\lambda-t)B
    \]
    and it is pseudoeffective only for $t\leq3\lambda$. The Zariski decomposition of the divisor is as follows:
       \begin{align*}
        P(t)=\begin{cases}
            2\lambda A^1+2\lambda A^2+(3\lambda-t)B \\
            (3\lambda-t)(A^1+A^2+B)
        \end{cases};\quad
        N(t)&=\begin{cases}
            0 &\ \ \ 0\leq t\leq\lambda,\\
            (t-\lambda)(A^1+A^2) &\ \ \ \lambda\leq t\leq3\lambda.
        \end{cases}
    \end{align*}
    Then,
    \[
        \mathrm{vol}\left(-K_{S_2}-(1-\lambda)C^2-tB\right)=P(t)^2=\begin{cases}
            7\lambda^2-2\lambda t-t^2, &0\leq t\leq\lambda,\\
            (3\lambda-t)^2, &\lambda\leq t\leq3\lambda,
        \end{cases}
    \]
    and hence,
    \[
        S_{S_2,(1-\lambda)C^2}(B)=\frac{25\lambda}{21}.
    \]
Therefore, if $p$ is on $B$, then 
    \begin{equation}\label{2_F-C_ub}
        \delta_p(S_2,(1-\lambda)C^2)\leq\frac{21}{25\lambda}.
    \end{equation}

\begin{lemma}\label{2_S-C}
    Suppose that $p$ is in $S_2\setminus C^2$. Then
    \[
        \delta_p(S_2,(1-\lambda)C^2)\left\{\aligned
        &\geq\frac{21}{23\lambda}, &p\not\in A^1\cup A^2\cup B,\\
        &=\frac{21}{23\lambda}, &p\in A^1\cup A^2\setminus B,\\
        &=\frac{21}{25\lambda}, &p\in B.\endaligned\right.
    \]
\end{lemma}

\begin{proof}
    Suppose that $p$ is not contained in $A^1\cup A^2\cup B\cup C^2$. Take a smooth curve $L$ in $|A^1+A^2+B|$ that contains $p$. We have
    \[
        -K_{S_2}-(1-\lambda)C^2-tL\equiv(2\lambda-t)A^1+(2\lambda-t)A^2+(3\lambda-t)B
    \]
    and it is pseudoeffective only for $t\leq 2\lambda$. Its Zariski decomposition is
         \begin{align*}
        P(t)&=\begin{cases}
            (2\lambda-t)A^1+(2\lambda-t)A^2+(3\lambda-t)B\\
            (2\lambda-t)(A^1+A^2+2B)
        \end{cases};\quad
        N(t)&=\begin{cases}
            0, & \ \ \ 0\leq t\leq\lambda,\\
            (t-\lambda)B, &\ \ \ \lambda\leq t\leq2\lambda.
        \end{cases}
    \end{align*}
    Thus, the volume function is given by
    \[
        \mathrm{vol}\left(-K_{S_2}-(1-\lambda)C^2-tL\right)=P(t)^2=\begin{cases}
            t^2-6\lambda t+7\lambda^2, &0\leq t\leq\lambda,\\
            2(2\lambda-t)^2, &\lambda\leq t\leq2\lambda,
        \end{cases}
    \]
    and hence,
    \[
        S_{S_2,(1-\lambda)C^2}(L)=\frac{5\lambda}{7}.
    \]

    Since $p$ does not lie on $B$, we have
    \[
        h(L,p,t)=\begin{cases}
            \frac{(3\lambda-t)^2}{2}, &0\leq t\leq\lambda,\\
            2(2\lambda-t)^2, &\lambda\leq t\leq 2\lambda,
        \end{cases}
    \]
    and hence,
    \[
        S(W^L_{\bullet,\bullet};p)=\frac{23\lambda}{21}.
    \]
    Put $K_L+\Delta_L:=(K_{S_2}+(1-\lambda)C^2+L)|_L$, then $A_{L,\Delta_L}(p)=1$ since $p$ is outside $C^2$. It then follows from Theorem~\ref{AZ} that
    \[
        \delta_p(S_2,(1-\lambda)C^2)\geq\min\left\{\frac{7}{5\lambda},\frac{21}{23\lambda}\right\}=\frac{21}{23}\lambda,
    \]
    and this completes the proof when $p$ is not in $A^1\cup A^2\cup B$.
    
    Next, suppose that $p$ is on $A^1\setminus B$, the integrand in \eqref{eq:SS} is
    \[
        h(A^1,p,t)=\begin{cases}
            \frac{(t+\lambda)^2}{2}, &0\leq t\leq\lambda,\\
            2\lambda^2, &\lambda\leq t\leq2\lambda,
        \end{cases}
    \]
    and hence,
    \[
        S(W^{A^1}_{\bullet,\bullet};p)=\frac{19\lambda}{21}.
    \]
    Put $K_{A^1}+\Delta_{A^1}:=(K_{S_2}+(1-\lambda)C^2+A^1)|_{A^1}$, then $A_{A^1,\Delta_{A^1}}(p)=1$ because $p$ is not in~$C^2$. Theorem~\ref{AZ} then implies that
    \begin{equation}\label{2_Ei-CF_lb}
        \delta_p(S_2,(1-\lambda)C^2)\geq\min\left\{\frac{21}{23\lambda},\frac{21}{19\lambda}\right\}=\frac{21}{23\lambda}.   
    \end{equation}
    Consequently, \eqref{2_Ei-CF_ub} and \eqref{2_Ei-CF_lb} complete the proof for the case when $p$ is on $A^1\setminus B$. The same holds when $p$ is on $A^2\setminus B$.

    Finally, suppose that $p$ is on $B\setminus C^2$. We compute
    \[
        h(B,p,t)=\begin{cases}
            \frac{(\lambda+t)^2}{2}, &0\leq t\leq\lambda,\\
            (3\lambda-t)\cdot\mathrm{ord}_p(t-\lambda)(p_1+p_2)+\frac{(3\lambda-t)^2}{2}, &\lambda\leq t\leq 3\lambda,
        \end{cases}
    \]
    which implies that
    \[
        S(W^B_{\bullet,\bullet};p)=\begin{cases}
            \frac{5\lambda}{7}, &p\neq p_1,p_2,\\
            \frac{19\lambda}{21}, &p=p_1, p_2.
        \end{cases}
    \]
    Here,  $p_1$ (resp. $p_2$) is the intersection point of $B$ and $A^1$ (resp. $A^2$).
    
    Put $K_B+\Delta_B:=(K_{S_2}+(1-\lambda)C^2+B)|_B$, then $A_{B,\Delta_B}(p)=1$ because $p$ is  not in $C^2$. Theorem~\ref{AZ} then implies that
    \begin{equation}\label{2_F-C_lb}
        \delta_p(S_2,(1-\lambda)C^2)\geq\left\{\aligned
            &\min\left\{\frac{21}{25\lambda},\frac{7}{5\lambda}\right\}, &p\neq p_1,p_2,\\
            &\min\left\{\frac{21}{25\lambda},\frac{19}{21\lambda}\right\}, &p=p_1,p_2.
        \endaligned\right\}=\frac{21}{25\lambda}.
    \end{equation}
    Consequently, \eqref{2_Ei-CF_ub} and \eqref{2_F-C_lb} complete the proof for the case when $p$ is on $B$.
\end{proof}

\begin{lemma}\label{2_C-AB_notflex}
    Suppose that $p$ is a point in $C^2\setminus A^1\cup A^2\cup B$ such that each $N^i$ is transverse to $C^2$ at $p$ and $\phi_2(p)$ is not an inflection point of the smooth cubic curve $\phi_2(C^2)$. Then, $$\frac{21+42\lambda}{53\lambda}\geq\delta_p(S_2,(1-\lambda)C^2)\geq\min\left\{\frac{21}{23\lambda},\frac{21+42\lambda}{53\lambda},\frac{42}{23}\right\}=\begin{cases}
        \frac{21}{23\lambda}, &\frac{15}{23}\leq\lambda\leq 1,\\
        \frac{21+42\lambda}{53\lambda}, &\frac{23}{63}\leq\lambda\leq\frac{15}{23},\\
        \frac{42}{23}, &0<\lambda\leq\frac{23}{63}.
    \end{cases}$$
\end{lemma}

\begin{proof}
    Let $L$ be the strict transform of the tangent line of $\phi_2(C^2)$ at $\phi_2(p)$. By the assumption, it belongs to the linear system $|A^1+A^2+B|$.
    Let $\sigma_2:\hat{S}_2\rightarrow (S_2,(1-\lambda)C^2)$ be the $(1,2)$-blowup with respect to $L$. Note that $q_{N^1}=q_{N^2}$ and it is an $\mathrm{A}_1$ singularity. We denote this point by $q_N$. This can be illustrated as follows:

    
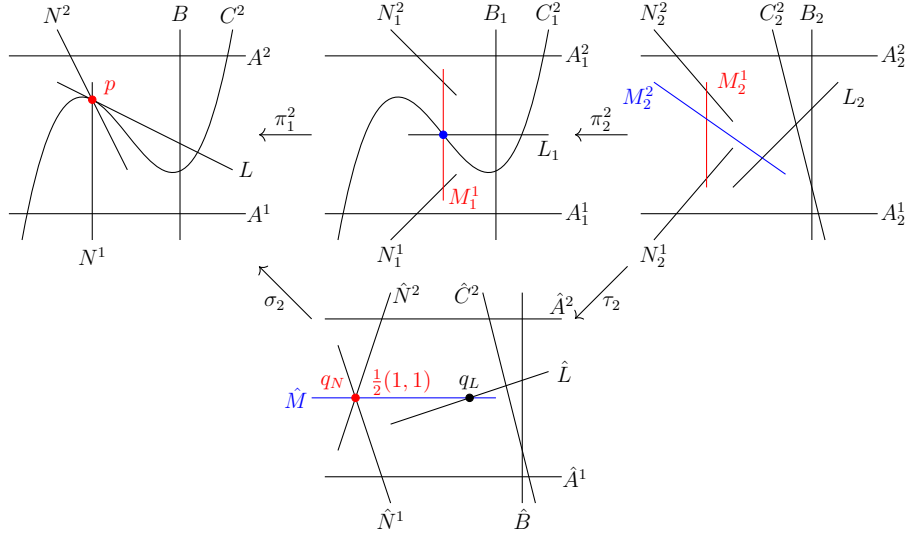
\begin{center}
    \begin{tikzpicture}[scale=0.7, every node/.style={scale=0.7}]
        \draw (-8.25,2.5) -- (-3.75,2.5);
        \draw (-3.5,2.5) node {$A^1$};
        \draw (-8.25,5.5) -- (-3.75,5.5);
        \draw (-3.5,5.5) node {$A^2$};
        \draw (-5,2) -- (-5,6);
        \draw (-5,6.3) node {$B$};
        \draw (-20/3,2) -- (-20/3,5);
        \draw (-20/3,1.7) node {$N^1$};
        \draw (-22/3,6) -- (-6,10/3);
        \draw (-22/3,6.3) node {$N^2$};
        \draw (-8,2) .. controls (-20/3,28/3) and (-16/3,-4/3) .. (-4,6);
        \draw (-4,6.3) node {$C^2$};
        \draw (-22/3,5) -- (-4,10/3);
        \draw (-3.75,10/3) node {$L$};
        \filldraw[red] (-20/3,14/3) circle (2pt);
        \draw[red] (-19/3,4.9) node {$p$};

        \draw[->] (-2.5,4) --(-3.5,4);
        \draw (-3,4.3) node {$\pi^2_1$};
        \draw[->] (-2.5,.5) -- (-3.5,1.5);
        \draw (-3.2,.8) node {$\sigma_2$};

        \draw (-2.25,2.5) -- (2.25,2.5);
        \draw (2.55,2.5) node {$A^1_1$};
        \draw (-2.25,5.5) -- (2.25,5.5);
        \draw (2.55,5.5) node {$A^2_1$};
        \draw (1,2) -- (1,6);
        \draw (1,6.3) node {$B_1$};
        \draw (-1,2) -- (.25,3.25);
        \draw (-1,1.7) node {$N^1_1$};
        \draw (-1,6) -- (.25,4.75);
        \draw (-1,6.3) node {$N^2_1$};
        \draw (-2,2) .. controls (-2/3,28/3) and (2/3,-4/3) .. (2,6);
        \draw (2,6.3) node {$C^2_1$};
        \draw (-2/3,4) -- (2,4);
        \draw (2,3.7) node {$L_1$};
        \draw[red] (0,2.75) -- (0,5.25);
        \draw[red] (.4,2.8) node {$M^1_1$};
        \filldraw[blue] (0,4) circle (2pt);

        \draw[->] (3.5,4) -- (2.5,4);
        \draw (3,4.3) node {$\pi^2_2$};
        \draw[->] (3.5,1.5) -- (2.5,.5);
        \draw (3.2,.8) node {$\tau_2$};

        \draw (3.75,2.5) -- (8.25,2.5);
        \draw (8.55,2.5) node {$A^1_2$};
        \draw (3.75,5.5) -- (8.25,5.5);
        \draw (8.55,5.5) node {$A^2_2$};
        \draw (7,2) -- (7,6);
        \draw (7,6.3) node {$B_2$};
        \draw (7.25,2) -- (6.25,6);
        \draw (6.25,6.3) node {$C^2_2$};
        \draw (7.5,5) -- (5.5,3);
        \draw (7.8,4.7) node {$L_2$};
        \draw (5.5,3.75) -- (4,2);
        \draw (4,1.7) node {$N^1_2$};
        \draw (5.5,4.25) -- (4,6);
        \draw (4,6.3) node {$N^2_2$};
        \draw[red] (5,5) -- (5,3);
        \draw[red] (5.5,5) node {$M^1_2$};
        \draw[blue] (6.5,3.25) -- (4,5);
        \draw[blue] (3.7,4.7) node {$M^2_2$};

        \draw (-2.25,-2.5) -- (2.25,-2.5);
        \draw (2.5,.-2.5) node {$\hat{A}^1$};  
        \draw (-2.25,.5) -- (2.25,.5);
        \draw (2.25,.8) node {$\hat{A}^2$};
        \draw (1.5,-3) -- (1.5,1);
        \draw (1.5,-3.3) node {$\hat{B}$};
        \draw (1.75,-3) -- (.75,1);
        \draw (.45,1) node {$\hat{C}^2$};
        \draw (-1,-3) -- (-2,0);
        \draw (-1,-3.3) node {$\hat{N}^1$};
        \draw (-1,1) -- (-2,-2);
        \draw (-.7,1) node {$\hat{N}^2$};
        \draw (-1,-1.5) -- (2,-.5);
        \draw (2.3,-.5) node {$\hat{L}$};
        \draw[blue] (-2.5,-1) -- (1,-1);
        \draw[blue] (-2.8,-1) node {$\hat{M}$};
        \filldraw[red] (-5/3,-1) circle (2pt);
        \draw[red] (-2.1,-.7) node {$q_N$};
        \draw[red] (-.8,-.7) node {$\frac{1}{2}(1,1)$};
        \filldraw (.5,-1) circle (2pt);
        \draw (.5,-.7) node {$q_L$};
    \end{tikzpicture}
    \end{center}

    We then obtain the following numerical equivalences
    \[
        \hat{N}^1\equiv\hat{A}^2+\hat{B}-\hat{M},\quad
        \hat{N}^2\equiv\hat{A}^1+\hat{B}-\hat{M},\quad
        \hat{L}\equiv\hat{A}^1+\hat{A}^2+\hat{B}-2\hat{M}.
    \]
    The pullbacks and the log discrepancy are given by
    \[
    \begin{split}
        &\sigma_2^\ast A^i=\hat{A}^i,\quad\sigma_2^\ast B=\hat{B}, \quad\sigma_2^\ast L=\hat{L}+2\hat{M},\sigma_2^\ast N^i=\hat{N}^i+\hat{M},\\
        &\sigma_2^\ast K_{S_2}=K_{\hat{S}_2}-2\hat{M},\quad\sigma_2^\ast C^2=\hat{C}^2+2\hat{M},\quad\text{for }i=1,2,
    \end{split}
    \]
    \[
        A_{S_2,(1-\lambda)C^2}(\hat{M})=1+2\lambda.
    \]

    We have the intersections on $\hat{S}_2$ as follows:
    \begin{longtable}{cccccccc}
            & $\hat{A}^1$ & $\hat{A}^2$ & $\hat{B}$ & $\hat{N}^1$    & $\hat{N}^2$    & $\hat{L}$      & $\hat{M}$ \\ 
            $\hat{A}^1$ & $-1$        & $0$         & $1$       & $1$            & $0$            & $0$            & $0$            \\
            $\hat{A}^2$ & $0$         & $-1$        & $1$       & $0$            & $1$            & $0$            & $0$            \\ 
            $\hat{B}$   & $1$         & $1$         & $-1$      & $0$            & $0$            & $1$            & $0$            \\ 
            $\hat{N}^1$ & $1$         & $0$         & $0$       & $-\frac{1}{2}$ & $\frac{1}{2}$  & $0$            & $\frac{1}{2}$  \\ 
            $\hat{N}^2$ & $0$         & $1$         & $0$       & $\frac{1}{2}$  & $-\frac{1}{2}$ & $0$            & $\frac{1}{2}$  \\ 
            $\hat{L}$   & $0$         & $0$         & $1$       & $0$            & $0$            & $-1$           & $1$            \\ 
            $\hat{M}$   & $0$         & $0$         & $0$       & $\frac{1}{2}$  & $\frac{1}{2}$  & $1$            & $-\frac{1}{2}$ \\
    \end{longtable}

    Since $T_2$ is a weak del Pezzo surface, Proposition~\ref{MDS} implies that $\hat{S}_2$ is also a Mori dream space, and its Mori cone is
    \[
        \overline{NE}(\hat{S}_2)=\mathrm{Cone}\{[\hat{A}^1],[\hat{A}^2],[\hat{B}],[\hat{N}^1],[\hat{N}^2],[\hat{L}],[\hat{M}]\}.
    \]
    We have the numerical equivalence
    \[
    \begin{split}
        \sigma_2^\ast\left(-K_{S_2}-(1-\lambda)C^2\right)-t\hat{M}&\equiv2\lambda\hat{A}^1+2\lambda\hat{A}^2+3\lambda\hat{B}-t\hat{M}\\
        &\equiv\frac{4-\lambda}{2}(\hat{A}^2+\hat{A}^2)+\frac{6-\lambda}{2}\hat{B}+\frac{t}{2}\hat{L},
    \end{split}
    \]
    and this divisor is pseudoeffective only for $t\leq 4\lambda$. Its Zariski decomposition is given by
       \begin{align*}
        P(t)=\begin{cases}
            \frac{4\lambda-t}{2}(\hat{A}^1+\hat{A}^2)+\frac{6\lambda-t}{2}\hat{B}+\frac{t}{2}\hat{L}\\
            \frac{4\lambda-t}{2}(\hat{A}^1+\hat{A}^2)+\frac{6\lambda-t}{2}(\hat{B}+\hat{L})
        \end{cases};\quad
        N(t)=\begin{cases}
            0, &\ \ \ 0\leq t\leq 3\lambda,\\
            (t-3\lambda)\hat{L}, &\ \ \ 3\lambda\leq t\leq 4\lambda.
        \end{cases}
    \end{align*}
    Then,
    \[
        \mathrm{vol}\left(\sigma_2^\ast\left(-K_{S_2}-(1-\lambda)C^2\right)-t\hat{M}\right)=P(t)^2=\begin{cases}
            7\lambda^2-\frac{1}{2}t^2, &0\leq t\leq3\lambda,\\
            \frac{1}{2}t^2-6\lambda t+16\lambda^2, &3\lambda\leq t\leq 4\lambda,
        \end{cases}
    \]
    and hence,
    \[
        S_{S_2,(1-\lambda)C^2}(\hat{M})=\frac{53\lambda}{21}.
    \]
    We then obtain the upper bound
    \begin{equation}\label{2_C-EF_notflex_ub}
        \delta_p(S_2,(1-\lambda)C^2)\leq\frac{21+42\lambda}{53\lambda}.
    \end{equation}
    On the other hand, we have
    \[
        h(\hat{M},q,t)=\begin{cases}
            \frac{t^2}{8}, &0\leq t\leq3\lambda,\\
            \frac{6\lambda-t}{2}\cdot\mathrm{ord}_q(t-3\lambda)q_L+\frac{(6\lambda-t)^2}{8}, &3\lambda\leq t\leq4\lambda,
        \end{cases}
    \]
    and hence,
    \[
        S(W^{\hat{M}}_{\bullet,\bullet};q)=\begin{cases}
            \frac{23\lambda}{42}, &q\neq q_L,\\
            \frac{5\lambda}{7}, &q=q_L.
        \end{cases}
    \]
    Put $K_{\hat{M}}+\Delta_{\hat{M}}:=(K_{\hat{S}_2}+(1-\lambda)\hat{C}+\hat{M})|_{\hat{M}}$, then the log discrepancy along $q$ is
    \[
        A_{\hat{M},\Delta_{\hat{M}}}(q)=\begin{cases}
            1, &q\neq q_N,q_{C^2},\\
            \frac{1}{2}, &q=q_N,\\
            \lambda, &q=q_{C^2}.
        \end{cases}
    \]
    It then follows from Theorem~\ref{AZ} that
    \begin{equation}\label{2_C-EF_notflex_lb}
        \delta_p(S_2,(1-\lambda)C^2)
        \geq\min\left\{\frac{21+42\lambda}{53\lambda},\frac{42}{23\lambda},\frac{42}{23},\frac{21}{23\lambda},\frac{7}{5\lambda}\right\}=\begin{cases}
            \frac{21}{23\lambda}, &\frac{20}{23}\leq\lambda\leq 1,\\
            \frac{21+42\lambda}{53\lambda}, &\frac{23}{60}\leq\lambda\leq\frac{20}{23},\\
            \frac{42}{23}, &0<\lambda\leq\frac{23}{60}.
        \end{cases}
    \end{equation}
    Consequently, \eqref{2_C-EF_notflex_ub} and \eqref{2_C-EF_notflex_lb} complete the proof.
\end{proof}

\begin{lemma}\label{2_C-AB_flex}
    Suppose that $p$ is a point in $C^2\setminus A^1\cup A^2\cup B$ such that $\phi_2(p)$ is an inflection point of $\phi_2(C^2)$. Then, we have $$\frac{21+63\lambda}{68\lambda}\geq\delta_p(S_2,(1-\lambda)C^2)\geq\min\left\{\frac{21}{23\lambda},\frac{21+63\lambda}{68\lambda},\frac{63}{28}\right\}=\begin{cases}
        \frac{21}{23\lambda}, &\frac{15}{23}\leq\lambda\leq1,\\
        \frac{21+63\lambda}{68\lambda}, &\frac{23}{60}\leq\frac{20}{23},\\
        \frac{42}{23}, &0<\lambda\leq\frac{23}{60}.
    \end{cases}$$
\end{lemma}

\begin{proof}
   As before, we take the curve $L$ in $|A^1+A^2+B|$ that is the strict transform of the tangent line of $\phi_2(C^2)$ at $\phi_2(p)$. Define $\sigma_2:\hat{S}_2\rightarrow (S_2,(1-\lambda)C^2)$ as the $(1,3)$-blowup with respect to $L$. We can see that $q_N$ is an $\mathrm{A}_2$ singularity, where $q_N$ is defined as in the previous lemma.
    The construction of $\hat{S}_2$ can be illustrated below:
    
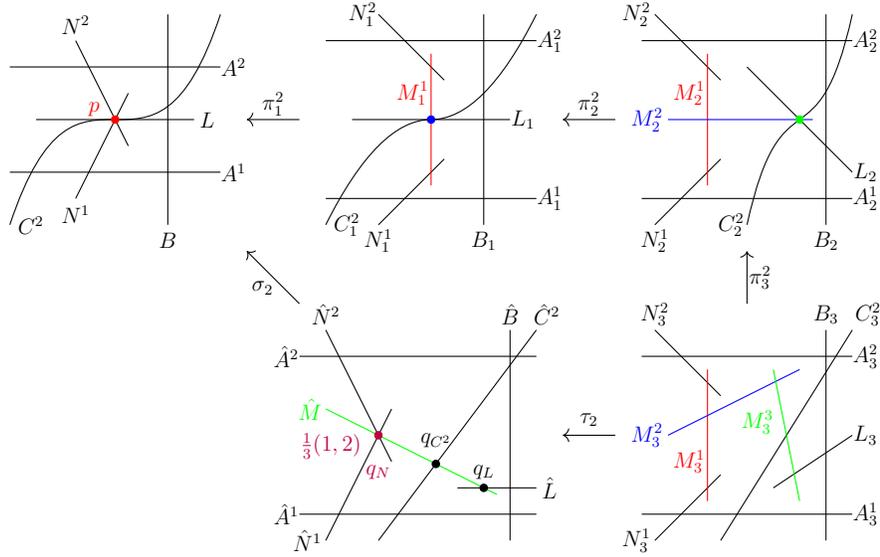
\begin{center}
    \begin{tikzpicture}[scale=0.7, every node/.style={scale=0.7}]
        \draw (-8,3) -- (-4,3);
        \draw (-3.75,3) node {$A^1$};
        \draw (-8,5) -- (-4,5);
        \draw (-3.75,5) node {$A^2$};
        \draw (-5,2) -- (-5,6);
        \draw (-5,1.7) node {$B$}; 
        \draw (-7.5,4) -- (-4.5,4);
        \draw (-4.25,4) node {$L$};
        \draw (-8,2) .. controls (-20/3,6) and (-16/3,2) .. (-4,6);
        \draw (-7.6,2) node {$C^2$};
        \draw (-6.75,2.5) -- (-5.75,4.5);
        \draw (-6.75,2.2) node {$N^1$};
        \draw (-6.75,5.5) -- (-5.75,3.5);
        \draw (-6.75,5.8) node {$N^2$};
        \filldraw[red] (-6,4) circle (2pt);
        \draw[red] (-6.4,4.2) node {$p$};

        \draw[->] (-2.5,4) --(-3.5,4);
        \draw (-3,4.3) node {$\pi^2_1$};
        \draw[->] (-2.5,.5) -- (-3.5,1.5);
        \draw (-3.2,.8) node {$\sigma_2$};

        \draw (-2,2.5) -- (2,2.5);
        \draw (2.25,2.5) node {$A^1_1$};
        \draw (-2,5.5) -- (2,5.5);
        \draw (2.25,5.5) node {$A^2_1$};
        \draw (1,2) -- (1,6);
        \draw (1,1.7) node {$B_1$};
        \draw (-1.5,4) -- (1.5,4);
        \draw (1.75,4) node {$L_1$};
        \draw (-2,2) .. controls (-4/3,10/3) and (-2/3,4) .. (0,4);
        \draw (2,6) .. controls (4/3, 14/3) and (2/3,4) .. (0,4);
        \draw (-1.6,2) node {$C^2_1$};
        \draw (.25,3.25) -- (-1,2);
        \draw (-1,1.7) node {$N^1_1$};
        \draw (.25,4.75) -- (-1,6);
        \draw (-1.3,6) node {$N^2_1$};
        \draw[red] (0,2.75) -- (0,5.25);
        \draw[red] (-.35,4.5) node {$M^1_1$};
        \filldraw[blue] (0,4) circle (2pt);

        \draw[->] (3.5,4) -- (2.5,4);
        \draw (3,4.3) node {$\pi^2_2$};

        \draw (4,2.5) -- (8,2.5);
        \draw (8.25,2.5) node {$A^1_2$};
        \draw (4,5.5) -- (8,5.5);
        \draw (8.25,5.5) node {$A^2_2$};
        \draw (7.5,2) -- (7.5,6);
        \draw (7.5,1.7) node {$B_2$};
        \draw (6,2) .. controls (20/3,5) and (22/3,3) .. (8,6);
        \draw (5.7,2) node {$C^2_2$};
        \draw (6,5) -- (8,3);
        \draw (8.25,3) node {$L_2$};
        \draw (5.5,3.25) -- (4.25,2);
        \draw (4.25,1.7) node {$N^1_2$};
        \draw (5.5,4.75) -- (4.25,6);
        \draw (3.9,6) node {$N^2_2$};
        \draw[red] (5.25,2.75) -- (5.25,5.25);
        \draw[red] (4.9, 4.5) node {$M^1_2$};
        \draw[blue] (4.5,4) -- (7.25,4);
        \draw[blue] (4.1,4) node {$M^2_2$};
        \filldraw[green] (7,4) circle (2pt);

        \draw[->] (6,.5) -- (6,1.5);
        \draw (6.25,1) node {$\pi^2_3$};

        \draw (4,-3.5) -- (8,-3.5);
        \draw (8.25,-3.5) node {$A^1_3$};
        \draw (4,-.5) -- (8,-.5);
        \draw (8.25,-.5) node {$A^2_3$};
        \draw (7.5,-4) -- (7.5,0);
        \draw (7.5,.3) node {$B_3$};
        \draw (5.5,-4) -- (8,0);
        \draw (8.3,.3) node {$C^2_3$};
        \draw (6.5,-3) -- (8,-2);
        \draw (8.25,-2) node {$L_3$};
        \draw (5.5,-2.75) -- (4.25,-4);
        \draw (3.9,-4) node {$N^1_3$};
        \draw (5.5,-1.25) -- (4.25,0);
        \draw (4.25,.3) node {$N^2_3$};
        \draw[red] (5.25,-3.25) -- (5.25,-.75);
        \draw[red] (4.9,-2.5) node {$M^1_3$};
        \draw[blue] (4.5,-2) -- (7,-.75);
        \draw[blue] (4.1,-2) node {$M^2_3$};
        \draw[green] (7,-3.25) -- (6.5,-.75);
        \draw[green] (6.2,-1.75) node {$M^3_3$};

        \draw[->] (3.5,-2) -- (2.5,-2);
        \draw (3,-1.7) node {$\tau_2$};

        \draw (-2.5,-3.5) -- (2,-3.5);
        \draw (-2.75,-3.5) node {$\hat{A}^1$};
        \draw (-2.5,-.5) -- (2,-.5);
        \draw (-2.75,-.5) node {$\hat{A}^2$};
        \draw (1.5,-4) -- (1.5,0);
        \draw (1.5,.3) node {$\hat{B}$};
        \draw (-1,-4) -- (2,0);
        \draw (2.2,.3) node {$\hat{C}^2$};
        \draw (.5,-3) -- (2,-3);
        \draw (2.25,-3) node {$\hat{L}$};
        \draw (-2,-4) -- (-.75,-1.5);
        \draw (-2.35,-4) node {$\hat{N}^1$};
        \draw (-2,0) -- (-.75,-2.5);
        \draw (-2,.3) node {$\hat{N}^2$};
        \draw[green] (1.25,-3.125) -- (-2,-1.5);
        \draw[green] (-2.3,-1.5) node {$\hat{M}$};
        \filldraw[purple] (-1,-2) circle (2pt);
        \draw[purple] (-1.9,-2.2) node {$\frac{1}{3}(1,2)$};
        \draw[purple] (-1,-2.7) node {$q_N$};
        \filldraw (1/11,-28/11) circle (2pt);
        \draw (1/11,-2.1) node {$q_{C^2}$};
        \filldraw (1,-3) circle (2pt);
        \draw (1,-2.7) node {$q_L$};
    \end{tikzpicture}
    \end{center}
    
    We then have
    \[
        \hat{N}^1\equiv\hat{A}^2+\hat{B}-\hat{M},\quad
        \hat{N}^2\equiv\hat{A}^1+\hat{B}-\hat{M},\quad
        \hat{L}\equiv\hat{A}^1+\hat{A}^2+\hat{B}-3\hat{M}.
    \]
    The pullbacks by $\sigma_2$ are given by
    \[
    \begin{split}
        &\sigma_2^\ast A^i=\hat{A}^i,\quad\sigma_2^\ast B =\hat{B}, \quad\sigma_2^\ast L=\hat{L}+3\hat{M},\quad\sigma_2^\ast N^i=\hat{N}_i+\hat{M},\\
        &\sigma_2^\ast K_{S_2}=K_{\hat{S}_2}-3\hat{M},\quad\sigma_2^\ast C^2=\hat{C}^2+3\hat{M},\quad\text{for }i=1,2,
    \end{split}
    \]
    which yield
       $ A_{S_2,(1-\lambda)C^2}(\hat{M})=1+3\lambda.$

    The intersections are given as follows:
    \begin{longtable}{cccccccc}
                        & $\hat{A}^1$ & $\hat{A}^2$ & $\hat{B}$ & $\hat{N}^1$    & $\hat{N}^2$    & $\hat{L}$      & $\hat{M}$ \\ 
            $\hat{A}^1$ & $-1$        & $0$         & $1$       & $1$            & $0$            & $0$            & $0$       \\
            $\hat{A}^2$ & $0$         & $-1$        & $1$       & $0$            & $1$            & $0$            & $0$       \\ 
            $\hat{B}$   & $1$         & $1$         & $-1$      & $0$            & $0$            & $1$            & $0$       \\
            $\hat{N}^1$ & $1$         & $0$         & $0$       & $-\frac{1}{3}$ & $\frac{1}{3}$  & $0$            & $\frac{1}{3}$       \\ 
            $\hat{N}^2$ & $0$         & $1$         & $0$       & $\frac{1}{3}$  & $-\frac{1}{3}$ & $0$            & $\frac{1}{3}$       \\ 
            $\hat{L}$   & $0$         & $0$         & $1$       & $0$            & $0$            & $-2$           & $1$       \\ 
            $\hat{M}$   & $0$         & $0$         & $0$       & $\frac{1}{3}$  & $\frac{1}{3}$  & $1$            & $-\frac{1}{3}$      \\ 
    \end{longtable}

    Since $T_3$ is a weak del Pezzo surface, $\hat{S}_2$ is a Mori dream space, and its Mori cone is given by
    \[
        \overline{NE}(\hat{S}_2)=\mathrm{Cone}\{[\hat{A}^1],[\hat{A}^2],[\hat{B}],[\hat{N}^1],[\hat{N}^2],[\hat{L}],[\hat{M}]\}.
    \]
   We have the numerical equivalence
    \[
    \begin{split}
        \sigma_2^\ast\left(-K_{S_2}-(1-\lambda)C^2\right)-t\hat{M}&\equiv2\lambda\hat{A}^1+2\lambda\hat{A}^2+3\lambda\hat{B}-t\hat{M}\\
        &\equiv\frac{6\lambda-t}{3}(\hat{A}^2+\hat{A}^2)+\frac{9\lambda-t}{3}\hat{B}+\frac{t}{3}\hat{L},
    \end{split}
    \]
    which implies that the divisor is pseudoeffective only for $t\leq6\lambda$.  Its Zariski decomposition is given by
       \begin{align*}
        P(t)=\begin{cases}
            \frac{6\lambda-t}{3}(\hat{A}^1+\hat{A}^2)+\frac{9\lambda-t}{3}\hat{B}+\frac{t}{3}\hat{L}\\
            \frac{6\lambda-t}{3}(\hat{A}^1+\hat{A}^2)+\frac{9\lambda-t}{6}(2\hat{B}+\hat{L})\\
            \frac{6\lambda-t}{3}(\hat{A}^1+\hat{A}^2+4\hat{B}+2\hat{L})
        \end{cases};\quad
        N(t)=\begin{cases}
            0, &0\leq t\leq 3\lambda,\\
            \frac{t-3\lambda}{2}\hat{L}, &3\lambda\leq t\leq 5\lambda,\\
            (t-5\lambda)\hat{B}+(t-4\lambda)\hat{L}, & 5\lambda\leq t\leq6\lambda.
        \end{cases}
    \end{align*}
    Then,
    \[
        \mathrm{vol}\left(\sigma_2^\ast\left(-K_{S_2}-(1-\lambda)C^2\right)-t\hat{M}\right)=P(t)^2=\begin{cases}
            7\lambda^2-\frac{1}{3}t^2, &0\leq t\leq3\lambda,\\
            \frac{1}{6}t^2-3\lambda t+\frac{23}{2}\lambda^2, &3\lambda\leq t\leq 5\lambda,\\
            \frac{2}{3}(6\lambda-t)^2, &5\lambda\leq t\leq 6\lambda,\\
        \end{cases}
    \]
    and hence,
    \[
        S_{S_2,(1-\lambda)C^2}(\hat{M})=\frac{68\lambda}{21}.
    \]
    Thus, the upper bound is given by
    \begin{equation}\label{2_C-EF_flex_ub}
        \delta_p(S_2,(1-\lambda)C^2)\leq\frac{21+63\lambda}{68\lambda}.
    \end{equation}
    To compute a lower bound, for each $q$ on $\hat{M}$, we have
    \[
        h(\hat{M},q,t)=\begin{cases}
            \frac{t^2}{18}, &0\leq t\leq3\lambda,\\
            \frac{9\lambda-t}{3}\cdot\mathrm{ord}_q\frac{t-3\lambda}{2}q_L+\frac{(9\lambda-t)^2}{72}, &3\lambda\leq t\leq5\lambda,\\
            \frac{12\lambda-2t}{3}\cdot\mathrm{ord}_q(t-4\lambda)q_L+\frac{2(6\lambda-t)^2}{9}, &5\lambda\leq t\leq6\lambda,
        \end{cases}
    \]
    and hence,
    \[
        S(W^{\hat{M}}_{\bullet,\bullet};q)=\begin{cases}
            \frac{23\lambda}{63}, &q\neq q_L,\\
            \frac{59\lambda}{63}, &q=q_L.
        \end{cases}
    \]
    Put $K_{\hat{M}}+\Delta_{\hat{M}}:=(K_{\hat{S}_2}+(1-\lambda)\hat{C}+\hat{M})|_{\hat{M}}$, then we obtain
    \[
        A_{\hat{M},\Delta_{\hat{M}}}(q)=\begin{cases}
            1, &q\neq q_N,q_{C^2},\\
            \frac{1}{3}, &q=q_N,\\
            \lambda, &q=q_{C^2}.
        \end{cases}
    \]
    It then follows from Theorem~\ref{AZ} that
    \begin{equation}\label{2_C-EF_flex_lb}
        \delta_p(S_2,(1-\lambda)C^2)
        \geq\min\left\{\frac{21+63\lambda}{68\lambda},\frac{63}{23\lambda},\frac{63}{23},\frac{21}{23\lambda},\frac{63}{59\lambda}\right\}
        =\begin{cases}
            \frac{21}{23\lambda}, &\frac{15}{23}\leq\lambda\leq 1,\\
            \frac{21+63\lambda}{68\lambda}, &\frac{23}{135}\leq\lambda\leq\frac{15}{23},\\
            \frac{63}{23}, &0<\lambda\leq\frac{23}{135}.
        \end{cases}
    \end{equation}
    The proof is concluded from \eqref{2_C-EF_flex_ub} and \eqref{2_C-EF_flex_lb}.
\end{proof}

\begin{lemma}\label{2_CN_tangent}
    Suppose that $p$ is a point in $C^2\setminus A^1\cup A^2\cup B$ and either $N^1$ or $N^2$ intersects $C^2$ tangentially at $p$. Then  $$\min\left\{\frac{7+14\lambda}{19\lambda},\frac{7}{3}\right\}\leq\delta_p(S_2,(1-\lambda)C^2)\leq\frac{7+14\lambda}{19\lambda}.$$ In particular, if $\lambda\geq\frac{3}{13}$, then $$\delta_p(S_2,(1-\lambda)C^2)=\frac{7+14\lambda}{19\lambda}.$$
\end{lemma}

\begin{proof}
    We may assume that $N^1$ is tangent to $C^2$ at $p$. Let $\sigma_2:\hat{S}_2\rightarrow (S_2,(1-\lambda)C^2)$ be the $(1,2)$-blowup at $p$ with respect to $N^1$. Note that $q_{N^2}$ is an $\mathrm{A}_1$ singularity.
    This process is illustrated as follows:
    
    
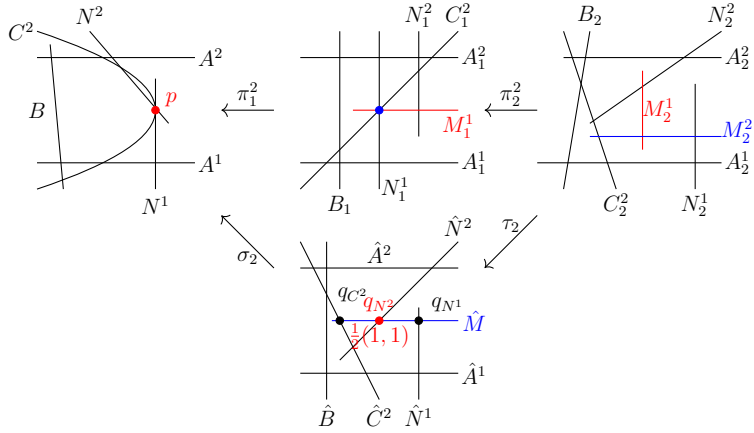
\begin{center}
    \begin{tikzpicture}[scale=0.7, every node/.style={scale=0.7}]
        \draw (-7,1.5) -- (-4,1.5);
        \draw (-3.7,1.5) node {$A^1$};
        \draw (-7,3.5) -- (-4,3.5);
        \draw (-3.7,3.5) node {$A^2$};
        \draw (-6.5,1) -- (-6.75,3.75);
        \draw (-7,2.5) node {$B$};
        \draw (-7,1) .. controls (-4,2) and (-4,3) .. (-7,4);
        \draw (-7.3,4) node {$C^2$};
        \draw (-4.75,1) -- (-4.75,3.1);
        \draw (-4.75,.7) node {$N^1$};
        \draw (-4.5,2.25) -- (-6,4);
        \draw (-6,4.3) node {$N^2$};
        \filldraw[red] (-4.75,2.5) circle (2pt);
        \draw[red] (-4.45,2.7) node {$p$};

        \draw[->] (-2.5,2.5) -- (-3.5,2.5);
        \draw (-3,2.8) node {$\pi^2_1$};
        \draw[->] (-2.5,-.5) -- (-3.5,.5);
        \draw (-3,-.3) node {$\sigma_2$};

        \draw (-2,1.5) -- (1,1.5);
        \draw (1.3,1.5) node {$A^1_1$};
        \draw (-2,3.5) -- (1,3.5);
        \draw (1.3,3.5) node {$A^2_1$};
        \draw (-2,1) -- (1,4);
        \draw (-1.25,1) -- (-1.25,4);
        \draw (-1.25,.7) node {$B_1$};
        \draw (1,4.3) node {$C^2_1$};
        \draw (-.5,1) -- (-.5,4);
        \draw (-.2,1) node {$N^1_1$};
        \draw (.25,2) -- (.25,4);
        \draw (.25,4.3) node {$N^2_1$};
        \draw[red] (-1,2.5) -- (1,2.5);
        \draw[red] (1,2.2) node {$M^1_1$};
        \filldraw[blue] (-.5,2.5) circle (2pt);

        \draw[->] (2.5,2.5) -- (1.5,2.5);
        \draw (2,2.8) node {$\pi^2_2$};
        \draw[->] (2.5,.5) -- (1.5,-.5);
        \draw (2,.3) node {$\tau_2$};

        \draw (2.5,1.5) -- (6,1.5);
        \draw (6.3,1.5) node {$A^1_2$};
        \draw (3,3.5) -- (6,3.5);
        \draw (6.3,3.5) node {$A^2_2$};
        \draw (3,1) -- (3.5,4);
        \draw (3.5,4.3) node {$B_2$};
        \draw (4,1) -- (3,4);
        \draw (4,.7) node {$C^2_2$};
        \draw (5.5,1) -- (5.5,3);
        \draw (5.5,.7) node {$N^1_2$};
        \draw (3.5,2.25) -- (6,4);
        \draw (6,4.3) node {$N^2_2$};
        \draw[red] (4.5,1.75) -- (4.5,3.25);
        \draw[red] (4.8,2.5) node {$M^1_2$};
        \draw[blue] (3.5,2) -- (6,2);
        \draw[blue] (6.3,2.1) node {$M^2_2$};

        \draw (-2,-2.5) -- (1,-2.5);
        \draw (1.3,-2.5) node {$\hat{A}^1$};
        \draw (-1.5,-3) -- (-1.5,0);
        \draw (-.5,-.2) node {$\hat{A}^2$};
        \draw (-2,-.5) -- (1,-.5);
        \draw (-1.5,-3.3) node {$\hat{B}$};
        \draw (-.5,-3) -- (-2,0);
        \draw (-.5,-3.3) node {$\hat{C}^2$};
        \draw (.25,-3) -- (.25,-1.25);
        \draw (.25,-3.3) node {$\hat{N}^1$};
        \draw (-1.25,-2.25) -- (1,0);
        \draw (1,.3) node {$\hat{N}^2$};
        \draw[blue] (-1.4,-1.5)--(1,-1.5);
        \draw[blue] (1.3,-1.5) node {$\hat{M}$};
        \filldraw (-1.25,-1.5) circle (2pt);
        \draw (-1,-1) node {$q_{C^2}$};
        \filldraw[red] (-.5,-1.5) circle (2pt);
        \draw[red] (-.5,-1.2) node {$q_{N^2}$};
        \draw[red] (-.5,-1.8) node {$\frac{1}{2}(1,1)$};
        \filldraw (.25,-1.5) circle (2pt);
        \draw (.8,-1.2) node {$q_{N^1}$};
    \end{tikzpicture}
    \end{center}

    We then have
    \[
        \hat{N}^1\equiv\hat{A}^2+\hat{B}-2\hat{M},\quad\hat{N}^2\equiv\hat
        A^1+\hat{B}-\hat{M}.
    \]
    We also compute
    \[
    \begin{split}
        &\sigma_2^\ast K_{S_2}=K_{\hat{S}_2}-2\hat{M},\quad\sigma_2^\ast N^1=\hat{N}^1+2\hat{M},\quad\sigma_2^\ast N^2=\hat{N}^2+\hat{M},\\
        &\sigma_2^\ast C^2=\hat{C}^2+2\hat{M},\quad\sigma_2^\ast A^i=\hat{A}^1,\quad\sigma_2^\ast A^2=\hat{A}^2,\quad\sigma_2^\ast B=\hat{B},
    \end{split}
    \]
    and it follows that
    \[
        A_{S_2,(1-\lambda)C^2}(\hat{M})=1+2\lambda.
    \]
    The intersections on $\hat{S}_2$ are given by
    \begin{longtable}{ccccccc}
            & $\hat{A}^1$ & $\hat{A}^2$ & $\hat{B}$ & $\hat{N}^1$    & $\hat{N}^2$    & $\hat{M}$ \\ 
            $\hat{A}^1$ & $-1$        & $0$         & $1$       & $1$            & $0$            & $0$       \\ 
            $\hat{A}^2$ & $0$         & $-1$        & $1$       & $0$            & $1$            & $0$       \\ 
            $\hat{B}$   & $1$         & $1$         & $-1$      & $0$            & $0$            & $0$       \\ 
            $\hat{N}^1$ & $1$         & $0$         & $0$       & $-2$ & $0$  & $1$       \\ 
            $\hat{N}^2$ & $0$         & $1$         & $0$       & $0$  & $-\frac{1}{2}$  & $\frac{1}{2}$       \\ 
            $\hat{M}$   & $0$ & $0$ & $0$ & $1$ & $\frac{1}{2}$ & $-\frac{1}{2}$      \\
    \end{longtable}

   As before, $\hat{S}_2$ is a Mori dream space, and its Mori cone is
    \[
        \overline{NE}(\hat{S}_2)=\mathrm{Cone}\{[\hat{A}^1],[\hat{A}^2],[\hat{B}],[\hat{N}^1],[\hat{N}^2],[\hat{M}]\}.
    \]
    Since we have the numerical equivalence
    \[
    \begin{split}
        \sigma_2^\ast\left(-K_{S_2}-(1-\lambda)C^2\right)-t\hat{M}
        &\equiv \frac{t}{2}\hat{N}^1+2\lambda\hat{A}^1+\frac{4\lambda-t}{2}\hat{A}^2+\frac{6\lambda-t}{2}\hat{B}\\
        &\equiv 2\lambda\hat{N}^1+(t-4\lambda)\hat{N}^2+(6\lambda-t)\hat{A}^1+(5\lambda-t)\hat{B},
    \end{split}
    \]
    the pseudoeffective threshold $\tau_{S_2,(1-\lambda)C^2}(\hat{M})$ is $5\lambda$. The Zariski decomposition of this divisor is given by
    \[
    \begin{split}
        P(t)&=\begin{cases}
            \frac{t}{2}\hat{N}^1+2\lambda\hat{A}^1+\frac{4\lambda-t}{2}\hat{A}^2+\frac{6\lambda-t}{2}\hat{B}, &0\leq t\leq2\lambda,\\
            \lambda\hat{N}^1+2\lambda\hat{A}^1+\frac{4\lambda-t}{2}\hat{A}^2+\frac{6\lambda-t}{2}\hat{B}, &2\lambda\leq t\leq4\lambda,\\
            (5\lambda-t)(\hat{N^1}+2\hat{A}^1+\hat{B}), &4\lambda\leq t\leq5\lambda;
        \end{cases}
        \\
        N(t)&=\begin{cases}
            0, &0\leq t\leq 2\lambda,\\
            \frac{t-2\lambda}{2}\hat{N}^1, &2\lambda\leq t\leq 4\lambda,\\
            (t-3\lambda)\hat{N}^1+(t-4\lambda)\hat{N}^2+(t-4\lambda)\hat{A}^1, &4\lambda\leq t\leq 5\lambda.
        \end{cases}
    \end{split}
    \]
    Then,
    \[
        \mathrm{vol}\left(\sigma_2^\ast\left(-K_{S_2}-(1-\lambda)C^2\right)-t\hat{M}\right)=P(t)^2=\begin{cases}
            7\lambda^2-\frac{t^2}{2}, & 0\leq t\leq 2\lambda,\\
            9\lambda^2-2\lambda t, &2\lambda\leq t\leq4\lambda,\\
            (5\lambda-t)^2, &4\lambda\leq t\leq 5\lambda,
        \end{cases}
    \]
    and hence,
    \[
        S_{S_2,(1-\lambda)C^2}(\hat{M})=\frac{19\lambda}{7}.
    \]
    We thus obtain the upper bound
    \begin{equation}\label{2_CN_tangent_ub}
        \delta_p(S_2,(1-\lambda)C^2)\leq\frac{7+14\lambda}{19\lambda}.
    \end{equation}
    Meanwhile, we also have
    \[
        h(\hat{M},q,t)=\begin{cases}
            \frac{t^2}{8}, &0\leq t\leq 2\lambda,\\
            \lambda\cdot\mathrm{ord}_q\frac{t-2\lambda}{2}q_{N^1}+\frac{\lambda^2}{2}, &2\lambda\leq t\leq4\lambda,\\
            (5\lambda-t)\cdot\mathrm{ord}_q\left((t-3\lambda)q_{N^1}+\frac{t-4\lambda}{2}q_{N^2}\right)+\frac{(5\lambda-t)^2}{2}, &4\lambda\leq t\leq 5\lambda,
        \end{cases}
    \]
    and hence,
    \[
        S(W^{\hat{M}}_{\bullet,\bullet};q)=\begin{cases}
            \frac{3\lambda}{7}. &q\neq q_{N^1},q_{N^2},\\
            \frac{19\lambda}{21}, &q=q_{N^1},\\
            \frac{19\lambda}{42}, &q=q_{N^2}.
        \end{cases}
    \]
    Put $\hat{M}+\Delta_{\hat{M}}:=(K_{\hat{S}_2}+(1-\lambda)\hat{C}^2+\hat{M})|_{\hat{M}}$. The log discrepancy of the restricted pair is
    \[
        A_{\hat{M},\Delta_{\hat{M}}}(q)=\begin{cases}
            1, &q\neq q_{N^2},q_{C^2},\\
            \frac{1}{2}, &q=q_{N^2},\\
            \lambda, &q=q_{C^2}.
        \end{cases}
    \]
    From Theorem~\ref{AZ}, we then deduce the lower bound
    \begin{equation}\label{2_CN_tangent_lb}
        \delta_p(S_2,(1-\lambda)C^2)\geq\min\left\{\frac{7+14\lambda}{19\lambda},\frac{7}{3\lambda},\frac{7}{3},\frac{21}{19\lambda}\right\}=\begin{cases}
            \frac{7+14\lambda}{19\lambda}, &\frac{3}{13}\leq t\leq 1,\\
            \frac{7}{3}, &0<t\leq\frac{3}{13}.
        \end{cases}
    \end{equation}
    Then, \eqref{2_CN_tangent_ub} and \eqref{2_CN_tangent_lb} complete the proof.
\end{proof}

\begin{lemma}\label{2_CA-B_notflex}

 Suppose that $p$ is a point in $C^2\cap (A^1\cup A^2)\setminus B$, and  that $\phi_2(p)$ is not an inflection point of $\phi_2(C^2)$. Then $$\min\left\{\frac{21}{23\lambda},\frac{1+\lambda}{2\lambda},\frac{42}{23}\right\}\leq\delta_p(S_2,(1-\lambda)C^2)\leq\min\left\{\frac{21}{23\lambda},\frac{1+\lambda}{2\lambda}\right\}.$$
    In particular, for $\lambda\geq\frac{23}{61}$, we have $$\delta_p(S_2,(1-\lambda)C^2)=\begin{cases}
        \frac{21}{23\lambda}, &\frac{19}{23}\leq\lambda\leq 1,\\
        \frac{1+\lambda}{2\lambda}, &\frac{23}{61}\leq\lambda\leq\frac{19}{23}.
    \end{cases}$$
\end{lemma}

\begin{proof}
    We may assume that $p$ is the intersection point of $C^2$ and $A^1$. By the assumption, $N^1$ and $C^2$ meet transversally  at $p$. Let $\sigma_2:\hat{S}_2\rightarrow (S_2,(1-\lambda)C^2)$ be the blowup at $p$ with the exceptional curve $\hat{M}$. This is the plt blowup of $(S_2,(1-\lambda)C^2)$ that we need. It can be illustrated as follows:
    
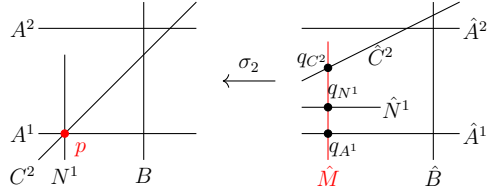
\begin{center}
    \begin{tikzpicture}[scale=0.7, every node/.style={scale=0.7}]
        \draw (-4,-1) -- (-1,-1);
        \draw (-4.3,-1) node {$A^1$};
        \draw  (-4,1) -- (-1,1);
        \draw (-4.3,1) node {$A^2$};
        \draw (-2,-1.5) -- (-2,1.5);
        \draw (-2,-1.8) node {$B$};
        \draw (-4,-1.5) -- (-1,1.5);
        \draw (-4.3,-1.8) node {$C^2$};
        \draw (-3.5,-1.5) -- (-3.5,.5);
        \draw (-3.5,-1.8) node {$N^1$};
        \filldraw[red] (-3.5,-1) circle (2pt);
        \draw[red] (-3.2,-1.3) node {$p$};

        \draw[->] (.5,0) -- (-.5,0);
        \draw (0,.3) node {$\sigma_2$};

        \draw (1,-1) -- (4,-1);
        \draw (4.3,-1) node {$\hat{A}^1$};
        \draw (1,1) -- (4,1);
        \draw (4.3,1) node {$\hat{A}^2$};
        \draw (3.5,-1.5) -- (3.5,1.5);
        \draw (3.5,-1.8) node {$\hat{B}$};
        \draw (1,0) -- (4,1.5);
        \draw (2.5,.5) node {$\hat{C}^2$};
        \draw (1,-.5) -- (2.5,-.5);
        \draw (2.8,-.5) node {$\hat{N}^1$};
        \draw[red] (1.5,-1.5) -- (1.5,.75);
        \draw[red] (1.5,-1.8) node {$\hat{M}$};
        \filldraw (1.5,.25) circle (2pt);
        \draw (1.2,.4) node {$q_{C^2}$};
        \filldraw (1.5,-.5) circle (2pt);
        \draw (1.8,-.2) node {$q_{N^1}$};
        \filldraw (1.5,-1) circle (2pt);
        \draw (1.8,-1.3) node {$q_{A^1}$};
    \end{tikzpicture}
    \end{center}

    Note that the strict transform of $N^1$ satisfies $\hat{N}^1\equiv\hat{A}^2+\hat{B}-\hat{M}$. We also have
    \[
    \begin{split}
        &\sigma_2^\ast A^1=\hat{A}^1+\hat{M},\quad\sigma_2^\ast A^2=\hat{A}^2,\quad\sigma_2^\ast B=\hat{B},\quad\sigma_2^\ast N^1=\hat{N}^1+\hat{M},\\
        &\sigma_2^\ast K_{S_2}=K_{\hat{S}_2}-\hat{M},\quad\sigma_2^\ast C^2=\hat{C}^2+\hat{M},
    \end{split}
    \]
    that directly imply
    \[
        A_{S_2,(1-\lambda)C^2}(\hat{M})=1+\lambda.
    \]
    
    The weak del Pezzo surface $\hat{S}_2$ is a Mori dream space, and its Mori cone is spanned by the classes $[\hat{A}^1]$, $[\hat{A}^2]$, $[\hat{B}]$, $[\hat{N}^1]$, and $[\hat{M}]$.
    We have the numerical equivalence
    \[
    \begin{split}
        \sigma_2^\ast\left(-K_{S_2}-(1-\lambda)C^2\right)-t\hat{M}&\equiv2\lambda\hat{A}^1+2\lambda\hat{A}^2+3\lambda\hat{B}+(2\lambda-t)\hat{M}\\
        &\equiv2\lambda\hat{A}^1+(4\lambda-t)\hat{A}^2+(5\lambda-t)\hat{B}+(t-2\lambda)\hat{N}^1,
    \end{split}
    \]
    and it is pseudoeffective only for $t$ not exceeding $4\lambda$.  Its Zariski decomposition is given by
    \begin{align*}
        P(t)&=\begin{cases}
            2\lambda\hat{A}^1+2\lambda\hat{A}^2+3\lambda\hat{B}+(2\lambda-t)\hat{M}, &0\leq t\leq\lambda,\\
            \frac{5\lambda-t}{2}\hat{A}^1+(4\lambda-t)\hat{A}^2+(5\lambda-t)\hat{B}+(t-2\lambda)\hat{N}^1, &\lambda\leq t\leq 2\lambda,\\
            \frac{5\lambda-t}{2}\hat{A}^1+(4\lambda-t)\hat{A}^2+(5\lambda-t)\hat{B}, &2\lambda\leq t\leq3\lambda,\\
            (4\lambda-t)(\hat{A}^1+\hat{A}^2+2\hat{B}), &3\lambda\leq t\leq 4\lambda;
        \end{cases}
        \\
        N(t)&=\begin{cases}
            0, &0\leq t\leq \lambda,\\
            \frac{t-\lambda}{2}\hat{A}^1, &\lambda\leq t\leq 2\lambda,\\
            \frac{t-\lambda}{2}\hat{A}^1+(t-2\lambda)\hat{N}^1, &2\lambda\leq t\leq3\lambda,\\
            (t-2\lambda)\hat{A}^1+(t-3\lambda)\hat{B}+(t-2\lambda)\hat{N}^1, &3\lambda\leq t\leq 4\lambda.
        \end{cases}
    \end{align*}
    Then,
    \[
        \mathrm{vol}\left(\sigma_2^\ast\left(-K_{S_2}-(1-\lambda)C^2\right)-t\hat{M}\right)=P(t)^2=\begin{cases}
            7\lambda^2-t^2, &0\leq t\leq\lambda,\\
            -\frac{1}{2}t^2-\lambda t+\frac{15}{2}\lambda^2, &\lambda\leq t\leq 2\lambda,\\
            \frac{1}{2}t^2-5\lambda t+\frac{23}{2}\lambda^2, &2\lambda\leq t\leq3\lambda,\\
            (4\lambda-t)^2, &3\lambda\leq t\leq 4\lambda,\\
        \end{cases}
    \]
    and hence,
    \[
        S_{S_2,(1-\lambda)C^2}(\hat{M})=2\lambda.
    \]
    From \eqref{2_Ei-CF_ub}, we obtain the upper bound
    \begin{equation}\label{２_CE-F_notflex_ub}
        \delta_p(S_2,(1-\lambda)C^2)\leq\min\left\{\frac{21}{23\lambda},\frac{1+\lambda}{2\lambda}\right\}.
    \end{equation}
    On the other hand, for each $q$ on $\hat{M}$, 
    \[
        h(\hat{M},q,t)=\begin{cases}
            \frac{t^2}{2}, &0\leq t\leq\lambda,\\
            \frac{t+\lambda}{2}\cdot\mathrm{ord}_q\frac{t-\lambda}{2}q_{A^1}+\frac{(t+\lambda)^2}{8}, &\lambda\leq t\leq2\lambda,\\
            \frac{5\lambda-t}{2}\cdot\mathrm{ord}_q\left(\frac{t-\lambda}{2}q_{A^1}+(t-2\lambda)q_{N^1}\right)+\frac{(5\lambda-t)^2}{8}, &2\lambda\leq t\leq3\lambda,\\
            (4\lambda-t)\cdot\mathrm{ord}_q((t-2\lambda)q_{A^1}+(t-2\lambda)q_{N^1}))+\frac{(4\lambda-t)^2}{2}, &3\lambda\leq t\leq4\lambda,
        \end{cases}
    \]
    and hence,
    \[
        S(W^{\hat{M}}_{\bullet,\bullet};q)=\begin{cases}
            \frac{23\lambda}{42}, &q\neq q_L,q_{A^1}\\
            \frac{19\lambda}{21}, &q=q_{N^1},\\
            \frac{23\lambda}{21}, &q=q_{A^1}.
        \end{cases}
    \]
    Put $K_{\hat{M}}+\Delta_{\hat{M}}:=(K_{\hat{S}_2}+(1-\lambda)\hat{C}+\hat{M})|_{\hat{M}}$, then
    \[
        A_{\hat{M},\Delta_{\hat{M}}}(q)=\begin{cases}
            1, &q\neq q_{C^2},\\
            \lambda, &q=q_{C^2}.
        \end{cases}
    \]
    It then follows from Theorem~\ref{AZ} that
    \begin{equation}\label{２_CE-F_notflex_lb}
        \delta_p(S_2,(1-\lambda)C^2)
        \geq\min\left\{\frac{1+\lambda}{2\lambda},\frac{42}{23\lambda},\frac{42}{23},\frac{21}{19\lambda},\frac{21}{23\lambda}\right\}
        =\begin{cases}
            \frac{21}{23\lambda}, &\frac{19}{23}\leq\lambda\leq 1,\\
            \frac{1+\lambda}{2\lambda}, &\frac{23}{61}\leq\lambda\leq\frac{19}{23},\\
            \frac{42}{23}, &0<\lambda\leq\frac{23}{61}.
        \end{cases}
    \end{equation}
    Consequently, \eqref{２_CE-F_notflex_ub} and \eqref{２_CE-F_notflex_lb} complete the proof.    
\end{proof}

\begin{lemma}\label{2_CA-B_flex}
    Suppose that $p$ is a point in  $C^2\cap (A^1\cup A^2)\setminus B$, and  that $\phi_2(p)$ is an inflection point of $\phi_2(C^2)$.  Then $$\min\left\{\frac{21}{23\lambda},\frac{21+42\lambda}{61\lambda},\frac{14}{5}\right\}\leq\delta_p(S_2,(1-\lambda)C^2)\leq\min\left\{\frac{21}{23\lambda},\frac{21+42\lambda}{61\lambda}\right\}.$$
    In particular, for $\lambda\geq\frac{15}{92}$, we have $$\delta_p(S_2,(1-\lambda)C^2)=\begin{cases}
        \frac{21}{23\lambda}, &\frac{19}{23}\leq\lambda\leq 1,\\
        \frac{21+42\lambda}{61\lambda}, &\frac{15}{92}\leq\lambda\leq\frac{19}{23}.
    \end{cases}$$
\end{lemma}

\begin{proof}
As in the previous lemma,     we may assume that $p$ is the intersection point of $C^2$ and $A^1$. By the assumption, $N^1$ is tangent to $C^2$ at $p$. Denote the $(1,2)$-blowup at $p$ with respect to $N^1$ by $\sigma_2:\hat{S}_2\rightarrow (S_2,(1-\lambda)C^2)$. Note that $q_{A^1}$ is an $\mathrm{A}_1$ singularity.
     The construction of $\hat{S}_2$ is shown as follows:

    
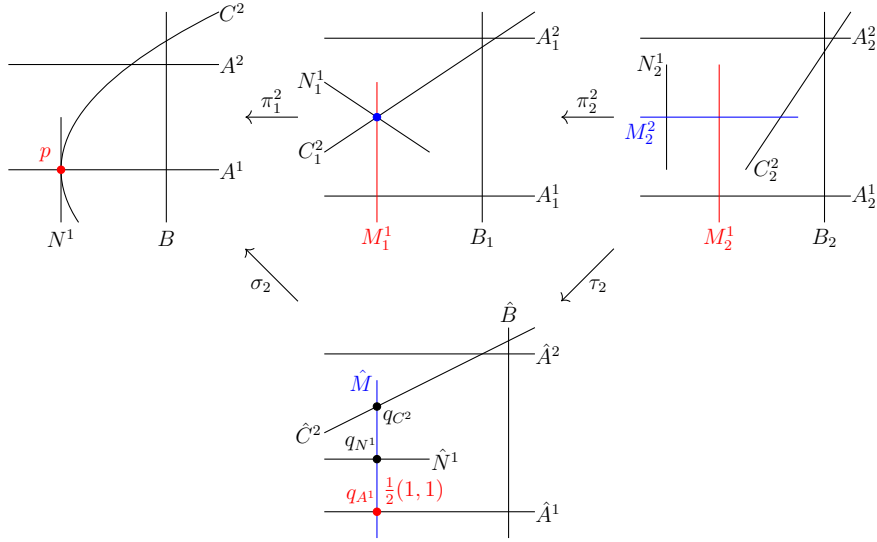
\begin{center}
    \begin{tikzpicture}[scale=0.7, every node/.style={scale=0.7}]
        \draw (-8,3) -- (-4,3);
        \draw (-3.75,3) node {$A^1$};
        \draw (-8,5) -- (-4,5);
        \draw (-3.75,5) node {$A^2$};
        \draw (-5,2) -- (-5,6);
        \draw (-5,1.7) node {$B$};
        \draw (-20/3,2) .. controls (-68/9,10/3) and (-20/3,14/3) .. (-4,6);
        \draw (-3.75,6) node {$C^2$};
        \draw (-7,2) -- (-7,4);
        \draw (-7,1.7) node {$N^1$};
        \filldraw[red] (-7,3) circle (2pt);
        \draw[red] (-7.3,3.3) node {$p$};

        \draw[->] (-2.5,4) --(-3.5,4);
        \draw (-3,4.3) node {$\pi^2_1$};
        \draw[->] (-2.5,.5) -- (-3.5,1.5);
        \draw (-3.2,.8) node {$\sigma_2$};

        \draw (-2,2.5) -- (2,2.5);
        \draw (2.25,2.5) node {$A^1_1$};
        \draw (-2,5.5) -- (2,5.5);
        \draw (2.25,5.5) node {$A^2_1$};
        \draw (1,2) -- (1,6);
        \draw (1,1.7) node {$B_1$};
        \draw (-2,10/3) -- (2,6);
        \draw (-2.25,10/3) node {$C^2_1$};
        \draw (-2,14/3) -- (0,10/3);
        \draw (-2.25,14/3) node {$N^1_1$};
        \draw[red] (-1,2) -- (-1,14/3);
        \draw[red] (-1,1.7) node {$M^1_1$};
        \filldraw[blue] (-1,4) circle (2pt);

        \draw[->] (3.5,4) -- (2.5,4);
        \draw (3,4.3) node {$\pi^2_2$};
        \draw[->] (3.5,1.5) -- (2.5,.5);
        \draw (3.2,.8) node {$\tau_2$};

        \draw (4,2.5) -- (8,2.5);
        \draw (8.25,2.5) node {$A^1_2$};
        \draw (4,5.5) -- (8,5.5);
        \draw (8.25,5.5) node {$A^2_2$};
        \draw (7.5,2) -- (7.5,6);
        \draw (7.5,1.7) node {$B_2$};
        \draw (6,3) -- (8,6);
        \draw (6.4,3) node {$C^2_2$};
        \draw (4.5,3) -- (4.5,5);
        \draw (4.2,5) node {$N^1_2$};
        \draw[red] (5.5,2) -- (5.5,5);
        \draw[red] (5.5,1.7) node {$M^1_2$};
        \draw[blue] (4,4) -- (7,4);
        \draw[blue] (4,3.7) node {$M^2_2$};

        \draw (-2,-3.5) -- (2,-3.5);
        \draw (2.25,-3.5) node {$\hat{A}^1$};
        \draw (-2,-.5) -- (2,-.5);
        \draw (2.25, -.5) node {$\hat{A}^2$};
        \draw (1.5,-4) -- (1.5,0);
        \draw (1.5,.3) node {$\hat{B}$};
        \draw (-2,-2) -- (2,0);
        \draw (-2.3,-2) node {$\hat{C}^2$};
        \draw (-2,-2.5) -- (0,-2.5);
        \draw (.3,-2.5) node {$\hat{N}^1$};
        \draw[blue] (-1,-4) -- (-1,-1);
        \draw[blue] (-1.3,-1) node {$\hat{M}$};
        \filldraw[red] (-1,-3.5) circle (2pt);
        \draw[red] (-.3,-3.1) node {$\frac{1}{2}(1,1)$};
        \draw[red] (-1.3,-3.2) node {$q_{A^1}$};
        \filldraw (-1,-2.5) circle (2pt);
        \draw (-1.3,-2.2) node {$q_{N^1}$};
        \filldraw (-1,-1.5) circle (2pt);
        \draw (-.6,-1.7) node {$q_{C^2}$};
    \end{tikzpicture}
    \end{center}

    We have $\hat{N}^1\equiv\hat{A}^2+\hat{B}-2\hat{M}$ and
    \[
    \begin{split}
        &\sigma_2^\ast A^1=\hat{A}^1+\hat{M},\quad\sigma_2^\ast A^2=\hat{A}^2,\quad\sigma_2^\ast B=\hat{B},\quad\sigma_2^\ast N^1=\hat{N}^1+2\hat{M},\\
        &\sigma_2^\ast K_{S_2}=K_{\hat{S}_2}-2\hat{M},\quad\sigma_2^\ast C^2=\hat{C}^2+2\hat{M}.
    \end{split}
    \]
    It then follows that
    \[
        A_{S_2,(1-\lambda)C^2}(\hat{M})=1+2\lambda.
    \]
    
    The pullbacks by $\sigma_2$ yield the intersections as follows:
        \[
    \begin{split}
        &(\hat{A}^1)^2=-\frac{3}{2},\quad (\hat{N}^1)^2=-2,\quad\hat{M}^2=-\frac{1}{2},\quad(\hat{A}^2)^2=\hat{B}^2=-1,\quad\hat{A}^1\cdot\hat{M}=\frac{1}{2}\\
        &\hat{A}^1\cdot\hat{A}^2=\hat{A}^1\cdot\hat{N}^1=\hat{A}^2\cdot\hat{N}^1=\hat{A}^2\cdot\hat{M}=\hat{B}\cdot\hat{N}^1=\hat{B}\cdot{M}=0,\quad\hat{A}^1\cdot\hat{B}=\hat{A}^2\cdot\hat{B}=\hat{N}^1\cdot\hat{M}=1.
    \end{split}
    \]

    The surface $\hat{S}_2$ is a Mori dream space, and its Mori cone is
    \[
        \overline{NE}(\hat{S}_2)=\mathrm{Cone}\{[\hat{A}^1],[\hat{A}^2],[\hat{B}],[\hat{N}^1],[\hat{M}]\},
    \]    
    by Proposition~\ref{MDS}. We have
    \[
    \begin{split}
        \sigma_2^\ast\left(-K_{S_2}-(1-\lambda)C^2\right)-t\hat{M}&\equiv2\lambda\hat{A}^1+2\lambda\hat{A}^2+3\lambda\hat{B}+(2\lambda-t)\hat{M}\\
        &\equiv2\lambda\hat{A}^1+\frac{6\lambda-t}{2}\hat{A}^2+\frac{8\lambda-t}{2}\hat{B}+\frac{t-2\lambda}{2}\hat{N}^1
    \end{split}
    \]
    and it is pseudoeffective only for $t\leq 6\lambda$. The Zariski decomposition is given by
    \begin{align*}
        P(t)&=\begin{cases}
            2\lambda\hat{A}^1+2\lambda\hat{A}^2+3\lambda\hat{B}+(2\lambda-t)\hat{M}, &0\leq t\leq2\lambda,\\
            \frac{8\lambda-t}{6}(2\hat{A}^1+3\hat{B})+\frac{6\lambda-t}{2}\hat{A}^2, &2\lambda\leq t\leq 5\lambda,\\
            \frac{6\lambda-t}{2}(2\hat{A}^1+\hat{A}^2+3\hat{B}), &5\lambda\leq t\leq 6\lambda;
        \end{cases}
        \\
        N(t)&=\begin{cases}
            0, &0\leq t\leq 2\lambda,\\
            \frac{t-2\lambda}{6}(2\hat{A}^1+3\hat{N}^1), &2\lambda\leq t\leq5\lambda,\\
            (t-4\lambda)\hat{A}^1+(t-5\lambda)\hat{B}+\frac{t-2\lambda}{2}\hat{N}^1, &5\lambda\leq t\leq 6\lambda.
        \end{cases}
    \end{align*}
    We then obtain the volume function
    \[
        \mathrm{vol}\left(\sigma_2^\ast\left(-K_{S_2}-(1-\lambda\right)C^2)-t\hat{M}\right)=P(t)^2=\begin{cases}
            7\lambda^2-\frac{1}{2}t^2, &0\leq t\leq2\lambda,\\
            \frac{1}{6}t^2-\frac{8}{3}\lambda t+\frac{29}{3}\lambda^2, &2\lambda\leq t\leq 5\lambda,\\
            \frac{(6\lambda-t)^2}{2}, &5\lambda\leq t\leq 6\lambda,\\
        \end{cases}
    \]
    and hence,
    \[
        S_{S_2,(1-\lambda)C^2}(\hat{M})=\frac{61\lambda}{21}.
    \]
    From \eqref{2_Ei-CF_ub}, the upper bound is given by
    \begin{equation}\label{２_CE-F_flex_ub}
        \delta_p(S_2,(1-\lambda)C^2)\leq\min\left\{\frac{21}{23\lambda},\frac{21+42\lambda}{61\lambda}\right\}.
    \end{equation}
   For each $q$ on $\hat{M}$, 
    \[
        h(\hat{M},q,t)=\begin{cases}
            \frac{t^2}{8}, &0\leq t\leq2\lambda,\\
            \frac{8\lambda-t}{6}\cdot\mathrm{ord}_q\left(\frac{t-2\lambda}{6}(q_{A^1}+3q_{N^1})\right)+\frac{(8\lambda-t)^2}{72}, &2\lambda\leq t\leq5\lambda,\\
            \frac{6\lambda-t}{2}\cdot\mathrm{ord}_q\left(\frac{t-4\lambda}{2}q_{A^1}+\frac{t-2\lambda}{2}q_{N^1})\right)+\frac{(6\lambda-t)^2}{8}, &5\lambda\leq t\leq6\lambda,
        \end{cases}
    \]
    and hence,
    \[
        S(W^{\hat{M}}_{\bullet,\bullet};q)=\begin{cases}
            \frac{5\lambda}{14}, &q\neq q_L,q_{A^1}\\
            \frac{19\lambda}{21}, &q=q_{N^1},\\
            \frac{23\lambda}{42}, &q=q_{A^1}.
        \end{cases}
    \]
    Put $K_{\hat{M}}+\Delta_{\hat{M}}:=(K_{\hat{S}_2}+(1-\lambda)\hat{C}+\hat{M})|_{\hat{M}}$. Then 
    \[
        A_{\hat{M},\Delta_{\hat{M}}}(q)=\begin{cases}
            1, &q\neq q_{A^1},q_{C^2},\\
            \frac{1}{2}, &q=q_{A^1},\\
            \lambda, &q=q_{C^2}.
        \end{cases}
    \]
    It then follows from Theorem~\ref{AZ} that
    \begin{equation}\label{２_CE-F_flex_lb}
        \delta_p(S_2,(1-\lambda)C^2)
        \geq\min\left\{\frac{21+42\lambda}{61\lambda},\frac{14}{5\lambda},\frac{14}{5},\frac{21}{19\lambda},\frac{21}{23\lambda}\right\}=\begin{cases}
            \frac{21}{23\lambda}, &\frac{19}{23}\leq\lambda\leq 1,\\
            \frac{21+42\lambda}{61\lambda}, &\frac{15}{92}\leq\lambda\leq\frac{19}{23},\\
            \frac{14}{5}, &0<\lambda\leq\frac{15}{92}.
        \end{cases}
    \end{equation}
    The proof is obtained by combining \eqref{２_CE-F_flex_ub} and \eqref{２_CE-F_flex_lb}.    
\end{proof}

\begin{lemma}\label{2_CB-A_notflex}
    Suppose that $p$ is the intersection point of $ C^2$ and $B$ that is not contained in $A^1\cup A^2$, and that $\phi_2(p)$ is not an inflection point of $\phi_2(C^2)$. Then
    $$\min\left\{\frac{21}{25\lambda},\frac{21+42\lambda}{55\lambda},\frac{63}{29}\right\}\leq\delta_p(S_2,(1-\lambda)C^2)\leq\min\left\{\frac{21}{25\lambda},\frac{21+42\lambda}{55\lambda}\right\}.$$ In particular, for $\lambda\geq\frac{29}{107}$, we have $$\delta_p(S_2,(1-\lambda)C^2)=\min\left\{\frac{21}{25\lambda},\frac{21+42\lambda}{55\lambda}\right\}=\begin{cases}
        \frac{21}{25\lambda}, &\frac{3}{5}\leq\lambda\leq 1,\\
        \frac{21+42\lambda}{55}, &\frac{29}{107}\leq\lambda\leq\frac{3}{5}.
    \end{cases}$$
\end{lemma}

\begin{proof}
    Take the curve $L$ in $|A^1+A^2+B|$ that is tangent to $C^2$ at $p$. Let $\sigma_2:\hat{S}_2\rightarrow (S_2,(1-\lambda)C^2)$ be the $(1,2)$-blowup at $p$ with respect to $L$. We can check that $q_B$ is an $\mathrm{A}_1$ singularity. 
    This can be illustrated as follows:
    
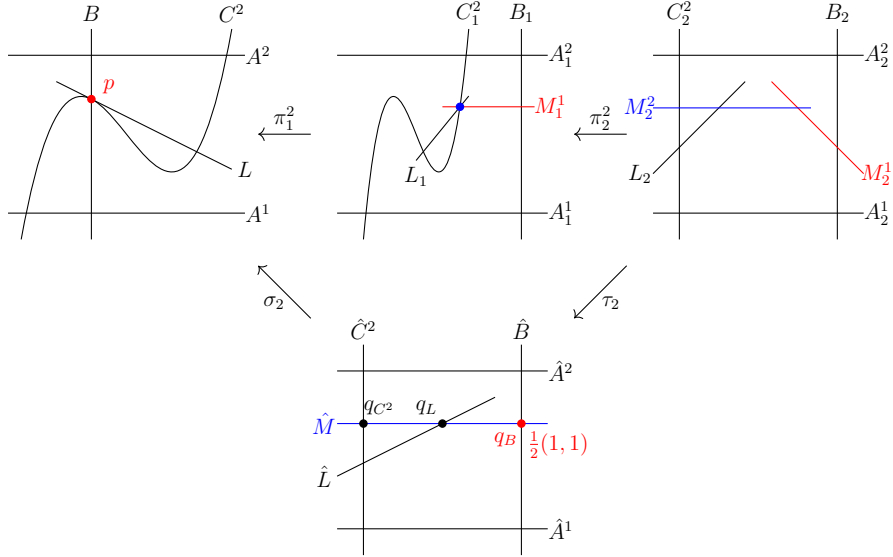
\begin{center}
    \begin{tikzpicture}[scale=0.7, every node/.style={scale=0.7}]
        \draw (-8.25,2.5) -- (-3.75,2.5);
        \draw (-3.5,2.5) node {$A^1$};
        \draw (-8.25,5.5) -- (-3.75,5.5);
        \draw (-3.5,5.5) node {$A^2$};
        \draw (-20/3,2) -- (-20/3,6);
        \draw (-20/3,6.3) node {$B$};
        \draw (-8,2) .. controls (-20/3,28/3) and (-16/3,-4/3) .. (-4,6);
        \draw (-4,6.3) node {$C^2$};
        \draw (-22/3,5) -- (-4,10/3);
        \draw (-3.75,10/3) node {$L$};
        \filldraw[red] (-20/3,14/3) circle (2pt);
        \draw[red] (-19/3,4.9) node {$p$};

        \draw[->] (-2.5,4) --(-3.5,4);
        \draw (-3,4.3) node {$\pi^2_1$};
        \draw[->] (-2.5,.5) -- (-3.5,1.5);
        \draw (-3.2,.8) node {$\sigma_2$};

        \draw (-2,2.5) -- (2,2.5);
        \draw (2.25,2.5) node {$A^1_1$};
        \draw (-2,5.5) -- (2,5.5);
        \draw (2.25,5.5) node {$A^2_1$};
        \draw (1.5,2) -- (1.5,6);
        \draw (1.5,6.3) node {$B_1$};
        \draw (-1.5,2) .. controls (-5/6,28/3) and (-1/6,-4/3) .. (.5,6);
        \draw (.5,6.3) node {$C^2_1$};
        \draw (-.5,3.5) -- (.5,189/40);
        \draw (-.5,3.2) node {$L_1$};
        \draw[red] (0,217/48) -- (1.75,217/48);
        \draw[red] (2.05,217/48) node {$M^1_1$};
        \filldraw[blue] (1/3,217/48) circle (2pt);

        \draw[->] (3.5,4) -- (2.5,4);
        \draw (3,4.3) node {$\pi^2_2$};
        \draw[->] (3.5,1.5) -- (2.5,.5);
        \draw (3.2,.8) node {$\tau_2$};

        \draw (4,2.5) -- (8,2.5);
        \draw (8.25,2.5) node {$A^1_2$};
        \draw (4,5.5) -- (8,5.5);
        \draw (8.25,5.5) node {$A^2_2$};
        \draw (7.5,2) -- (7.5,6);
        \draw (7.5,6.3) node {$B_2$};
        \draw (4.5,2) -- (4.5,6);
        \draw (4.5,6.3) node {$C^2_2$};
        \draw (4,3.25) -- (5.75,5);
        \draw (3.75,3.25) node {$L_2$};
        \draw[red] (6.25,5) -- (8,3.25);
        \draw[red] (8.25,3.25) node {$M^1_2$};
        \draw[blue] (4,4.5) -- (7,4.5);
        \draw[blue] (3.75,4.5) node {$M^2_2$};

        \draw (-2,-3.5) -- (2,-3.5);
        \draw (2.25,-3.5) node {$\hat{A}^1$};
        \draw (-2,-.5) -- (2,-.5);
        \draw (2.25,-.5) node {$\hat{A}^2$};
        \draw (1.5,-4) -- (1.5,0);
        \draw (1.5,.3) node {$\hat{B}$};
        \draw (-1.5,-4) -- (-1.5,0);
        \draw (-1.5,.3) node {$\hat{C}^2$};
        \draw (-2,-2.5) -- (1,-1);
        \draw (-2.25,-2.5) node {$\hat{L}$};
        \draw[blue] (-2,-1.5) -- (2,-1.5);
        \draw[blue] (-2.25,-1.5) node {$\hat{M}$};
        \filldraw (-1.5,-1.5) circle (2pt);
        \draw (-1.2,-1.2) node {$q_{C^2}$};
        \filldraw (0,-1.5) circle (2pt);
        \draw (-.3,-1.2) node {$q_L$};
        \filldraw[red] (1.5,-1.5) circle (2pt);
        \draw[red] (2.2,-1.9) node {$\frac{1}{2}(1,1)$};
        \draw[red] (1.2,-1.8) node {$q_B$};
    \end{tikzpicture}
    \end{center}
    We  have $\hat{L}\equiv \hat{A}^1+\hat{A}^2+\hat{B}-\hat{M}$ and
    \[
    \begin{split}
        &\sigma_2^\ast A^1=\hat{A}^1,\quad\sigma_2^\ast A^2=\hat{A}^2,\quad\sigma_2^\ast B=\hat{B}+\hat{M},\quad\sigma_2^\ast L=\hat{L}+2\hat{M},\\
        &\sigma_2^\ast K_{S_2}=K_{\hat{S}_2}-2\hat{M},\quad\sigma_2^\ast C^2=\hat{C}^2+2\hat{M},
    \end{split}
    \]
    which imply
    \[
        A_{S_2,(1-\lambda)C^2}(\hat{M})=1+2\lambda.
    \]

    The intersections on $\hat{S}_2$ are given as follows:
        \[
    \begin{split}
        &(\hat{A}^i)^2=\hat{L}^2=-1,\quad\hat{B}^2=-\frac{3}{2},\quad\hat{M}^2=-\frac{1}{2},\quad\hat{B}\cdot\hat{M}=\frac{1}{2},\\
        &\hat{A}^i\cdot\hat{A}^{3-i}=\hat{A}^i\cdot\hat{L}=\hat{A}^i\cdot\hat{M}=\hat{B}\cdot\hat{L}=0,\quad\hat{A}^i\cdot\hat{B}=\hat{L}\cdot\hat{M}=1,\quad\text{for }i=1,2.
    \end{split}
    \]

    As in the previous Lemmas, the surface $\hat{S}_2$ is a Mori dream space, and its Mori cone is spanned by 
    $[\hat{A}^1]$, $[\hat{A}^2]$, $[\hat{B}]$,  $[\hat{L}]$, and $[\hat{M}]$.   
    We compute
    \[
    \begin{split}
        \sigma_2^\ast\left(-K_{S_2}-(1-\lambda)C^2\right)-t\hat{M}&\equiv2\lambda\hat{A}^1+2\lambda\hat{A}^2+3\lambda\hat{B}+(3\lambda-t)\hat{M}\\
        &\equiv(5\lambda-t)\hat{A}^1+(5\lambda-t)\hat{A}^2+(6\lambda-t)\hat{B}+(t-3\lambda)\hat{L}.
    \end{split}
    \]
    The divisor  is pseudoeffective only for $t$ not exceeding $5\lambda$. We compute its Zariski decomposition as follows:
        \begin{align*}
      P(t)=\begin{cases}
           2\lambda\hat{A}^1+2\lambda\hat{A}^2+3\lambda\hat{B}+(3\lambda-t)\hat{M}\\
           \frac{5\lambda-t}{3}(3\hat{A}^1+3\hat{A}^2+4\hat{B})+(t-3\lambda)\hat{L}\\
           \frac{5\lambda-t}{3}(3\hat{A}^1+3\hat{A}^2+4\hat{B})
       \end{cases};\quad
        N(t)=\begin{cases}
           0, & 0\leq t\leq 2\lambda,\\
          \frac{t-2\lambda}{3}\hat{B}, & 2\lambda\leq t\leq3\lambda,\\
           \frac{t-2\lambda}{3}\hat{B}+(t-3\lambda)\hat{L}, & 3\lambda\leq t\leq 5\lambda.
       \end{cases}
    \end{align*}
    This implies
    \[
        \mathrm{vol}\left(\sigma_2^\ast\left(-K_{S_2}-(1-\lambda)C^2\right)-t\hat{M}\right)=P(t)^2=\begin{cases}
            7\lambda^2-\frac{1}{2}t^2, &0\leq t\leq2\lambda,\\
            -\frac{1}{3}t^2-\frac{2\lambda}{3}t+\frac{23}{3}\lambda^2, &2\lambda\leq t\leq3\lambda,\\
            \frac{2(5\lambda-t)^2}{3}, &3\lambda\leq t\leq 5\lambda,\\
        \end{cases}
    \]
    and hence,
    \[
        S_{S_2,(1-\lambda)C^2}(\hat{M})=\frac{55\lambda}{21}.
    \]
    From \eqref{2_F-C_ub}, we obtain the upper bound
    \begin{equation}\label{２_CF-E_notflex_ub}
        \delta_p(S_2,(1-\lambda)C^2)\leq\min\left\{\frac{21}{25\lambda},\frac{21+42\lambda}{55\lambda}\right\}.
    \end{equation}
    On the other hand, for each $q$ on $\hat{M}$, we compute the integrand in \eqref{eq:SS} as follows:
    \[
        h(\hat{M},q,t)=\begin{cases}
            \frac{t^2}{8}, &0\leq t\leq2\lambda,\\
            \frac{t+\lambda}{3}\cdot\mathrm{ord}_q\frac{t-2\lambda}{6}q_B+\frac{(t+\lambda)^2}{18}, &2\lambda\leq t\leq3\lambda,\\
            \frac{10\lambda-2t}{3}\cdot\mathrm{ord}_q\left(\frac{t-2\lambda}{6}q_B+(t-3\lambda)q_L)\right)+\frac{2(5\lambda-t)^2}{9}, &3\lambda\leq t\leq5\lambda,
        \end{cases}
    \]
    and hence,
    \[
        S(W^{\hat{M}}_{\bullet,\bullet};q)=\begin{cases}
            \frac{29\lambda}{63}, &q\neq q_L,q_B\\
            \frac{5\lambda}{7}, &q=q_L,\\
            \frac{25\lambda}{42}, &q=q_B.
        \end{cases}
    \]
    Put $K_{\hat{M}}+\Delta_{\hat{M}}:=(K_{\hat{S}_2}+(1-\lambda)\hat{C}+\hat{M})|_{\hat{M}}$. Then
    \[
        A_{\hat{M},\Delta_{\hat{M}}}(q)=\begin{cases}
            1, &q\neq q_B,q_{C^2},\\
            \frac{1}{2}, &q=q_B,\\
            \lambda, &q=q_{C^2}.
        \end{cases}
    \]
    Theorem~\ref{AZ} now gives the lower bound
    \begin{equation}\label{２_CF-E_notflex_lb}
        \delta_p(S_2,(1-\lambda)C^2)
        \geq\min\left\{\frac{21+42\lambda}{55\lambda},\frac{63}{29\lambda},\frac{63}{29},\frac{7}{5\lambda},\frac{21}{25\lambda}\right\}=\begin{cases}
            \frac{21}{25\lambda}, &\frac{3}{5}\leq\lambda\leq 1,\\
            \frac{21+42\lambda}{55\lambda}, &\frac{29}{107}\leq\lambda\leq\frac{3}{5},\\
            \frac{63}{29}, &0<\lambda\leq\frac{29}{107}.
        \end{cases}
    \end{equation}
    Then, \eqref{２_CF-E_notflex_ub} and \eqref{２_CF-E_notflex_lb} complete the proof.
\end{proof}

\begin{lemma}\label{2_CB-A_flex}
    Suppose that $p$ is the intersection point $C^2$ and $B$  that is not contained in $A^1\cup A^2$, and that $\phi_2(p)$ is an inflection point of $\phi_2(C^2)$. Then $$\min\left\{\frac{21}{25\lambda},\frac{3+9\lambda}{10\lambda},3\right\}\leq\delta_p(S_2,(1-\lambda)C^2)\leq\min\left\{\frac{21}{25\lambda},\frac{3+9\lambda}{10\lambda}\right\}.$$ In particular, for $\lambda\geq\frac{1}{7}$, we have $$\delta_p(S_2,(1-\lambda)C^2)=\min\left\{\frac{21}{25\lambda},\frac{3+9\lambda}{10\lambda}\right\}=\begin{cases}
        \frac{21}{25\lambda}, &\frac{3}{5}\leq\lambda\leq1,\\
        \frac{3+9\lambda}{10\lambda}, &\frac{1}{7}\leq\lambda\leq\frac{3}{5}.
    \end{cases}$$
\end{lemma}

\begin{proof}
    Take the curve $L$ in $|A^1+A^2+B|$ that is tangent to $C^2$ at $p$. Let $\sigma_2:\hat{S}_2\rightarrow (S_2,(1-\lambda)C^2)$ be the $(1,3)$-blowup at $p$ with respect to $L$. Note that $q_B$ is an $\mathrm{A}_2$ singularity. This construction is illustrated below:

    
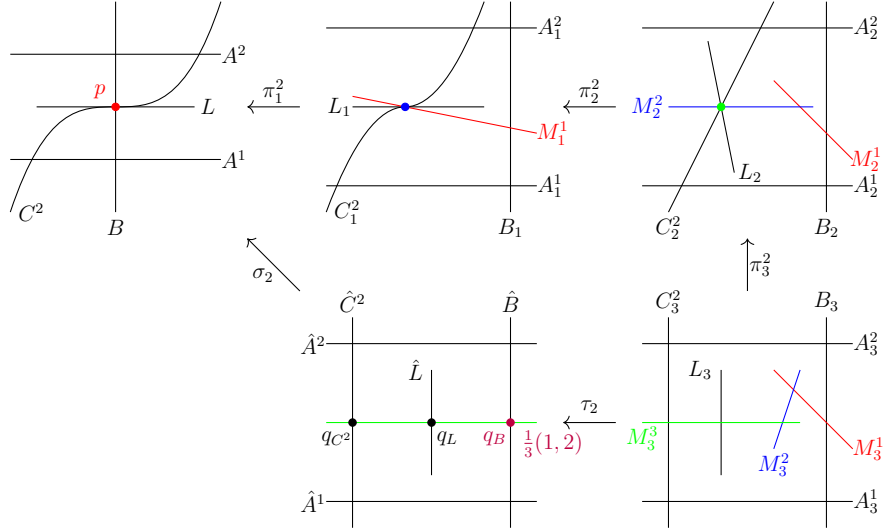
\begin{center}
    \begin{tikzpicture}[scale=0.7, every node/.style={scale=0.7}]
        \draw (-8,3) -- (-4,3);
        \draw (-3.75,3) node {$A^1$};
        \draw (-8,5) -- (-4,5);
        \draw (-3.75,5) node {$A^2$};
        \draw (-6,2) -- (-6,6);
        \draw (-6,1.7) node {$B$}; 
        \draw (-7.5,4) -- (-4.5,4);
        \draw (-4.25,4) node {$L$};
        \draw (-8,2) .. controls (-20/3,6) and (-16/3,2) .. (-4,6);
        \draw (-7.6,2) node {$C^2$};
        \filldraw[red] (-6,4) circle (2pt);
        \draw[red] (-6.3,4.3) node {$p$};

        \draw[->] (-2.5,4) --(-3.5,4);
        \draw (-3,4.3) node {$\pi^2_1$};
        \draw[->] (-2.5,.5) -- (-3.5,1.5);
        \draw (-3.2,.8) node {$\sigma_2$};

        \draw (-2,2.5) -- (2,2.5);
        \draw (2.25,2.5) node {$A^1_1$};
        \draw (-2,5.5) -- (2,5.5);
        \draw (2.25,5.5) node {$A^2_1$};
        \draw (1.5,2) -- (1.5,6);
        \draw (1.5,1.7) node {$B_1$};
        \draw (-1.5,4) -- (1,4);
        \draw (-1.75,4) node {$L_1$};
        \draw (-2,2) .. controls (-1.5,10/3) and (-1,4) .. (-.5,4);
        \draw (1,6) .. controls (.5, 14/3) and (0,4) .. (-.5,4);
        \draw (-1.6,2) node {$C^2_1$};
        \draw[red] (-1.5,4.2) -- (2,3.5);
        \draw[red] (2.3,3.5) node {$M^1_1$};
        \filldraw[blue] (-.5,4) circle (2pt);

        \draw[->] (3.5,4) -- (2.5,4);
        \draw (3,4.3) node {$\pi^2_2$};

        \draw (4,2.5) -- (8,2.5);
        \draw (8.25,2.5) node {$A^1_2$};
        \draw (4,5.5) -- (8,5.5);
        \draw (8.25,5.5) node {$A^2_2$};
        \draw (7.5,2) -- (7.5,6);
        \draw (7.5,1.7) node {$B_2$};
        \draw (4.5,2) -- (6.5,6);
        \draw (4.5,1.7) node {$C^2_2$};
        \draw (5.75,2.75) -- (5.25,5.25);
        \draw (6.05, 2.75) node {$L_2$};
        \draw[red] (6.5,4.5) -- (8,3);
        \draw[red] (8.25,3) node {$M^1_2$};
        \draw[blue] (4.5,4) -- (7.25,4);
        \draw[blue] (4.1,4) node {$M^2_2$};
        \filldraw[green] (5.5,4) circle (2pt);

        \draw[->] (6,.5) -- (6,1.5);
        \draw (6.25,1) node {$\pi^2_3$};

        \draw (4,-3.5) -- (8,-3.5);
        \draw (8.25,-3.5) node {$A^1_3$};
        \draw (4,-.5) -- (8,-.5);
        \draw (8.25,-.5) node {$A^2_3$};
        \draw (7.5,-4) -- (7.5,0);
        \draw (7.5,.3) node {$B_3$};
        \draw (4.5,-4) -- (4.5,0);
        \draw (4.5,.3) node {$C^2_3$};
        \draw (5.5,-3) -- (5.5,-1);
        \draw (5.1,-1) node {$L_3$};
        \draw[red] (6.5,-1) -- (8,-2.5);
        \draw[red] (8.3,-2.5) node {$M^1_3$};
        \draw[blue] (6.5,-2.5) -- (7,-1);
        \draw[blue] (6.5,-2.8) node {$M^2_3$};
        \draw[green] (4,-2) -- (7,-2);
        \draw[green] (4,-2.3) node {$M^3_3$};

        \draw[->] (3.5,-2) -- (2.5,-2);
        \draw (3,-1.7) node {$\tau_2$};

        \draw (-2,-3.5) -- (2,-3.5);
        \draw (-2.25,-3.5) node {$\hat{A}^1$};
        \draw (-2,-.5) -- (2,-.5);
        \draw (-2.25,-.5) node {$\hat{A}^2$};
        \draw (1.5,-4) -- (1.5,0);
        \draw (1.5,.3) node {$\hat{B}$};
        \draw (-1.5,-4) -- (-1.5,0);
        \draw (-1.5,.3) node {$\hat{C}^2$};
        \draw (0,-1) -- (0,-3);
        \draw (-.3,-1) node {$\hat{L}$};
        \draw[green] (-2,-2) -- (2,-2);
        \filldraw[purple] (1.5,-2) circle (2pt);
        \draw[purple] (2.3,-2.4) node {$\frac{1}{3}(1,2)$};
        \draw[purple] (1.2,-2.3) node {$q_B$};
        \filldraw (-1.5,-2) circle (2pt);
        \draw (-1.8,-2.3) node {$q_{C^2}$};
        \filldraw (0,-2) circle (2pt);
        \draw (.3,-2.3) node {$q_L$};
    \end{tikzpicture}
    \end{center}

    We then obtain $\hat{L}\equiv\hat{A}^1+\hat{A}^2+\hat{B}-2\hat{M}$ and
    \[
    \begin{split}
        &\sigma_2^\ast A^1=\hat{A}^1,\quad\sigma_2^\ast A^2=\hat{A}^2,\quad\sigma_2^\ast B=\hat{B}+\hat{M},\quad\sigma_2^\ast L=\hat{L}+3\hat{M},\\
        &\sigma_2^\ast K_{S_2}=K_{\hat{S}_2}-3\hat{M},\quad\sigma_2^\ast C^2=\hat{C}^2+3\hat{M}.
    \end{split}
    \]
    In particular, the log discrepancy of $\hat{M}$ with respect to $(S_2,(1-\lambda)C^2)$ is
    \[
        A_{S_2,(1-\lambda)C^2}(\hat{M})=1+3\lambda.
    \]
    
    The intersections are  given by
        \[
    \begin{split}
        &(\hat{A}^i)^2=-1,\quad\hat{B}^2=-\frac{4}{3},\quad\hat{L}^2=-2,\quad\hat{M}^2=-\frac{1}{3},\quad\hat{B}\cdot\hat{M}=\frac{1}{3},\\
        &\hat{A}^i\cdot\hat{A}^{3-i}=\hat{A}^i\cdot\hat{L}=\hat{A}^i\cdot\hat{M}=\hat{B}\cdot\hat{L}=0,\quad\hat{A}^i\cdot\hat{B}=\hat{L}\cdot\hat{M}=1,\quad\text{for }i=1,2.
    \end{split}
    \]

    Since $T_2$ is a weak del Pezzo surface, Proposition~\ref{MDS} implies that $\hat{S}_2$ is a Mori dream space, and its Mori cone is spanned by $[\hat{A}^1]$, $[\hat{A}^2]$, $[\hat{B}]$, $[\hat{L}]$, and $[\hat{M}]$. Thus, we obtain
    \[
        \tau_{S_2,(1-\lambda)C^2}(\hat{M})=7\lambda,
    \]
    since
    \[
    \begin{split}
        \sigma_2^\ast\left(-K_{S_2}-(1-\lambda)C^2\right)-t\hat{M}&\equiv2\lambda\hat{A}^1+2\lambda\hat{A}^2+3\lambda\hat{B}+(3\lambda-t)\hat{M}\\
        &\equiv\frac{7\lambda-t}{2}\hat{A}^1+\frac{7\lambda-t}{2}\hat{A}^2+\frac{9\lambda-t}{2}\hat{B}+\frac{t-3\lambda}{2}\hat{L}.
    \end{split}
    \]
    The Zariski decomposition of the divisor is given by
     \begin{align*}
        P(t)=\begin{cases}
            2\lambda\hat{A}^1+2\lambda\hat{A}^2+3\lambda\hat{B}+(3\lambda-t)\hat{M}\\
            \frac{7\lambda-t}{4}(2\hat{A}^1+2\hat{A}^2+3\hat{B})
        \end{cases};\quad
        N(t)=\begin{cases}
            0, &\ \ \ 0\leq t\leq 3\lambda,\\
            \frac{t-3\lambda}{4}\hat{B}+\frac{t-3\lambda}{2}\hat{L}, &\ \ \ 3\lambda\leq t\leq 7\lambda.
        \end{cases}
    \end{align*}
    We then compute
    \[
        \mathrm{vol}\left(\sigma_2^\ast\left(-K_{S_2}-(1-\lambda)C^2\right)-t\hat{M}\right)=P(t)^2=\begin{cases}
            7\lambda^2-\frac{1}{3}t^2, &0\leq t\leq3\lambda,\\
            \frac{(7\lambda-t)^2}{4}, &3\lambda\leq t\leq 7\lambda,\\
        \end{cases}
    \]
    and hence,
    \[
        S_{S_2,(1-\lambda)C^2}(\hat{M})=\frac{10\lambda}{3}.
    \]
    From \eqref{2_F-C_ub}, we obtain the upper bound
    \begin{equation}\label{２_CF-E_flex_ub}
        \delta_p(S_2,(1-\lambda)C^2)\leq\min\left\{\frac{21}{25\lambda},\frac{3+9\lambda}{10\lambda}\right\}.
    \end{equation}
    On the other hand, for each $q$ on $\hat{M}$, 
    \[
        h(\hat{M},q,t)=\begin{cases}
            \frac{t^2}{18}, &0\leq t\leq3\lambda,\\
            \frac{7\lambda-t}{4}\cdot\mathrm{ord}_q\left(\frac{t-3\lambda}{12}q_B+\frac{t-3\lambda}{2}q_L\right)+\frac{(7\lambda-t)^2}{32}, &3\lambda\leq t\leq7\lambda,
        \end{cases}
    \]
    and hence,
    \[
        S(W^{\hat{M}}_{\bullet,\bullet};q)=\begin{cases}
            \frac{\lambda}{3}, &q\neq q_L,q_B\\
            \frac{25\lambda}{63}, &q=q_B,\\
            \frac{5\lambda}{7}, &q=q_L.
        \end{cases}
    \]
    Put $K_{\hat{M}}+\Delta_{\hat{M}}:=(K_{\hat{S}_2}+(1-\lambda)\hat{C}+\hat{M})|_{\hat{M}}$, then we have
    \[
        A_{\hat{M},\Delta_{\hat{M}}}(q)=\begin{cases}
            1, &q\neq q_B,q_{C^2},\\
            \frac{1}{3}, &q=q_B,\\
            \lambda, &q=q_{C^2}.
        \end{cases}
    \]
    It then follows from Theorem~\ref{AZ} that
    \begin{equation}\label{２_CF-E_flex_lb}
        \delta_p(S_2,(1-\lambda)C^2)
        \geq\min\left\{\frac{3+9\lambda}{10\lambda},\frac{3}{\lambda},3,\frac{7}{5\lambda},\frac{21}{25\lambda}\right\}
        =\begin{cases}
            \frac{21}{25\lambda}, &\frac{3}{5}\leq\lambda\leq 1,\\
            \frac{3+9\lambda}{10\lambda}, &\frac{1}{7}\leq\lambda\leq\frac{3}{5},\\
            3, &0<\lambda\leq\frac{1}{7}.
        \end{cases}
    \end{equation}
    The proof is completed by combining \eqref{２_CF-E_flex_ub} and \eqref{２_CF-E_flex_lb}.    
\end{proof}

\begin{lemma}\label{2_CAB} 
    Suppose that $p$ is the intersection point of two $(-1)$-curves, and that $C^2$ passes through $p$.
    Then, $$\min\left\{\frac{21}{25\lambda},\frac{7+7\lambda}{16\lambda},\frac{7}{3}\right\}\leq\delta_p(S_2,(1-\lambda)C^2)\leq\min\left\{\frac{21}{25\lambda},\frac{7+7\lambda}{16\lambda}\right\}.$$ In particular, for $\lambda\geq\frac{3}{13}$, we have $$\delta_p(S_2,(1-\lambda)C^2)=\min\left\{\frac{21}{25\lambda},\frac{7+7\lambda}{16\lambda}\right\}=\begin{cases}
        \frac{21}{25\lambda}, &\frac{23}{25}\leq\lambda\leq 1,\\
        \frac{7+7\lambda}{16\lambda}, &\frac{3}{13}\leq\lambda\leq\frac{23}{25}.
    \end{cases}$$
\end{lemma}

\begin{proof}
    We may assume that $p$ is the intersection point of $C^2$ and $A^1$. Let $\sigma_2:\hat{S}_2\rightarrow (S_2,(1-\lambda)C^2)$ be the blowup at $p$ with the exceptional curve $\hat{M}$. This ordinary blowup realizes the desired plt blowup of $(S_2,(1-\lambda)C_2)$. This can be illustrated as follows:

    
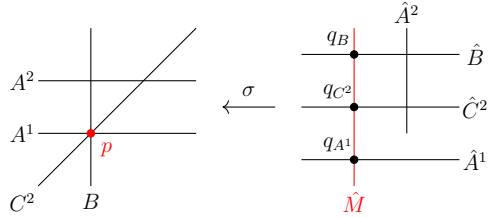
\begin{center}
    \begin{tikzpicture}[scale=0.7, every node/.style={scale=0.7}]
        \draw (-4,-.5) -- (-1,-.5);
        \draw (-4.3,-.5) node {$A^1$};
        \draw (-4,.5) -- (-1,.5);
        \draw (-4.3,.5) node {$A^2$};
        \draw (-3,-1.5) -- (-3,1.5);
        \draw (-3,-1.8) node {$B$};
        \draw (-4,-1.5) -- (-1,1.5);
        \draw (-4.3,-1.8) node {$C^2$};
        \filldraw[red] (-3,-.5) circle (2pt);
        \draw[red] (-2.7,-.8) node {$p$};

        \draw[->] (.5,0) -- (-.5,0);
        \draw (0,.3) node {$\sigma$};

        \draw (1,-1) -- (4,-1);
        \draw (4.3,-1) node {$\hat{A}^1$};
        \draw (3,-.5) -- (3,1.5);
        \draw (3,1.8) node {$\hat{A}^2$};
        \draw (1,1) -- (4,1);
        \draw (4.3,1) node {$\hat{B}$};
        \draw (1,0) -- (4,0);
        \draw (4.3,0) node {$\hat{C}^2$};
        \draw[red] (2,-1.5) -- (2,1.5);
        \draw[red] (2,-1.8) node {$\hat{M}$};
        \filldraw (2,1) circle (2pt);
        \draw (1.7,1.3) node {$q_B$};
        \filldraw (2,0) circle (2pt);
        \draw (1.7,.3) node {$q_{C^2}$};
        \filldraw (2,-1) circle (2pt);
        \draw (1.7,-.7) node {$q_{A^1}$};
    \end{tikzpicture}
    \end{center}

    The pullbacks by $\sigma_2$ are computed as
    \[
    \begin{split}
        &\sigma_2^\ast A^1=\hat{A}^1+\hat{M},\quad\sigma_2^\ast A^2=\hat{A}^2,\quad\sigma_2^\ast F=\hat{F}+\hat{M},\\
        &\sigma_2^\ast K_{S_2}=K_{\hat{S}_2}-\hat{M},\quad\sigma_2^\ast C^2=\hat{C}^2+\hat{M}.
    \end{split}
    \]
    It directly follows that
    \[
        A_{S_2,(1-\lambda)C^2}(\hat{M})=1+\lambda.
    \]

    The weak del Pezzo surface $\hat{S}_2$ is a Mori dream space, and its Mori cone is generated by 
    $[\hat{A}^1]$, $[\hat{A}^2]$, $[\hat{F}]$, and $[\hat{M}]$.
    We have
    \[
        \sigma_2^\ast\left(-K_{S_2}-(1-\lambda)C^2\right)-t\hat{M}\equiv2\lambda\hat{A}^1+2\lambda\hat{A}^2+3\lambda\hat{B}+(5\lambda-t)\hat{M},
    \]
    and it is pseudoeffective only for $t$ not exceeding $5\lambda$. Its Zariski decomposition is given by
    \begin{align*}
        P(t)&=\begin{cases}
            2\lambda\hat{A}^1+2\lambda\hat{A}^2+3\lambda\hat{B}+(5\lambda-t)\hat{M}, &0\leq t\leq\lambda,\\
            \frac{5\lambda-t}{2}\hat{A}^1+2\lambda\hat{A}^2+\frac{7\lambda-t}{2}\hat{B}+(5\lambda-t)\hat{M}, &\lambda\leq t\leq 3\lambda,\\
            \frac{5\lambda-t}{2}(\hat{A}^1+2\hat{A}^2+2\hat{B}+2\hat{M}), &3\lambda\leq t\leq 5\lambda;
        \end{cases}
        \\
        N(t)&=\begin{cases}
            0, &0\leq t\leq \lambda,\\
            \frac{t-\lambda}{2}(\hat{A}^1+\hat{B}), &\lambda\leq t\leq3\lambda,\\
            \frac{t-\lambda}{2}\hat{A}^1+(t-3\lambda)\hat{A}^2+(t-2\lambda)\hat{B}, &3\lambda\leq t\leq 5\lambda.
        \end{cases}
    \end{align*}
    Then
    \[
        \mathrm{vol}\left(\sigma_2^\ast\left(-K_{S_2}-(1-\lambda)C^2\right)-t\hat{M}\right)=P(t)^2=\begin{cases}
            7\lambda^2-t^2, &0\leq t\leq\lambda,\\
            -2\lambda t+8\lambda^2, &\lambda\leq t\leq 3\lambda,\\
            \frac{(5\lambda-t)^2}{2}, &3\lambda\leq t\leq 5\lambda,\\
        \end{cases}
    \]
    and hence,
    \[
        S_{S_2,(1-\lambda)C^2}(\hat{M})=\frac{16\lambda}{7}.
    \]
    From \eqref{2_F-C_ub}, we obtain the upper bound
    \begin{equation}\label{２_CEF_ub}
        \delta_p(S_2,(1-\lambda)C^2)\leq\min\left\{\frac{21}{25\lambda},\frac{7+7\lambda}{16\lambda}\right\}.
    \end{equation}
    For each $q$ on $\hat{M}$, 
    \[
        h(\hat{M},q,t)=\begin{cases}
            \frac{t^2}{2}, &0\leq t\leq\lambda,\\
            \lambda\cdot\mathrm{ord}_q\left(\frac{t-\lambda}{2}q_{A^1}+\frac{t-\lambda}{2}q_B\right)+\frac{\lambda^2}{2}, &\lambda\leq t\leq3\lambda,\\
            \frac{5\lambda-t}{2}\cdot\mathrm{ord}_q\left(\frac{t-\lambda}{2}q_{A^1}+(t-2\lambda)q_B)\right)+\frac{(5\lambda-t)^2}{8}, &3\lambda\leq t\leq5\lambda,
        \end{cases}
    \]
    and hence,
    \[
        S(W^{\hat{M}}_{\bullet,\bullet};q)=\begin{cases}
            \frac{3\lambda}{7}, &q\neq q_B,q_{A^1}\\
            \frac{25\lambda}{21}, &q=q_B,\\
            \frac{23\lambda}{21}, &q=q_{A^1}.
        \end{cases}
    \]
    Put $K_{\hat{M}}+\Delta_{\hat{M}}:=(K_{\hat{S}_2}+(1-\lambda)\hat{C}+\hat{M})|_{\hat{M}}$, then
    \[
        A_{\hat{M},\Delta_{\hat{M}}}(q)=\begin{cases}
            1, &q\neq q_{C^2},\\
            \lambda, &q=q_{C^2}.
        \end{cases}
    \]
    It then follows from Theorem~\ref{AZ} that
    \begin{equation}\label{２_CEF_lb}
        \delta_p(S_2,(1-\lambda)C^2)
        \geq\min\left\{\frac{7+7\lambda}{16\lambda},\frac{7}{3\lambda},\frac{7}{3},\frac{21}{25\lambda},\frac{21}{23\lambda}\right\}
        =\begin{cases}
            \frac{21}{25\lambda}, &\frac{23}{25}\leq\lambda\leq 1,\\
            \frac{7+7\lambda}{16\lambda}, &\frac{3}{13}\leq\lambda\leq\frac{23}{25},\\
            \frac{7}{3}, &0<\lambda\leq\frac{3}{13}.
        \end{cases}
    \end{equation}
    Consequently, \eqref{２_CEF_ub} and \eqref{２_CEF_lb} complete the proof.
\end{proof}

Observe that at least three irreducible members in the pencil $|N^i|$ are tangent to $C^2$ for each~$i$ and that
the plane cubic curve $\phi_2(C^2)$ has at least six inflection points outside $\phi_2(B)$.
It then follows from Lemmas~\ref{2_S-C}, \ref{2_C-AB_notflex}, \ref{2_C-AB_flex}, and \ref{2_CN_tangent} that
\begin{equation}\label{2-semitotal}
    \inf_{p\in S_2\setminus(C^2\cap(A^1\cup A^2\cup B))}\delta_p(S_2,(1-\lambda)C^2)\begin{cases}
        =\frac{21}{25\lambda},&\frac{16}{25}\leq\lambda\leq1,\\
        =\frac{7+14\lambda}{19\lambda},&\frac{23}{68}\leq\lambda\leq\frac{16}{25},\\
        \geq\frac{42}{23},&0<\lambda\leq\frac{23}{68}.
    \end{cases}
\end{equation}

Recall that the line $\phi_2(B)$ and the smooth cubic curve $\phi_2(C^2)$ pass through the points $x_1$ and $x_2$. Let $y$ be the remaining intersection point. Note that $y$ can be either $x_1$ or~$x_2$.

First, suppose that $y$ is $x_1$ or $x_2$. We may assume $y=x_1$. Since $y$ cannot be an inflection point, there are two cases: $x_2$ is an inflection point, or it is not. If $x_2$ is not a inflection point, then, combining with \eqref{2-semitotal}, we obtain from Lemmas~\ref{2_CB-A_notflex} and \ref{2_CAB}  that \[
    \delta(S_2,(1-\lambda)C^2)\begin{cases}
        =\frac{21}{25\lambda}, &\frac{23}{25}\leq\lambda\leq1,\\
        =\frac{7+7\lambda}{16\lambda}, &\frac{23}{73}\leq\lambda\leq\frac{23}{25},\\
        \geq\frac{42}{23}, &0<\lambda\leq\frac{23}{73}.
    \end{cases}
\]
If $x_2$ is an inflection point, Lemmas~\ref{2_CB-A_flex} and \ref{2_CAB} give the same $\delta(S_2,(1-\lambda)C^2)$. Consequently, these results directly imply the first statement of Theorem~\ref{thm2}.

We now suppose that $y$ is neither $x_1$ nor $x_2$. Recall that if two of the three points $x_1$, $x_2$,~$y$ are inflection points, then so is the remaining point. As a consequence, we obtain the following possibilities:
\begin{enumerate}
    \item None of $x_1$, $x_2$, and $y$ are inflection points;
    \item One of $x_1$ and $x_2$ is an inflection point, and the other two points are not;
    \item The point $y$ is an inflection point, and the others are not;
    \item All of them are inflection points.
\end{enumerate}

For each case, combining with \eqref{2-semitotal}, we obtain the global $\delta$-invariant as follows:
\begin{enumerate}
    \item By Lemmas~\ref{2_CA-B_notflex} and \ref{2_CB-A_notflex},
    \[
        \delta(S_2,(1-\lambda)C^2)\begin{cases}
        =\frac{21}{25\lambda},&\frac{17}{25}\leq\lambda\leq1,\\
        =\frac{1+\lambda}{2\lambda},&\frac{5}{9}\leq\lambda\leq\frac{17}{25},\\
        =\frac{7+14\lambda}{19\lambda},&\frac{23}{68}\leq\lambda\leq\frac{5}{9},\\
        \geq\frac{42}{23},&0<\lambda\leq\frac{23}{68};
    \end{cases}
    \]
    \item By Lemmas~\ref{2_CA-B_notflex}, \ref{2_CA-B_flex}, and \ref{2_CB-A_notflex},
    \[
        \delta(S_2,(1-\lambda)C^2)\begin{cases}
            =\frac{21}{25\lambda}, &\frac{18}{25}\leq\lambda\leq 1,\\
            =\frac{21+42\lambda}{61\lambda}, &\frac{23}{76}\leq\lambda\leq\frac{18}{25},\\
            \geq\frac{42}{23}, &0<\lambda\leq\frac{23}{76};
        \end{cases}
    \]
    \item  By Lemmas~\ref{2_CA-B_notflex} and \ref{2_CB-A_flex},
    \[
        \delta(S_2,(1-\lambda)C^2)\begin{cases}
            =\frac{21}{25\lambda}, &\frac{17}{25}\leq\lambda\leq 1,\\
            =\frac{1+\lambda}{2\lambda}, &\frac{5}{9}\leq \lambda\leq\frac{17}{25},\\
            =\frac{7+14\lambda}{19\lambda}, &\frac{13}{31}\leq\lambda\leq\frac{5}{9},\\
            =\frac{3+9\lambda}{10\lambda}, &\frac{23}{71}\leq\lambda\leq\frac{13}{31},\\
            \geq\frac{42}{23}, &0<\lambda\leq\frac{23}{71};
        \end{cases}
    \]
    \item By Lemmas~\ref{2_CA-B_flex} and \ref{2_CB-A_flex},
    \[
        \delta(S_2,(1-\lambda)C^2)\begin{cases}
            =\frac{21}{25\lambda}, &\frac{18}{25}\leq\lambda\leq 1,\\
            =\frac{21+42\lambda}{61\lambda}, &\frac{23}{76}\leq\lambda\leq\frac{18}{25},\\
            \geq\frac{42}{23}, &0<\lambda\leq\frac{23}{76}.
        \end{cases}
    \]
\end{enumerate}
The second statement of Theorem~\ref{thm2} then follows by combining these results.

\textbf{Acknowledgements.} The authors were supported by IBS-R003-D1 from the Institute for Basic Science in Korea.

\bibliographystyle{plain}
%\bibliography{ref}

\end{document}